\documentclass[a4paper,twoside,tbtags]{article}
\usepackage{stackrel} \usepackage{bm}\usepackage{amssymb}
\usepackage{amsmath}\usepackage{amsfonts}\usepackage{amssymb} \usepackage{stackrel} 
\usepackage{bm}

\usepackage{latexsym}\usepackage{epsfig}\usepackage{graphicx}\usepackage{oldgerm}
\usepackage{amsthm}
\begin{document}


 \pagestyle{myheadings}
 \markboth{Infinite-dimensional stochastic differential equations and tail $ \sigma$-fields II}
{Yosuke Kawamoto, Hirofumi Osada, and Hideki Tanemura}

\medskip 

\begin{center}\begin{Large} 
{\textsf{ Infinite-dimensional stochastic differential equations \\and tail $ \sigma$-fields II: the IFC condition}
}\end{Large}\end{center}
\begin{center}
\texttt{Yosuke Kawamoto$^{1}$, Hirofumi Osada$^{2}$, Hideki Tanemura$^{3}$
}
\end{center}

\begin{center}
{ (To appear in Journal of the Mathematical Society of Japan)}
\end{center}

\medskip

\begin{footnotesize}
\begin{description}
\item{1} \quad \ 
Fukuoka Dental College, Fukuoka 814-0193, Japan. \\
Email: {kawamoto@college.fdcnet.ac.jp} \\
\item{2} \quad \ 
Faculty of Mathematics, Kyushu University, \\ Fukuoka 819-0395, Japan.
E-mail: osada@math.kyushu-u.ac.jp \\
\item{3} \quad \ 
Department of Mathematics, Keio university,\\
Kohoku-ku, Yokohama 223-8522, Japan. \\ E-mail:  tanemura@math.keio.ac.jp 
\end{description}
\end{footnotesize}

\noindent 
\textbf{MSC 2020:  }{Primary 60K35; Secondary 60H10, 82C22, 60B20}

\noindent 
\textbf{Keywords: } {interacting Brownian motions, infinite-dimensional stochastic differential equations, random matrices}

\newcommand\DN{\newcommand} \newcommand\DR{\renewcommand}
\numberwithin{equation}{section}
\newcounter{Const} \setcounter{Const}{0}
\DN\Ct{\refstepcounter{Const}c_{\theConst}}
\DN\cref[1]{c_{\ref{#1}}}	
\newtheorem{definition}{Definition}[section]
\newtheorem{cdn}[definition]{Condition}
 \newtheorem{example}{Example}[section]
 \newtheorem{problem}[definition]{Problem}
 \newtheorem{theorem}{Theorem}[section]
 \newtheorem{lemma}[theorem]{Lemma}
 \newtheorem{corollary}[theorem]{Corollary}
 \newtheorem{proposition}[theorem]{Proposition}
\newtheorem{remark}{Remark}[section]
\newcommand{\mib}[1]{\mbox{\boldmath $#1$}}
\newcommand{\SSC}[1]{\section{#1}\setcounter{equation}{0}}
\newcommand{\newsection}[1]{\setcounter{equation}{0}\setcounter{theorem}{0} \section{#1}} 
\renewcommand{\theequation}{\thesection.\arabic{equation}}

\DN\0{ \sum_{i=1}^{\infty} \int_0^t }
\DN\1{M_{\qQ ,t} ^{\NN ,j}}
\DN\2{M_{\qQ } ^{\NN ,j}}
\DN\3{M_{\qQ ,t} ^{\NN ,\infty }}
\DN\4{M_{\qQ } ^{\NN ,\infty }}

\DN\5{\upsilon ( |X_0^i-X_0^j| )^2} 

\DN\6{ \frac{ \vartheta (X_u^i-X_u^j) }{ \upsilon (|X_u^i-X_u^j| ) }}
\DN\7{ \frac{ \vartheta (X_u^1-X_u^2) }{ \upsilon (|X_u^1-X_u^2| ) }}
\DN\8{\upsilon ( |X_u^1-X_u^2| )^2}

\DN\pPF{\pP _{\{ \Fs \}}}
\DN\Ps{P _{\mathbf{s}}}

\DN\RA{\Rightarrow}
\DN\LA{\Leftarrow}
\DN\Zum{\mathbf{Z}_u^m}
\DN\Ztm{\mathbf{Z}_t^m}
\DN\Zu{\mathbf{Z}_u}
\DN\ZumH{\hat{\mathbf{Z}}_u^m}
\DN\ZtmH{\hat{\mathbf{Z}}_t^m}
\DN\ZmH{\hat{\mathbf{Z}}^m}
\DN\BmHu{\hat{\mathbf{B}}_u^m}
\DN\BmHt{\hat{\mathbf{B}}_t^m}
\DN\BmH{\hat{\mathbf{B}}^m}

\DN\lref[1]{Lemma~\ref{#1}}\DN\tref[1]{Theorem~\ref{#1}}\DN\pref[1]{Proposition~\ref{#1}}
\DN\sref[1]{Section~\ref{#1}}
\DN\ssref[1]{Subsection~\ref{#1}}
\DN\dref[1]{Definition~\ref{#1}} 
\DN\rref[1]{Remark~\ref{#1}} 
\DN\corref[1]{Corollary~\ref{#1}}
\DN\eref[1]{Example~\ref{#1}}

\DN\bs{\bigskip}\DN\ms{\medskip}
\DN\N{\mathbb{N}}\DN\R{\mathbb{R}}\DN\Q{\mathbb{Q}}\DN\C{\mathbb{C}}
\DN\map[3]{#1\!:\!#2\!\to\!#3}
\DN\ot{\otimes} \DN\ts{\times }
\DN\PD[2]{\frac{\partial#1}{\partial#2}}
\DN\half{\frac{1}{2}}
\DN\elaw{\stackrel{\mathrm{law}}{=}}
\DN\eac{\stackrel{\mathrm{ac}}{\sim}}

\DN\Rd{\R ^d}
\DN\RtwoN{\R ^{2\N }}
\DN\Rtwo{\R ^2}
\DN\Rtwom{\R ^{2m}}

\DN\RdN{\R ^{d\N }} 

\DN\limi[1]{\lim_{#1\to\infty}} \DN\limz[1]{\lim_{#1\to0}}
\DN\limsupi[1]{\limsup_{#1\to\infty}}\DN\liminfi[1]{\liminf_{#1\to\infty}} 	
\DN\limsupz[1]{\limsup_{#1\to 0}}\DN\liminfz[1]{\liminf_{#1\to 0}} 	

\DN\As[1]{$ ($\textbf{#1}$)$}
\DN\Ass[1]{$ \{ $\textbf{#1}$\}$}

\DN\Section{\section}\DN\Ssection{\subsection}\DN\SSsection{\subsubsection}


 \DN\SO{\overline{\sS }}
 \DN\SOm{\SO ^m}
 \DN\RRr{\SOm _{\rr }} 
 \DN\RRpr{\SOm _{\pp ,\rr }}
\DN\RRprs{\RRpr (\sss )} 
\DN\RRprC{\Sm _{\pp ,\rr }}
\DN\RRprCs{\Sm _{\pp ,\rr }(\sss )}

\DN\LL{\mathcal{L}_{\qQ } }
\DN\yy{\mathfrak{y}}
 \DN\WSs{\WT (\Ss )}

\DN\WSsiNE{\WT _{\mathrm{NE}}(\Ssi )}

\DN\muN{\mu ^{\nN }}
\DN\Xidt{ \XX _t^{\idia }}
\DN\XNidt{\XX _t^{\nN , \idia }} 
\DN\idia{i \diamondsuit }
\DN\onedia{1 \diamondsuit }

\DN\sN{\sigma ^{\nN }}
\DN\aN{a^{\nN }}
\DN\SrSS{\Sr \ts \sSS }

\DN\nablax{\nabla _x}
\DN\supN{\sup_{\nN \in \N }}
\DN\muNone{\mu ^{\nN ,[1]}}
\DN\muNonebar{\bar{\mu }^{\nN ,[1]}}

\DN\bNrs{b_{r,s}^{\nN }} \DN\bN{b ^{\nN }}

\DN\tail{\mathrm{tail}}
\DN\bNrsp{{b}_{r,s,p }^{\nN }}

\DN\brsp{\bbb _{r,s,\p }}
\DN\brs{\bbb _{r,s }}
\DN\btail{\bbb ^{\tail }}
\DN\bNrstail{\bbb _{r,s}^{\nN ,\tail }}

\DN\JL[2]{\mathbf{J}_{[#1]}^{#2}}
\DN\SStwo{\sSS ^{[2]}}
\DN\Etwo{E _{\la }^{[2]}}
\DN\QQtwo{\QQ _{\la }^{[2]}}
\DN\SBhat{\{ \widetilde{\sigma }^m , \, \widetilde{ b }^m \}}
\DN\SB{(\widetilde{\sigma }^m , \, \widetilde{ b }^m )} 

\DN\simpr{\sim _{\pp ,\rr }}
\DN\HmnP{\Pi _2 (\HanC )} 
\DN\HmnPc{\Pi _2 (\Han )}
\DN\HanC{\Han ^{\circ }}

\DN\pq{\pp ,\rr }\DN\pqr{\pp ,\qq ,\rr } 
\DN\NNNthree{\NNN _3 }
\DN\NNNtwo{\NNN _2 }

\DN\HHns{\mathbb{H}_{\nn }(\sss )}

\DN\YY{\mathfrak{Y} }
\DN\sss{\mathfrak{s}}
\DN\XX{\mathfrak{X}}
\DN\xx{\mathfrak{x}}
\DN\s{\mathbf{s}} 

\DN\zN{\{ 0 \} \cup \N }

\DN\XBXm{(\mathbf{X}^m , \mathbf{B}^m , \XX ^{m*}) }

\DN\ginSDE{\eqref{:35c}--\eqref{:35e}}

\DN\Ft{\{ \mathcal{F}_t \}}
\DN\FtB{$\{ \mathcal{F}_t \}$-Brownian motion }

\DN\ASTpath{$ \mathbf{T}_{\mathrm{path}}1$}
\DN\ASTpatH{$ \mathbf{T}_{\mathrm{path}}2$}

\DN\Ehat{\mathcal{C}}
\DN\Ehatm{\Ehat^{m}}
\DN\Ehatmt{\Ehatm _t}

\DN\WSm{ \WT (\Sm )} 
\DN\lpathX{\lpath (\XX ) }

\DN\sigmam{\sigma ^m}
\DN\smn{\Sigma _{\nn }} 
\DN\smnT{\smn \wedge T }

\DN\PmNI[1]{\Pm (\XX _t\not\in #1 \text{ for some } 0\le t < \infty )=0}

\DN\vv{\mathfrak{v}}
\DN\UV{(\mathbf{u},\vv )}
\DN\INmu{\As{AC}, \As{SIN}, \As{NBJ}, and \As{IFC}} 

\DN\PPPm{\widetilde{\PP } ^m }

\DN\wti{\w _t^i}
\DN\w{\mathit{w}}
\DN\wt{\mathit{w}_t}
\DN\ww{\mathfrak{w}}
\DN\wwt{\ww _t}
\DN\www{\mathbf{w}}\DN\wwwt{\www _t}\DN\wtt{\www (t)}
\DN\wwwm{\www ^{[m]}}

\DN\la{\lambda }

\DN\mul{\mu \circ {\lab }^{-1}}
\DN\Pml{\mathbf{P}_{\mul } }

\DN\zti{ 0 \le t < \infty }\DN\zzti{ 0 < t < \infty }
\DN\bj{big jump\ }

\DN\Vnu{V_{\nu }}
\DN\Eq{Eq. \!\!}

\DN\y{\mathbf{y}}

\DN\pirc{\pi _r^c}
\DN\piRc{\pi _{\rR }^c}
\DN\pisc{\pi _s^c}
\DN\pitc{\pi _t^c}
\DN\pir{\pi _r}
\DN\piR{\pi _{\rR }}
\DN\pis{\pi _s}

\DN\nnNNN{\nn \in \NNN }

\DN\pp{{p}}\DN\qq{{q}}\DN\rr{{r}}
\DN\p{p} \DN\phat{\hat{\p }}

\DN\ppp{p}
\DN\qqq{q}

	\DN\NNN{\mathrm{N}}
	\DN\NN{N}
	\DN\nN{N}
	\DN\PP{\mathit{P}}
	\DN\pP{P}
	\DN\qQ{Q}

\DN\rR{R} 
\DN\sS{S} \DN\sSS{\mathfrak{S}}
\DN\SrR{\sS _{\rR }}

\DN\pPm{\pP _{\mu }}
\DN\pPs{\Pts }

\DN\nn{\mathrm{n}} 
\DN\rrr{r} 

\DN\lL{l}
\DN\kK{k}
\DN\kl{\kK \lL }
\DN\lD{_{ \lL =1}^d}
\DN\kD{_{ \kK =1}^d}
\DN\klD{_{\kK , \lL =1}^d }
\DN\sumklD{\sum \klD } \DN\sumkD{\sum \kD}
\DN\akl{a_{\kK \lL }}
\DN\sigmakl{\sigma _{\kl }}
\DN\il{i,\lL }

\DN\ijn{_{i,j=1}^{n}}

\DN\vq{\varphi _{\qq }}
\DN\vr{\varphi _{r }}
\DN\az{a_{0}}
\DN\ak{a_{\qq }}
\DN\akk{a_{\qq +1}}

\DN\K{\mathcal{K} }
\DN\Ka{\K [\mathbf{a}]}
	\DN\Kaa{\K [\mathbf{a}^+]}
\DN\Kak{\K [\mathit{a}_{\qq }]}
\DN\Kakk{\K [\mathit{a}_{\qq }^+]}

\DN\KQ{\K _{\qQ }}
\DN\KaQ{\K _{\qQ }[ \mathbf{a}]}
\DN\KaQk{\K _{\qQ }[\mathit{a}_{\qq }]}
\DN\KaQkk{\K _{\qQ }[\mathit{a}_{\qq }^+]}

\DN\hH{\mathfrak{H}}
\DN\Ha{\hH [\mathbf{a}]}
\DN\HaC{\hH [\mathbf{a}]^{\circ }}
\DN\Hb{\Ha _{\pqr }}
\DN\HbC{\Ha _{\pqr }^{\circ }}

\DN\Hak{\hH [\mathbf{a}]_{\qq }}
\DN\Han{\Ha _{\nn }}
\DN\Hann{\Ha _{\nn +1}}
\DN\II{\mathfrak{I}}
\DN\IIg{\Hg ^{[1]}}

\DN\YmX{(\mathbf{Y}^{m} , \mathbf{X}^{m*})}
\DN\XmX{(\mathbf{X}^{m} , \mathbf{X}^{m*})}
\DN\Xms{ \mathbf{X}^{m*} }
\DN\xm{_{\XX }^m}

\DN\XXms{\XX ^{m*}}

\DN\XM{\mathbf{X}^{[m]}}

\DN\rgn{\rho _{\mathrm{Gin}}}
\DN\kg{{K}_{\mathrm{Gin}}}
\DN\mug{\mu _{\mathrm{Gin}}}
\DN\mugx{\mu _{\mathrm{Gin},x}}
\DN\muSinb{\muone _{\mathrm{sin}, \beta }}
\DN\muAi{\mu _{\mathrm{Ai}}}

\DN\mub{\mu _{\mathrm{Be},\alpha }}
\DN\muone{\mu ^{[1]}}
\DN\mutwo{\mu ^{[2]}}
\DN\mum{\mu ^{[m]}}

\DN\rhohat{\hat{\rho}}

\DN\SST{\widetilde{\sSS }}
\DN\SSTI{\widetilde{\sSS }^{\infty }}

\DN\Sm{\sS ^m } \DN\SmSS{\Sm \ts \sSS }
\DN\Srm{\Sr ^m} \DN\SrmSS{\Srm \ts \sSS }

\DN\SSsdeg{\sSS _{\mathrm{sde}}}
\DN\SSsde{\sSS _{\mathrm{sde}}}
\DN\SSSsde{\mathbf{S}_{\mathrm{sde}}}

\DN\SSsdemt{\sSS _{\mathrm{sde}}^m(t,\XX )}
\DN\SSsdemtw{\sSS _{\mathrm{sde}}^m(t,\ww )}
\DN\SSSsdemt{\mathbf{S}_{\mathrm{sde}}^m(t,\XX )}
\DN\SSSsdemtg{\mathbf{S}_{\mathrm{sde}}^m(t,\XX )}
\DN\SSSsdemtw{\mathbf{S}_{\mathrm{sde}}^m(t,\ww )}

\DN\SSrm{\sSS _r^m }
\DN\SSRm{\SSR ^m }
\DN\SSR{\sSS _{\rR }}

\DN\SSsm{\sSS _s^m }
\DN\SSrNk{\sSS _{r,N-k}}
\DN\SSk{\sSS ^{[k]}}
\DN\SSksingle{\Ssi ^{k}}

\DN\Si{\sSS _{\mathrm{i}}}
\DN\Ssi{\sSS _{\mathrm{s,i}}}
\DN\Ss{\sSS _{\mathrm{s}}}
\DN\Ssip{\sSS _{\mathrm{s.i}}^{+f}}
\DN\SSone{\sSS ^{\mathbf{1}}}
\DN\SoneSS{\sS \ts \sSS }

\DN\Sk{\sS ^{k}}
\DN\Sp{\sS _{\qq }}
\DN\Srk{\Sr ^{k}}
\DN\Sr{\sS _{r}}
\DN\SR{\sS _{\rR }}
\DN\SRm{\sS _{\rR }^m }
\DN\SN{\sS ^{\mathbb{N}}} 
\DN\Jrs{J_{ \rR ,\sss }}

\DN\dlog{\mathfrak{d}}
\DN\dmu{\dlog ^{\mu }}
\DN\dpsi{\dlog ^{\mupsi }}
\DN\dmuone{\dlog ^{\mu ^{1}}}
\DN\dmuhat{\hat{\dlog } ^{\mu }}

\DN\dG{\dlog ^{\mug }}
\DN\dSine{\dlog ^{\mu _{\mathrm{sin},\beta }} }

\DN\sigmaXms{\sigma \xm }
\DN\bbbXms{\mathit{b} \xm }

\DN\bbb{\mathit{b}}

\DN\dom{\mathcal{D}} 
\DN\Damu{\mathcal{D}^{\amu }} 
\DN\di{\dom _{\circ }}

\DN\Lmu{L^2(\sSS ,\mu )}
\DN\Lmg{L^{2}(\sSS ,\mug )}


\DN\Eamu{\mathcal{E}^{\amu }}
\DN\Emu{\mathcal{E}^{\mu } }
\DN\E{\mathcal{E} }

\DN\amu{a ,\mu }
\DN\amuone{a ,\muone }
\DN\DDDa{\mathbb{D}^{a}}
\DN\D{\mathbf{D}}

\DN\ulab{\mathfrak{u} }
\DN\lab{\mathfrak{l} } 

\DN\labm{\lab _m } 
\DN\labi{\lab ^i}

\DN\labN{\lab ^{\nN }}

\DN\labmN{\labm ^{\nN }}

\DN\muNl{\muN \circ (\labN )^{-1}}

\DN\ulabm{\mathfrak{u} _{[m]}}
\DN\ulabone{\mathfrak{u} _{[1]}}

\DN\upath{\ulab _{\mathrm{path}}}
\DN\lpath{\lab _{\mathrm{path}}} 

\DN\lpathmstar{\lpath ^{m* }} 
\DN\lpathmstarM{\lpath ^{(m*) }} 
\DN\lpathi{\lpath ^i}

\DN\OF{(\Omega , \mathcal{F} )}
\DN\OFF{(\Omega ,\mathcal{F},\{ \mathcal{F}_t \} )}
\DN\OFP{(\Omega ,\mathcal{F}, P )}
\DN\OFPF{(\Omega ,\mathcal{F}, P ,\{ \mathcal{F}_t \} )}

\DN\OFPFts{(\Omega ,\mathcal{F}, \Pts ,\{ \mathcal{F}_t \} )}
\DN\OFPFmg{(\Omega ,\mathcal{F}, \PP _{\mug } , \{ \mathcal{F}_t \} )}
\DN\OFPFs{(\Omega ,\mathcal{F}, \PPs , \{ \mathcal{F}_t \} )}
\DN\OFPFss{(\Omega ,\mathcal{F}, \Ps , \{ \mathcal{F}_t \} )}

\DN\OFPFlambda{(\Omega ,\mathcal{F}, \PP _{\la } , \{ \mathcal{F}_t \} )}

\DN\QQ{\mathit{Q}} 

\DN\QQla{\QQ _{\la }} \DN\QQxs{\QQ _{\mathbf{x},\sss }}
\DN\QQs{\QQ _{\sss }} \DN\QQmg{\QQ _{\mug }} 

\DN\QQlaM{\QQla ^{[m]}} \DN\QQxsM{\QQxs ^{[m]}}

\DN\QQlaMR{\QQ _{\la ,\rR } ^{[m]}}
\DN\QQlaMRR{\QQ _{\la ,\rR '} ^{[m]}}

\DN\OFQFm{(\Omega ,\mathcal{F},\{ \QQ _{\ulab (\mathbf{x}) + \sss } \}, \{ \mathcal{F}_t \} )}
\DN\OFQF{(\Omega ,\mathcal{F},\{ \QQs \}, \{ \mathcal{F}_t \} )}
\DN\OFQFs{(\Omega ,\mathcal{F}, \QQs , \{ \mathcal{F}_t \} )}
\DN\OFQFlambda{(\Omega ,\mathcal{F}, \QQla , \{ \mathcal{F}_t \} )}

%
%

\DN\musin{\mu_{\mathrm{sin}, 2}}
\DN\musinone{\mu_{\mathrm{sin}, 1}}
\DN\musinfour{\mu_{\mathrm{sin}, 4}}
\DN\musinbeta{\mu_{\mathrm{sin}, \beta}}

\DN\fg{\widetilde{b}^{m}}
\DN\fgfg{\widetilde{b}^{m}}

\DN\fgsigma{\widetilde{\sigma}^{m}}
\DN\fgfgsigma{\widetilde{\sigma}^{m}}

\DN\G{G}

\DN\Ge{G ^{\epsilon }}

\DN\Hg{\hH } \DN\lH{\lab (\hH ) }
\DN\FF{\widetilde{F}_{\nn ,T}}
\DN\F{\mathbb{F}}
\DN\Fs{\F _{\s }}
\DN\Fsm{\Fs ^{m}}

\DN\WT{W } 

\DN\WS{\WT (\sSS )} 
\DN\WSN{\WT (\SN )} 
\DN\Wsde{\WT (\SSsde )}
\DN\WSsi{\WT (\Ssi )}

\DN\WSz{\WT _{\mathbf{0}}(\SN )} 
\DN\WSNs{\WT ^{\s }(\SN )}

\DN\WRNz{\WT _{\mathbf{0}} (\RdN )}
\DN\WRN{\WT (\RdN )}

	\DN\WRdzm{\WT _{\mathbf{0}} (\Rdm )}

\DN\WWdm{\WRdzm \ts \WT (\RdN )}

\DN\Rtm{\R ^{2m}}
\DN\Rtmstar{\R ^{2m*}}
\DN\RtN{\R ^{2\N }}
\DN\Rdm{\R ^{dm}}

\DN\Wsol{\mathbf{W} ^{\mathrm{sol}}}
\DN\Wsols{\mathbf{W} ^{\mathrm{sol}}_{\s } }

\DN\Bt{\mathcal{B}_t }
\DN\Btm{\mathcal{B}_t^m }
\DN\Bw{\widetilde{\mathcal{B} }}

\DN\Tail{\mathcal{T}} 
\DN\TS{\Tail \, (\sSS )}
\DN\TSsi{\Tail (\Ssi )}

\DN\TSone{\Tail ^{\{1\}}(\sSS ; }

\DN\mupsi{\mu _{\Psi_{0}}}

\DN\PPs{\PP _{\sss }}
\DN\Ptm{\widetilde{P}_m} 
\DN\PsBg{\PsB ^{\mathrm{Gin}}}

\DN\PBr{P _{\mathrm{Br}}^{\infty}} \DN\PBrm{P _{\mathrm{Br}}^{m}}

\DN\Pm{{P}_{\mu }}
\DN\Pmg{{P}_{\mug }}

\DN\Ptztm{\Ptzt ^m}

\DN\hp{h _{\pp }}
\DN\XB{(\mathbf{X},\mathbf{B})}
\DN\XBhat{(\mathbf{X},\hat{\mathbf{B}})}
\DN\sXB{(\s ,\mathbf{B},\mathbf{X})}
\DN\XzBXm{(\mathbf{B}^m,\mathbf{X}^{m*})}
\DN\BX{(\mathbf{B},\mathbf{X})}
\DN\cB{(\cdot ,\mathbf{B})}
\DN\vw{(\mathbf{v},\mathbf{w})}

\DN\M{\mathrm{M}}
\DN\mrX{\M _{r,T} (\XX )}
\DN\mrXX{\M _{r,T} (\mathbf{X})}
\DN\mrXXN{\M _{r,T} (\mathbf{X}^{\nN})}
\DN\mrT{\M _{r,T} }
\DN\MrT{\M _{r,T}}
\DN\pqk{\pp ,\qq ,\rr }\DN\pqkl{\pp ,\qq ,\rr , l }

\DN\ellell{\ell }
\DN\g{{g}}\DN\h{{h}}

\DN\TRN{T, R \in \N }

\DN\chin{\chi _{\nn }}
\DN\chinn{\chi _{\nn +1}}

\DN\chiwt{\widetilde{\chi }_{\qQ }}
\DN\chiwtN{\chiwt ^{\NN }}

\DN\chic{\check{\chi}_{\qQ }}

	\DN\chiwtp{\PD{\chiwtN }{x_{\kK }}}
	\DN\chiwtq{\PD{\chiwtN }{x_{\lL }}}
	\DN\chiwtpq{\frac{\partial ^2 \chiwtN }{\partial x_{\kK } \partial x_{\lL }}} 

\DN\chiwtIp{\PD{\chiwtI }{x_{\kK }} }
\DN\chiwtIq{\PD{\chiwtI }{x_{\lL }} } 
\DN\chiwtIpq{\frac{\partial ^2 \chiwtI }{\partial x_{\kK } \partial x_{\lL }}}

%

\DN\chiwtI{\widetilde{\chi }}
\DN\chiwtII{\widetilde{\chi }_{\infty }}

\DN\TRe{\sS _R^{2,\varepsilon }}

\DN\chik{\chi _{\qq }}
\DN\chiktwo{\chi _{\qq } ^2}

\DN\dk{\mathbf{d}_{\qq }} 
\DN\dkQ{\mathbf{d}_{\qq , \qQ }}
\DN\chiqQ{\chi _{\qq ,\qQ }}
\DN\chirQ{\chi _{\rr ,\qQ }}

\DN\kappaq{\kappa _{\qq }}
\DN\kappai{\kappa _{\infty }}

\DN\tauz{\tau ^{0}}
\DN\tauRe{\tau _{R }^{\epsilon }}
\DN\tauRz{\tau _{R }^{0}}

\DN\tauR{\tau _{R }} 
\DN\tauRet{t\wedge \tauRe }
\DN\tauRetdot{\cdot \wedge \tauRe }
\DN\tauReT{T\wedge \tauRe }

\begin{abstract} 
In a previous report, the second and third authors gave general theorems for unique strong solutions of infinite-dimensional stochastic differential equations (ISDEs) describing the dynamics of infinitely many interacting Brownian particles. One of the critical assumptions is the \lq\lq IFC'' condition. The IFC condition requires that, for a given weak solution, the scheme consisting of the finite-dimensional stochastic differential equations (SDEs) related to the ISDEs exists. Furthermore, the IFC condition implies that each finite-dimensional SDE has unique strong solutions. Unlike other assumptions, the IFC condition is challenging to verify, and so the previous report only verified it for solutions given by quasi-regular Dirichlet forms. In the present paper, we provide a sufficient condition for the IFC requirement in more general situations. In particular, we prove the IFC condition without assuming the quasi-regularity or symmetry of the associated Dirichlet forms. As an application of the theoretical formulation, the results derived in this paper are used to prove the uniqueness of Dirichlet forms and the dynamical universality of random matrices.
\end{abstract}


\section{Introduction}\label{s:1} 
We consider the dynamics of infinitely many interacting Brownian particles in the Euclidean space $ \Rd $. 
We assume that each particle $ X^i $ moves under the effect of itself and the other infinitely many particles. 
The dynamics $ \mathbf{X} = (X_t^i)_{i\in \N }$ can be described by the following infinite-dimensional stochastic differential equation (ISDE):
\begin{align}\label{:11y}&
X_t^i -X_0^i= \int_0^t \sigma (X_u^i, \XX _u^{\idia }) dB_u^i + 
 \int_0^t \bbb (X_u^i, \XX _u^{\idia }) du 
\quad (i \in \N )
.\end{align}
Here, $ \mathbf{B}=(B^i)_{i=1}^{\infty}$, where $ \{B^i\}_{i\in\N }$ denotes 
independent copies of $ d $-dimensional Brownian motion, and 
$ \XX ^{\idia } = \{ \XX _t^{\idia } \} $ represents the unlabeled dynamics given by 
\begin{align}\label{:11z}&
\XX _t^{\idia } = \sum_{j\not=i}^{\infty} \delta _{X_t^j}
.\end{align}
The coefficients $ \sigma $ and $ \bbb $ 
are defined on $ \Rd \ts \sSS $, where $ \sSS $ denotes the configuration space over $ \Rd $ (see \eqref{:21a}). 
Note that the functions $ \sigma $ and $ \bbb $ are independent of $ i \in \N $ and 
all particles $ \{ X_t^j \}_{j\in \N, j\not=i} $ are indistinguishable in \eqref{:11z}. 
These conditions enable \eqref{:11y} to describe the motion of identical interacting particles.

A pair of $ (\Rd )^{\N }$-valued, continuous processes $ (\mathbf{X},\mathbf{B})$ 
defined on a filtered probability space 
$ (\Omega ,\mathcal{F},P,\{ \mathcal{F}_t \} )$ satisfying \eqref{:11y} is called a weak solution, where 
$ \mathbf{B}$ is an $ \{ \mathcal{F}_t \}$-Brownian motion. 
Loosely speaking, if $ \mathbf{X}$ is a functional of $ \mathbf{B}$ and an initial starting point $ \s $, 
then the weak solution $ (\mathbf{X},\mathbf{B})$ is called a strong solution. 
We say the pathwise uniqueness of solutions holds if any pair of weak solutions $ \XB $ and $ (\mathbf{X}',\mathbf{B})$ 
with the same Brownian motion $ \mathbf{B}$ defined on the common filtered probability space 
$ (\Omega ,\mathcal{F},P,\{ \mathcal{F}_t \} )$ with $ \mathbf{X}_0=\mathbf{X}_0'$ almost surely (a.s.) satisfies 
$ P (\mathbf{X}=\mathbf{X}') = 1 $. 
We say that uniqueness in law holds if the distributions of $ \mathbf{X}$ and $ \mathbf{X}'$ coincide for any pair of weak solutions $ \XB $ and $ (\mathbf{X}',\mathbf{B}')$ with the same initial distribution. 
The pathwise uniqueness of weak solutions implies the uniqueness in law because of the Yamada-Watanabe theory \cite{IW,o-t.tail}.

Typical examples of ISDEs are interacting Brownian motions. 
Each particle moves under the force of its self-potential $ \Phi (x)$ and the interaction potential $ \Psi (x,y) $. Then 
\begin{align}\label{:11a}&
X_t^i -X_0^i= B_t^i 
- \frac{\beta }{2} \int_0^t \nabla_x \Phi (X_u^i) du - 
 \frac{\beta }{2} 
\int_0^t 
\sum_{j\not=i}^{\infty} \nabla_x \Psi (X_u^i,X_u^j) du 
\quad (i \in \N )
.\end{align}
Here, 
$ \nabla_x = (\PD {}{x_i} )_{i=1}^d$ and $ \beta $ is a positive constant called the inverse temperature. 

Lang derived general solutions to ISDE \eqref{:11a} by constructing a reversible solution starting from almost all points under the condition $ \Phi = 0$ and $ \Psi \in C_0^3(\Rd )$ \cite{lang.1,lang.2}. 
Here, we say that a solution $ \mathbf{X}=(X^i)_{i\in\N }$ is reversible with respect to a random point field $ \mu $ ($ \mu $-reversible) if the associated unlabeled process $ \XX =\sum_{i\in\N } \delta_ {X^i}$ is $ \mu $-symmetric and $ \mu $ is an invariant probability of $ \XX $. Recall that a random point field $ \mu $ is a probability measure on the configuration space $ \sSS $ by definition. 

Fritz explicitly described the set of starting points for up to four dimensions \cite{Fr}, and the third author of the present paper solved the equation for hardcore Brownian balls \cite{T2}. These results used It$ \hat{\mathrm{o}}$'s method, and required the coefficients to be smooth and have compact support. These conditions exclude physically interesting examples of long-range interaction potentials, such as the Lennard-Jones 6-12 potential and Riesz potentials. 
In particular, the logarithmic potential that appears in random matrix theory is also excluded.

Typical examples of ISDEs with logarithmic interaction potentials are 
the Dyson model in $ \R $ and the Ginibre interacting Brownian motion in $ \R ^2 $: 
\begin{align}& \label{:12a} 
 X_t^i - X_0^i = B_t^i + 
 \frac{\beta }{2} \int_0^t
 \lim_{r\to\infty }\sum_{|X_u^i - X_u^j |< r, \ j\not= i}^{\infty} 
 \frac{1}{X_u^i - X_u^j } du \quad (i\in\mathbb{Z})
\end{align}
and 
\begin{align}\label{:11h}&
 X_t^i - X_0^i = B_t^i 
- \int_0^tX_u^i du + 
\int_0^t
\limi{r} \sum_{|X_u^j|<r,\ j\not=i } 
\frac{X_u^i-X_u^j}{|X_u^i-X_u^j|^2} du \quad (i\in\mathbb{N})
.\end{align}
The Gaussian unitary ensemble relates to the former with $ \beta = 2 $, and the Ginibre ensemble corresponds to the latter. 
These two examples indicate that strong long-range interactions cause significant difficulties in solving ISDEs. 
Indeed, in the conventional approach based on It$ \hat{\mathrm{o}}$'s scheme, the coefficients of the stochastic differential equations (SDEs) must have local Lipschitz continuity. In the case of ISDEs, the coefficients are not defined on the whole space and are never Lipschitz continuous, even locally. 
Nevertheless, Tsai \cite{tsai.14} proved the pathwise uniqueness and existence of strong solutions of \eqref{:12a} with a general $ \beta$. 
For this, he used a specific ISDE structure that was valid only for that model.

The second author of the present paper solved ISDEs using Dirichlet form techniques \cite{o.tp,o.isde,o.rm,o.rm2}. 
The results were applied to an extensive range of interaction potentials, including all of Ruelle's class potentials and the logarithmic potential. 
However, the solution was only a weak solution.

In \cite{o-t.tail}, the second and third authors established a general theory for the existence and pathwise uniqueness of strong solutions $ \mathbf{X}=(X^i)_{i\in\N }$ of ISDEs. 
This result proved the existence of a strong solution and the pathwise uniqueness of solutions under almost the same generality as \cite{o.isde}. 
However, it was assumed that the solution was associated with a quasi-regular Dirichlet form. 
In \cite{o.udf}, we constructed a weak solution associated with a Dirichlet form, which may {\em not necessarily} be quasi-regular. Thus, the uniqueness determined in the previous papers \cite{o.isde,o-t.tail} must be considered unsatisfactory. 

One of the critical assumptions of the general theory in \cite{o-t.tail} is the \lq\lq IFC'' condition. 
This requirement is a weak point of \cite{o-t.tail} (see \sref{s:25}). 
In \cite{o-t.tail}, the IFC condition is verified if the solution of an ISDE is associated with a quasi-regular Dirichlet form. 
The purpose of this paper is to present a sufficient condition under which the IFC condition holds. 
In particular, we shall prove the IFC condition without assuming the quasi-regularity of the associated Dirichlet form (\tref{l:34}) or the symmetry of the dynamics (\tref{l:58}).

We now explain the IFC condition. For a given weak solution $ \XB $ of \eqref{:11y} 
we introduce an \textbf{I}nfinite system of \textbf{F}inite-dimensional SDEs with \textbf{C}onsistency (IFC). 
That is, we consider the family of finite-dimensional SDEs of 
$ \mathbf{Y}^m =(Y^{m,i})_{i=1}^m $, $ m \in \N $, given by 
\begin{align}\label{:13a}
Y_t^{m,i} -Y_0^{m,i} = 
\int_0^t \sigma (Y _u^{m,i}, \YY _u^{m,\idia } + \XX _u^{m*} ) d B_u^i + 
\int_0^t \bbb (Y _u^{m,i}, \YY _u^{m,\idia } + \XX _u^{m*} ) d u 
\end{align}
with the initial condition 
\begin{align} & \label{:13b} 
\mathbf{Y}_0^m 
= \s ^m 
.\end{align}
Here, for each $ m \in \N $, we set $ \s ^m = (s_i)_{i=1}^{m} $ for $ \s =(s_i)_{i=1}^{\infty}$, 
and let $ \mathbf{B}^m = (B^i)_{i=1}^{m} $ denote the $ (\Rd )^m$-valued Brownian motions which is the first $ m $-components of the original infinite-dimensional Brownian motion $ \mathbf{B}=(B^i)_{i\in\N }$. 
Furthermore, 
\begin{align}&\notag 
 \YY _u^{m,\idia } = \sum_{j\not=i}^m \delta_{Y_u^{m,j}} \quad \text{and }\quad 
 \XX _u^{m*} = \sum_{j=m+1}^{\infty} \delta_{X_u^j}
.\end{align}

We set $ \mathbf{X}^{m*}=(X^{i})_{i=m+1}^{\infty} $. Then $ \XX _u^{m*}$ is a function of $ \mathbf{X}^{m*}$. Hence, $ \mathbf{X}^{m*}$ is a component of the coefficients in SDE \eqref{:13a}. 
We regard $ \mathbf{X}^{m*}$ as a random environment and call \eqref{:13a} an SDE of random environment type. 
By construction, SDE \eqref{:13a} becomes time-inhomogeneous although the original ISDE \eqref{:11y} is time-homogeneous. 

We call $ (\mathbf{Y}^m,\mathbf{B}^m,\mathbf{X}^{m*})$ a weak solution of \eqref{:13a} starting at $ \mathbf{s}^m $ if it satisfies \eqref{:13a} and \eqref{:13b}. 
Note that $ \mathbf{Y}^m $ is defined on the same filtered space $ (\Omega ,\mathcal{F},P,\{ \mathcal{F}_t \} )$ 
as $ \XB $. 
We call the weak solution 
$ (\mathbf{Y}^m,\mathbf{B}^m,\mathbf{X}^{m*})$ a strong solution if 
$ \mathbf{Y}^m $ is a function of the initial starting point $ \mathbf{s}^m$ and 
$ (\mathbf{B}^m,\mathbf{X}^{m*})$. 
We postulate that $ \mathbf{Y}^m $ is a function of not only $ \mathbf{B}^m$, but also 
$ \mathbf{X}^{m*}$. 
Thus, the definition of a strong solution is different from the conventional form. 
We say that the pathwise uniqueness of solutions of \eqref{:13a} holds if any two weak solutions 
$ (\mathbf{Y}^m,\mathbf{B}^m,\mathbf{X}^{m*})$ and 
$ (\hat{\mathbf{Y}}^m,\mathbf{B}^m,\mathbf{X}^{m*})$ with $\mathbf{Y}_0^m=\hat{\mathbf{Y}}_0^m $ 
 satisfy 
$ P (\mathbf{Y}^m=\hat{\mathbf{Y}}^m) = 1 $; see \sref{s:24} for details.

We say that ISDE \eqref{:11y} satisfies the IFC condition for a weak solution $ \XB $ starting at $ \mathbf{s}$ 
if the SDE \eqref{:13a} has a pathwise unique strong solution for each $ m \in \N $.

Following \cite{o-t.tail}, we explain how the IFC condition implies the pathwise uniqueness of solutions and the existence of strong solutions. 
Let $ \mathbf{X}^m =(X^1,\ldots,X^m) $ be the first $ m $-components of $ \mathbf{X}=(X^i)_{i\in\N }$.
Obviously, $ (\mathbf{X}^m,\mathbf{B}^m,\mathbf{X}^{m*})$ is a weak solution of \eqref{:13a} for $ \XB $. 
Hence, the IFC condition implies {\em consistency} of $ \mathbf{Y}^m$ 
 in the sense that 
\begin{align}
\label{:13d}&
\mathbf{Y}^m = \mathbf{X}^m 
\quad \text{ $ P $-a.s.\, for each $ m \in \N $}
.\end{align}
The identity \eqref{:13d} plays a crucial role in the general theory in \cite{o-t.tail}. 
Indeed, taking the limit as $ m \to \infty $ in \eqref{:13d}, we obtain 
\begin{align}
\label{:13f}&
\mathbf{X} = \limi{m} \mathbf{Y}^m \quad \text{ $ P $-a.s}
.\end{align}
Let $ \mathcal{T}_{\mathrm{path}}$ be the tail $ \sigma $-field of the labeled process: 
\begin{align}
\label{:13g}&
 \mathcal{T}_{\mathrm{path}} = \bigcap_{m=1}^{\infty}\sigma [\mathbf{X}^{m*}] 
.\end{align}
Note that $ \mathbf{Y}^m $ is a function of $ (\mathbf{s},\mathbf{B},\mathbf{X}^{m*})$. 
Let $ (\mathbf{s},\mathbf{B})$ be fixed. 
Then $ \mathbf{Y}^m $ becomes a function of $\mathbf{X}^{m*} $. 
Because $ \mathbf{Y}^m $ is $ \sigma [\mathbf{X}^{m*}] $-measurable for each $ m \in \N $, 
we have that $ \mathbf{X}$ is $ \mathcal{T}_{\mathrm{path}} $-measurable from 
\eqref{:13f} and \eqref{:13g}. 
Hence, $ \mathbf{X}$ is a function of $ (\mathbf{s},\mathbf{B})$ 
if $ \mathcal{T}_{\mathrm{path}} $ is $ P (\, \cdot \, | ( \mathbf{s},\mathbf{B}) )$-trivial, where 
$ P (\, \cdot \, | ( \mathbf{s},\mathbf{B}) )$ denotes the regular conditional probability of $ P $ 
conditioned at $ (\mathbf{s},\mathbf{B})$. 
Therefore, $ \mathbf{X}$ is a strong solution. 
Similarly, we can prove the pathwise uniqueness in terms of the tail $ \sigma $-field $ \mathcal{T}_{\mathrm{path}} $. 
Thus the problem reduces to the study of the tail $ \sigma $-field $ \mathcal{T}_{\mathrm{path}} $ of the labeled path space under $ P (\, \cdot \, | ( \mathbf{s},\mathbf{B}) )$. 

Let $ \XX = \{ \XX _t \} $ be the unlabeled dynamics of $ \mathbf{X}=(X^i)$ 
such that $ \XX _t = \sum_{i\in\N } \delta_{X_t^i}$. 
By definition $ \XX $ is an $ \sSS $-valued process, where $ \sSS $ is the configuration space over $ \Rd $. 
We assume that $ \XX $ has an equilibrium state $ \mu $. 
Let $ \mathcal{T} $ be the tail $ \sigma $-field of $ \sSS $. 
Then, the triviality of $ \mathcal{T}_{\mathrm{path}} $ under $ P $ can be deduced from 
the triviality of $ \mathcal{T} $ under $ \mu $ \cite[Theorem 5.1]{o-t.tail}. 

All the determinantal random point fields are tail trivial \cite{o-o.tail,ly.18,bqs}. 
Furthermore, even if $ \mu $ is not tail trivial, we can decompose $ \mu $ into tail trivial components for a wide range of random point fields $ \mu $ called quasi-Gibbs measures \cite[Theorem 3.2]{o-t.tail}. From these, we can construct pathwise unique, strong solutions of various ISDEs arising from random matrix theory (see \sref{s:7}).

The IFC condition asserts that a weak solution $ \mathbf{X} $ remains in a well-behaved subset where the coefficients of the finite-dimensional SDEs \eqref{:13a} have sufficient regularity such that \eqref{:13a} has a unique strong solution. Thus, the problem is to prevent $ \mathbf{X}$ from reaching undesirable domains. 
In \cite{o-t.tail}, the second and third authors of the present paper proved the IFC condition for the case where a weak solution $ \XB $ is associated with a quasi-regular Dirichlet form. 
Quasi-regularity of the associated Dirichlet form allows us to use the notion of capacity, 
which is a critical tool in proving such a condition. 
In the present paper, we shall prove the IFC condition without utilizing the concept of capacity.

Once we have established the IFC condition under general requirements in the main theorem \tref{l:31}, we have the uniqueness of weak solutions of ISDEs under the same circumstances. 
This uniqueness yields various striking applications. 
The first application is the uniqueness of Dirichlet forms \cite{o.udf}. 
For a given random point field $ \mu $, there are two natural Dirichlet forms, which are called 
the upper Dirichlet form $ (\E ,\dom )$ and 
the lower Dirichlet form $ (\underline{\E}, \underline{\dom })$ in \cite{o.udf}. 
Each of these satisfies the relation 
\begin{align}\label{:14a}&
 (\underline{\E}, \underline{\dom }) \le (\E ,\dom )
\end{align}
in the sense that $ \underline{\dom } \supset \dom $ and $\underline{\E} (f,f) \le \E (f,f) $ for all $ f \in \dom $. 
Using \tref{l:31} and \cite{o.udf,o-t.tail}, we deduce that the equality 
$ (\underline{\E}, \underline{\dom }) = (\E ,\dom )$ holds in \eqref{:14a}. 

The identity $ (\underline{\E}, \underline{\dom }) = (\E ,\dom )$ has a further application. 
The universality of random matrices is a subject that has been extensively studied over the past two decades. 
The universality implies that under very mild constraints, $ N $-particle systems converge to the equilibrium states appearing from random matrix theory. 
In \cite{k-o.du}, the first and second authors of the present paper derived a dynamical counterpart to this result. 
That is, they proved the weak convergence of the stochastic dynamics naturally associated with $ N $-particle systems to those of limit random point fields. 
For this, they used the identity $ (\underline{\E}, \underline{\dom }) = (\E ,\dom )$, 
which follows from the results of the present paper. 

In \cite{o.dfa}, upper Dirichlet forms are quasi-regular under a quite mild assumption such that 
correlation functions are locally bounded. 
We then present a candidate of a lower Dirichlet form which is not quasi-regular. 
We consider infinite particle systems on $ \R $. 

Let $ m_i $ be positive continuous functions such that $ \limi{|x|} m_i (x) = \infty$ $ (i=1,2) $. 
Let $ \varepsilon (f,g) = \half \int f'g' m_1dx $ be the Dirichlet form on $ L^2 (\R , m_2dx)$ such that 
$  \varepsilon $ is the increasing limit of $ \varepsilon _R(f,g) = \half \int_{[-R , R]} f'g' m_1dx $ on $ L^2 (\R , m_2dx)$ 
with the domain given by the closure of $ C_0^{\infty} (\R )$. 
Assume that the associated diffusion {\em explodes}. 

Let $  \mathcal{P}_i $ be the Poisson point process with intensity $  m_i dx $. 
Consider the bilinear form $ \mathcal{E} = \int \mathbb{D} [f,g ] d\mathcal{P}_1  $ on $ L^2 (\sSS , \mathcal{P}_2 )$. 
Here $ \mathbb{D}$ is given by \eqref{:52a} by taking $  d= 1$ and $ a = 1 $.  
Let $(\underline{\E}, \underline{\dom }) $ and $ (\overline{\E}, \overline{\dom})$ 
be the associated lower and upper Dirichlet form, respectively. 
By the result in \cite{o.dfa} combined with the time change,  
we see $ (\overline{\E}, \overline{\dom})$ is quasi-regular. 
On the other hand, we conjecture that $(\underline{\E}, \underline{\dom }) $ is not quasi-regular.

We emphasize that we shall prove the IFC condition without using the Dirichlet form theory. 
A typical advantage of doing this is that we can apply the result to ISDEs with skew-symmetric interactions. 
See \eref{e:78}. The ISDE \eqref{:78z} is naturally associated with a non-symmetric bilinear form. 
It is, however, not clear that the form satisfies the weak sector condition. 
Lack of the weak sector condition prevents us from using even for non-symmetric Dirichlet form theory \cite{mr}.

The remainder of the paper is organized as follows. 
In \sref{s:2}, we prepare a set of notions to set up the problem. 
 \sref{s:3} states the main theorems (Theorems \ref{l:31}, \ref{l:32}, and \ref{l:33}), and 
 \sref{s:4} presents proofs of Theorems \ref{l:31}, \ref{l:32}, and \ref{l:33}. 
These theorems give a sufficient condition for the IFC condition. 
In \sref{s:5}, we prove other main theorems (Theorems \ref{l:58} and \ref{l:5X}), which present 
the non-collision and non-exit conditions from the tame set for non-symmetric dynamics. 
In \sref{s:6}, we prove another main theorem (\tref{l:34}) 
using Lyons-Zheng type martingale decomposition for solutions of ISDEs. 
\tref{l:34} proves the result in \sref{s:5} for the symmetric case. 
In \sref{s:7}, we present various examples of ISDEs such as sine, Bessel, and Ginibre interacting Brownian motions, as well as interacting Brownian motions with Ruelle's class potentials. 
Furthermore, we present ISDEs with skew-symmetric interactions. 
\sref{s:8} (Appendix I) quotes some general results on ISDEs from previous studies, and 
\sref{s:9} (Appendix II) provides a proof for a Lyons-Zheng type martingale decomposition for solutions of ISDEs.

\section{Preliminaries}\label{s:2} 

\subsection{Configuration spaces and Campbell measures}\label{s:21}
Let $ \sS $ be a closed set in $ \Rd $ such that the interior $ \sS _{\mathrm{int}}$
 is a connected open set satisfying $ \overline{\sS _{\mathrm{int}}} = \sS $ and the boundary $ \partial \sS $ has a Lebesgue measure of zero. 
Let $ \sSS $ be the configuration space over $ \sS $, that is, 
\begin{align} & \label{:21a}
\sSS = \{ \sss = \sum_{i} \delta _{s_i}\, ;\, 
 \text{ $\sss ( K ) < \infty $ for all compact sets $ K \subset \sS $} \} 
.\end{align}
We equip $ \sSS $ with the vague topology, under which $ \sSS $ is a Polish space. 
A probability measure on $ (\sSS , \mathcal{B}(\sSS ) )$ is called a random point field (a point process) 
on $ \sS $. 
Here, in the present paper, we denote by $ \mathcal{B}( \cdot )$ the Borel $ \sigma $-field of $ \cdot $ 
for a topological space $ \cdot $. 

Let $ \mu $ be a random point field on $ \sS $. 
A symmetric and locally integrable function 
$ \map{\rho ^n }{\sS ^n}{[0,\infty ) } $ is called the $ n $-point correlation function of $ \mu $ 
with respect to the Lebesgue measure if $ \rho ^n $ satisfies 
\begin{align} & \notag 
\int_{A_1^{k_1}\ts \cdots \ts A_m^{k_m}} 
\rho ^n (x_1,\ldots,x_n) dx_1\cdots dx_n 
 = \int _{\sSS } \prod _{i = 1}^{m} 
\frac{\sss (A_i) ! }
{(\sss (A_i) - k_i )!} \mu (d\sss )
 \end{align}
for any sequence of disjoint bounded measurable sets 
$ A_1,\ldots,A_m \in \mathcal{B}(\sS ) $ and a sequence of natural numbers 
$ k_1,\ldots,k_m $ satisfying $ k_1+\cdots + k_m = n $. 
When $ \sss (A_i) - k_i < 0$, according to our interpretation, 
${\sss (A_i) ! }/{(\sss (A_i) - k_i )!} = 0$ by convention.

Let $ \tilde{\mu }^{[1]}$ be the measure on $(S\times \sSS , \mathcal{B}(S)\ts \mathcal{B}(\sSS ))$ determined by
\begin{align}&\notag 
\tilde{\mu }^{[1]}(A\times B )=
\int_{ B } \sss (A)\mu(d\sss ), 
\quad A\in\mathcal{B}(S), \; B \in \mathcal{B}(\sSS ) 
.\end{align}
The measure $\tilde{\mu }^{[1]}$ is called the one-Campbell measure of $\mu$. 
If $\mu$ has a one-point correlation function $\rho^1$, there exists a regular conditional probability $\tilde{\mu}_x$ of $ \mu $ satisfying
\begin{align}\notag 
&\int_{A}\tilde{\mu}_x ( B )\rho^1(x)dx = 
\tilde{\mu }^{[1]} (A\times B ), 
\quad A\in\mathcal{B}(S), \; B \in \mathcal{B}(\sSS )
.\end{align}
The measure $\tilde{\mu}_x$ is called the Palm measure of $\mu$ \cite{kal}. 

In this paper, we use the probability measure $\mu_x(\cdot)\equiv \tilde{\mu}_x(\cdot -\delta_x)$, which is called the reduced Palm measure of $\mu$. Informally, $ \mu_x $ is given by 
$  \mu _{x} = \mu (\cdot - \delta_x | \, \sss (\{ x \} ) \ge 1 ) $. 

We consider the Radon measure $ \muone $ on $ \sS \times \sSS $ such that 
$  \muone (dx d\sss ) = \rho ^1 (x) \mu _{x} (d\sss ) dx $. 
We always use $ \muone $ instead of $\tilde{\mu }^{[1]}$. 
Hence, we call $ \muone $ the one-Campbell measure of $ \mu $. 
Similarly, we define $ \mum $ by 
\begin{align}&\notag 
\mum (A\ts B )= \int _{A \ts B } \rho ^{m} (\mathbf{x}) 
\mu _{\mathbf{x}} (d\sss )
d\mathbf{x}
.\end{align}
Here, $ \mu _{\mathbf{x}}$ is the reduced Palm measure of $ \mu $ conditioned at $ \mathbf{x} \in \Sm $. 
We call $ \mum $ the $ m $-Campbell measure of $ \mu $. 
We set $ \mu ^{[0]} = \mu $ and call it the zero-Campbell measure. 
Note that $ \mum $ is not necessarily a probability measure for $ m \ge 1 $ and, in particular, is always an infinite measure if $ \mu $ is translation invariant and does not concentrate at the empty configuration, whereas $\mu ^{[0]} = \mu $ is always a probability measure by definition. 

For a subset $ A \subset \sS $, we set $ \map{\pi_{A}}{\sSS }{\sSS } $ by $ \pi_{A} (\sss ) = \sss (\cdot \cap A )$. 
A function $ f $ on $ \sSS $ is said to be local if $ f $ is $ \sigma [\pi_{K}]$-measurable 
for some compact set $ K $ in $ \sS $. 
For such a local function $ f $ on $ \sSS $ and a relatively compact open set $ O $ in $ \sS $ 
such that $ K \subset O $, we set a function $ \check{f}=\check{f}_{O }$ defined on 
$ \sum_{k=0}^{\infty} O^k $ such that $ \check{f}_{O } (x_1,\ldots x_k)$ restricted to $ O^k $ 
is symmetric in $ x_j $ ($ j=1,\ldots,k$) for each $ k $ and that for $ \xx = \sum _i \delta _{x_i}$ 
\begin{align*}&
f (\xx ) = \check{f}_{O }(x_1,\ldots,x_k) 
.\end{align*}
Here, the case $ k=0$, that is, $ \sS ^0 $ corresponds to a constant function. 
Note that for any relatively compact open sets $ O $ and $ O' $ including $ K $ 
\begin{align}\label{:21h}&
\check{f}_{O }(x_1,\ldots,x_k) = 
\check{f}_{O' } (x_1,\ldots,x_k) 
\quad \text{ for all }(x_1,\ldots,x_k) \in (O \cap O')^k 
.\end{align}
Hence, $ \check{f} $ is well defined. 
We say a local function $ f $ is smooth if $ \check{f} $ is smooth in $ (x_1,\ldots,x_k)$ for each $ k $. 

\subsection{Labeled and unlabeled path spaces}\label{s:22}
For a subset $ A $ of a topological space, the set consisting of the $ A $-valued continuous paths on 
$ [0,\infty) $ is denoted by $ \WT (A)=C([0,\infty);A) $. 
We equip $ \WSN $ with the Fr\'{e}chet metric $ \mathrm{dist}(\cdot,*) $ given by 
\begin{align}& \notag 
\mathrm{dist}(\mathbf{w},\mathbf{w}') = 
\sum_{T=1}^{\infty} \frac{1}{2^{T}} \Big\{
\sum_{n=1}^{\infty} \frac{1}{2^{n}} \min\{ 1, \|w_n-w_n'\|_{C([0,T];\sS )} \} \Big\}
\end{align}
 for $ \mathbf{w}=(w_n)_{n\in\mathbb{N}} $ and 
$ \mathbf{w}'=(w_n')_{n\in\mathbb{N}} $, 
where we set $ \| w \|_{C([0,T];\sS )} = \sup_{t\in[0,T]} |w (t)|$.

Let $ \SST = \{ \sss = \sum_i \delta_{s_i}\} $ be the set of all measures on $ \sS $ consisting of countable point measures. 
By definition, $ \sSS \subset \SST $. 
Let $ \mathbb{\sS } = \{ \bigcup_{\qq =0}^{\infty} \sS ^\qq \} \bigcup \SN $. 
Let $ \map{\ulab }{\mathbb{\sS } }{\SST }$ be such that 
\begin{align}&\notag 
\ulab ((s_i)_i ) = \sum_{i} \delta_{s_i}
.\end{align}
Then, $ \ulab (\s ) = \sss $ for $ \s = (s_i)_{i}$ and $ \sss = \sum_{i} \delta _{s_i}$. 
Here $ \sS ^0 $ is regarded as $\sS ^0 = \{ \emptyset \} $ and $ \ulab (\emptyset )$ equals the zero measure. 
We call $ \ulab $ an unlabeling map.

We endow $ \SN $ with the product topology. 
For $ \www = \{\wwwt \} = \{ (\wti )\} \in \WSN $, we set 
\begin{align}\label{:22c}&
\upath (\www )_t := \ulab (\wwwt ) = \sum_{i=1}^{\infty} \delta_{\wti }
.\end{align}
We call $ \upath (\www ) $ the unlabeled path of $ \www $. 
Note that $ \upath (\www ) $ is not necessarily an element of $ \WS $, even if $\upath (\www )_t \in \sSS $ 
for all $ t $; see \cite[Remark 3.10]{o-t.tail}, for example. 

Let $ \Ssi $ be the subset of $ \sSS $ consisting of an infinite number of particles with no multiple points. 
By definition, 
$ \Ssi = \Ss \cap \Si $, where $ \Ss $ and $ \Si $ are given by 
\begin{align} \label{:22g} &
\Ss = \{ \sss \in \sSS \, ;\, \sss ( \{ x \} ) \le 1 \text{ for all } x \in \sS \} 
,\quad 
 \Si = \{ \sss \in \sSS \, ;\, \sss (\sS ) = \infty \} 
.\end{align} 
A measurable map $ \map{\lab }{\Ssi }{\SN }$ is called a label on $ \Ssi $ if 
$ \ulab \circ \lab (\sss ) = \sss $ for all $ \sss \in \Ssi $.

Let $ \WSs $ and $ \WSsi $ be the sets consisting of all $ \Ss $- and $ \Ssi $-valued continuous paths 
on $ [0,\infty)$. 
Each $ \ww \in \WSs $ can be written as $ \ww _t = \sum_i \delta_{w_t^i}$, 
where $ w^i $ is an $ \sS $-valued continuous path defined on an interval $ I_i $ of the form 
$ [0,b_i)$ or $ (a_i,b_i)$, where $ 0 \le a_i < b_i \le \infty $. 
Taking maximal intervals of this form, we can choose $ [0,b_i)$ and $ (a_i,b_i)$ uniquely up to labeling. 
We remark that 
$ \lim_{t\downarrow a_i} |w_t^i| =\infty $ and 
$ \lim_{t\uparrow b_i} |w_t^i| =\infty $ for $ b_i < \infty $ for all $ i $. 
We call $ w^i$ a tagged path of $ \ww $ and $ I_i $ the defining interval of $ w^i $. Let 
\begin{align}
\label{:22h}&
 \WSsiNE = \{ \ww \in \WSsi \, ;\, I_i = [0,\infty ) \text{ for all $ i $}\} 
.\end{align}
It is said that the tagged path $ w^i $ of $ \ww $ {\it does not explode} if $ b_i = \infty $, 
and {\it does not enter} if 
$ I_i = [0,b_i)$, where $ b_i $ is the right end of the defining interval of $ w^i $. 
Thus, $ \WSsiNE $ is the set consisting of all non-exploding and non-entering paths. 

We can naturally lift each $ \ww = \{\sum_i \delta_{w_t^i}\}_{t\in[0,\infty)} \in \WSsiNE $ 
to the labeled path 
\begin{align}&\notag 
\www =(\w ^i )_{i\in\N } = \{ \wwwt \}_{t\in[0,\infty)} = \{ (\wti )_{i\in\mathbb{N}} \}_{t\in[0,\infty)} \in \WSN 
\end{align}
using a label $ \lab = (\labi )_{i\in\N }$. 
Indeed, for each $ \ww \in \WSsiNE $, we can construct the labeled process 
$ \www =\{ (\wti )_{i\in\mathbb{N}} \}_{t\in[0,\infty)} $ such that 
$ \www _0 = \lab (\ww _0 ) $, because each tagged particle can carry the initial label $ i $ 
by the non-collision and non-explosion properties of $ \ww $. 
We write this correspondence as
\begin{align}&\label{:22i}
 \lpath (\ww ) =(\lpath ^i(\ww ))_{i\in\N }
.\end{align}
Setting $ \www =(\w ^i)_{i\in\N } = \lpath (\ww ) $, we have $ \w ^i =\lpath ^i(\ww ) $ by construction. 
We remark that $ \upath (\www )_t = \ulab (\wt )$ by \eqref{:22c}, whereas 
$ \lpath (\ww )_t \not= \lab (\ww _t)$ in general. 

For a labeled path $ \www = (w^i)$, we set $ \ww ^{m*} = \{ \ww _t^{m*} \}_{t \in [0,\infty )}$ by 
$ \ww _t^{m*} = \sum_{i>m}\delta_{\w _t^i}
$. 

We call the path 
$ \wwwm = (\lpath ^1(\ww ),\ldots,\lpath ^m(\ww ), \ww ^{m*}) $ an $ m $-labeled path. 
Similarly, for a labeled path $ \www =(\w ^i ) \in \WT (\SN )$, we set 
\begin{align}
\label{:22q}&
\www ^{[m]} = (w^1,\ldots,w ^m, \ww ^{m*}) 
.\end{align}

\subsection{ISDEs}\label{s:23}
Let $ \mathbf{X}= (X ^i)_{i\in\mathbb{N}}$ be an $ \SN $-valued continuous process. 
We write $ \mathbf{X} = \{ \mathbf{X}_t \}_{t\in[0,\infty )} $ and 
$ X^i = \{ X_t^i \}_{t\in[0,\infty )}$. 
For $ \mathbf{X}$ and $ i\in\mathbb{N}$, 
we define the unlabeled processes 
$ \XX =\{ \XX _t \}_{t\in[0,\infty )} $ and 
$ \XX ^{\idia } = 
\{ \XX _t^{\idia } \}_{t\in[0,\infty )} $ 
as $ \XX _t = \sum_{i\in\mathbb{N}} \delta_{X_t^i} $ and 
$ \XX _t^{\idia } = \sum_{j\in\mathbb{N} ,\ j\not=i } \delta_{X_t^j} $. 

Let $ \hH $ and $ \SSsde $ be Borel subsets of $ \sSS $ such that 
$ \hH \subset \SSsde \subset \Si $. 
Let $ \map{\ulabone }{\sS \ts \sSS }{\sSS }$ be such that 
$ \ulabone ((x,\sss )) = \delta _{x} + \sss $ for $ x \in \sS $ and $ \sss \in \sSS $. 
Define $ \SSSsde \subset \SN $ and $ \SSsde ^{[1]} \subset \sS \times \sSS $ by %
\begin{align}\label{:23y}&
\SSSsde = \ulab ^{-1} (\SSsde ) 
,\quad \SSsde ^{[1]} = \ulabone ^{-1} (\SSsde ) 
.\end{align}
Let 
$ \map{\sigma }{\SSsde ^{[1]} }{\mathbb{R}^{d^2}}$ and 
$ \map{\bbb }{\SSsde ^{[1]} }{\Rd }$ be Borel measurable functions. 
We consider an ISDE of $ \mathbf{X} =(X^i)_{i\in\mathbb{N}}$ starting on $ \lH $ with state space $ \SSSsde $ such that 
\begin{align}\label{:23a}&
dX_t^i = \sigma (X_t^i,\XX _t^{\idia }) dB_t^i + 
\bbb (X_t^i,\XX _t^{\idia }) dt \quad (i\in\mathbb{N}) 
,\\\label{:23b}&
\mathbf{X} \in \WT (\SSSsde ) 
,\\\label{:23c}&
\mathbf{X}_0 \in \lH 
.\end{align}
Here, $ \mathbf{B}=(B^i)_{i\in\N }$ is an $ \RdN $-valued Brownian motion, where $ \RdN = (\Rd )^{\N }$. 
By definition, $ \{B^i\}_{i\in\N }$ are independent copies of a $ d$-dimensional Brownian motion starting at the origin.

In infinite dimensions, it is natural to consider the coefficients 
$ \sigma $ and $ \bbb $ defined only on a suitable subset 
$ \SSsde ^{[1]} $ of $ \sS \ts \sSS $. 
From \eqref{:23b}, the process $ \mathbf{X}$ moves in the set $ \SSSsde $. 
Equivalently, the unlabeled dynamics $ \XX =\upath (\mathbf{X})$ 
move in $ \SSsde $. 
Moreover, each tagged particle $ X^i $ of $\mathbf{X}= (X^i )_{i\in\N }$ never explodes.

By \eqref{:23b}, $ \XX _t \in \SSsde $ for all $ t \ge 0$, and 
in particular the initial starting point $ \s $ in \eqref{:23c} is assumed to satisfy 
 $ \s \in \lH \subset \SSSsde $, which implies $ \ulab (\mathbf{s} ) \in \hH \subset \SSsde $. 
We take $ \hH $ such that \eqref{:23a}--\eqref{:23c} 
have a solution for each $ \s \in \lH $. 

Following \cite[Chapter IV]{IW} in finite dimensions, we present a set of notions related to solutions of ISDEs. 
In \dref{d:21}, we used the terminology ``weak solution'' instead of ``solution'' to distinguish it from the strong solution in \dref{d:24}. 

\begin{definition}[weak solution]\label{d:21} 
A weak solution of ISDE \eqref{:23a}--\eqref{:23b} is 
an $ \sS ^{\mathbb{N}} \ts \RdN $-valued continuous stochastic process $ \XB $ 
defined on a probability space $ (\Omega , \mathcal{F}, P )$ 
with a reference family $ \{ \mathcal{F}_t \}_{t \ge 0 } $ such that \thetag{i}--\thetag{iv} hold. 

\noindent 
\thetag{i} $ \mathbf{X}=(X^i)_{i=1}^{\infty} $ 
is an $ \SSSsde $-valued continuous process. 
Furthermore, $ \mathbf{X}$ is adapted to 
$ \{ \mathcal{F}_t \}_{t \ge 0 }$, 
that is, $ \mathbf{X}_t $ is 
$ \mathcal{F}_t /\Bt $-measurable for each $ \zti $, where 
\begin{align}&\notag 
\Bt = \sigma [ \www _s ; 0\le s \le t ,\, \www \in \WSN ] 
.\end{align}
\thetag{ii} $ \mathbf{B} = (B^i)_{i=1}^{\infty}$ is an $ \RdN $-valued 
\FtB 
with $ \mathbf{B}_0 = \mathbf{0}$. 
\\
\thetag{iii} 
The families of measurable $ \{ \mathcal{F}_t \}_{t \ge 0 } $-adapted 
processes $ \Phi ^i$ and $ \Psi ^i$ defined by 
\begin{align}&\notag 
\Phi ^i(t,\omega ) = \sigma 
(X_t^i(\omega ),\XX _t^{\idia }(\omega ))
,\quad 
\Psi ^i(t,\omega ) = \bbb 
(X_t^i(\omega ),\XX _t^{\idia }(\omega ))
\end{align}
belong to $ \mathcal{L}^{2} $ and $ \mathcal{L}^1 $, respectively. 
Here, $ \mathcal{L}^{p} $ is the set of all measurable, $ \{ \mathcal{F}_t \}_{t \ge 0 } $-adapted 
processes $ \alpha $ such that $ E[ \int_0^T|\alpha (t,\omega)|^p dt ] < \infty $ for all $ T $. 
We can and do take a predictable version of $ \Phi ^i $ and $ \Psi ^i$ (see pp. 45--46 in \cite{IW}). 
\\
\thetag{iv} With probability one, the process $ \XB $ satisfies, for all $ t $, 
\begin{align}&\notag 
X_t^i - X_0^i = 
\int_0^t \sigma (X_u^i,\XX _u^{\idia }) dB_u^i 
 + 
\int_0^t 
\bbb (X_u^i,\XX _u^{\idia }) du \quad (i\in\mathbb{N}) 
.\end{align}
\end{definition}

We say $ \mathbf{X}$ is a weak solution if the accompanied Brownian motion is obvious or not important. 

\begin{definition}[uniqueness in law]\label{d:22}
We say that the uniqueness in law of solutions starting on $ \lH $ for \eqref{:23a}--\eqref{:23b} holds if, 
whenever $ \mathbf{X}$ and $ \mathbf{X}'$ are two solutions 
whose initial distributions coincide, the laws of the processes 
$ \mathbf{X}$ and $ \mathbf{X}'$ on the space $ \WSN $ coincide. 
If this uniqueness holds for an initial distribution $ \delta _{\s }$, then we say that 
the uniqueness in law of solutions for \eqref{:23a}--\eqref{:23b} 
starting at $ \s $ 
holds. 
\end{definition}

\begin{definition}[pathwise uniqueness]\label{d:23}
We say that the pathwise uniqueness of solutions for 
\eqref{:23a}--\eqref{:23b} starting on $ \lH $ holds 
if, whenever $ \mathbf{X}$ and $ \mathbf{X}'$ are two solutions 
defined on the same probability space $ (\Omega , \mathcal{F} ,P )$ with the same reference family $ \{ \mathcal{F}_t \}_{t \ge 0 }$ and the same $ \RdN $-valued \FtB 
$ \mathbf{B} $ such that $ \mathbf{X}_0=\mathbf{X}_0' \in \lH $ a.s., 
\begin{align}\label{:23o}&
P (\text{$ \mathbf{X}_t=\mathbf{X}_t'$ for all $ t \ge 0 $}) = 1 
.\end{align} 
We say that the pathwise uniqueness of solutions starting at $ \s $ of 
\eqref{:23a}--\eqref{:23b} holds 
if 
\eqref{:23o} holds whenever the above conditions are satisfied and $ \mathbf{X}_0=\mathbf{X}_0' = \s $ a.s. 
\end{definition}

We now define a strong solution in a form that is analogous to 
Definition 1.6 in \cite[p. 163]{IW}. 
Let $ \PBr $ be the distribution of an $ \RdN $-valued Brownian motion 
$ \mathbf{B} $ with $ \mathbf{B}_0 = \mathbf{0}$. 
Let $ \WRNz = \{ \www \in \WT (\RdN )\, ; \www _0 =\mathbf{0} \}$. 
Clearly, $ \PBr (\WRNz )= 1 $. 

Let $ \Bt (\PBr )$ be the completion of 
$ \sigma [ \www _s ; 0\le s \le t ,\, \www \in \WRNz ] $ with respect to $ \PBr $. 
Let $ \mathcal{B}(\PBr ) $ be the completion of $ \mathcal{B}(\WRNz ) $ 
with respect to $ \PBr $. 

\begin{definition}[strong solution starting at $ \s $] \label{d:24} 
A weak solution $ \mathbf{X}$ of \eqref{:23a}--\eqref{:23b} 
with an $ \RdN $-valued $ \Ft $-Brownian motion $\mathbf{B}$ 
 defined on $ \OFPF $ is called a strong solution starting at $ \s $ 
if $ \mathbf{X}_0=\s $ a.s.\! and if there exists a function 
$ \map{\Fs }{\WRNz }{\WSN }$ such that 
$ \Bt (\PBr ) /\Bt $-measurable for each $ t $, and $ \Fs $ satisfies 
\begin{align}&\notag 
\mathbf{X} = \Fs (\mathbf{B}) \quad \text{ a.s.}
\end{align}
We also call $ \mathbf{X}= \Fs (\mathbf{B})$ a strong solution starting at $ \s $. 
Additionally, we call $ \Fs $ itself a strong solution starting at $ \s $. 
\end{definition}

\begin{definition}[a unique strong solution starting at $ \s $] \label{d:25} 
We say \eqref{:23a}--\eqref{:23b} have a unique strong solution starting at $ \mathbf{s}$ if 
there exists a function 
$ \map{\Fs }{\WRNz }{\WSN }$ 
such that, for any weak solution $ (\hat{\mathbf{X}},\hat{\mathbf{B}})$ 
of \eqref{:23a}--\eqref{:23b} starting at $ \mathbf{s}$, it holds that 
\begin{align}\notag &
\hat{\mathbf{X}}=\Fs (\hat{\mathbf{B}}) \quad \text{ a.s.}
\end{align}
and if, for any $ \RdN $-valued $ \{\mathcal{F}_t\} $-Brownian motion $ \mathbf{B} $ 
defined on $ \OFPF $ with $ \mathbf{B}_0=\mathbf{0}$, 
the continuous process $ \mathbf{X} = \Fs (\mathbf{B})$ 
is a strong solution of \eqref{:23a}--\eqref{:23b} starting at $ \mathbf{s}$. 
Also we call $ \Fs $ a unique strong solution starting at $ \mathbf{s}$. 
\end{definition}

We next present a variant of the notion of a unique strong solution. 

\begin{definition}[a unique strong solution under constraint] \label{d:26} 
For a condition \As{$ \bullet$}, we say \eqref{:23a}--\eqref{:23b} have a unique strong solution starting at $ \mathbf{s}$ under the constraint \As{$ \bullet$} if 
there exists a function $ \map{\Fs }{\WRNz }{\WSN }$ 
such that, for any weak solution 
$ (\hat{\mathbf{X}},\hat{\mathbf{B}})$ of \eqref{:23a}--\eqref{:23b} starting at $ \mathbf{s}$ 
satisfying \As{$ \bullet$}, it holds that 
\begin{align} &\notag 
\hat{\mathbf{X}}=\Fs (\hat{\mathbf{B}}) \quad \text{ a.s.}
\end{align}
and if for any $ \RdN $-valued $ \{\mathcal{F}_t\} $-Brownian motion $ \mathbf{B}$
defined on $ \OFPF $ with $ \mathbf{B}_0=\mathbf{0}$ 
the continuous process $ \mathbf{X} = \Fs (\mathbf{B})$ 
is a strong solution of \eqref{:23a}--\eqref{:23b} starting at $ \mathbf{s}$ satisfying \As{$ \bullet$}. 
Also we call $ \Fs $ a unique strong solution starting at $ \mathbf{s}$ 
under the constraint \As{$ \bullet$}. 
\end{definition}

\subsection{Finite-dimensional SDEs with random environments} \label{s:24}

Let $ (\sigma (x,\sss ), b (x,\sss )) $ be the coefficients of ISDE \eqref{:23a}. 
We set 
 \begin{align} \label{:24z}&
 \sigma ^m (\mathbf{y},\sss ) = (\sigma (y_i, \yy ^{\idia }+\sss ))_{i=1}^m , 
\quad b ^m (\mathbf{y},\sss ) = 
( b (y_i,\yy ^{\idia }+\sss ))_{i=1}^m 
 .\end{align}
Here, $ \yy ^{\idia } = \sum_{j\not=i}^m \delta_{y_j}$ for $ \mathbf{y}=(y_1,\ldots,y_m)$. 
For a given unlabeled process $ \XX = \sum_{i=1}^{\infty} \delta _{X^i}$, 
we define the functions 
$ \map{ \sigmaXms } 
{[0,\infty) \ts \sS ^{m} }{\mathbb{R}^{d^2}}$ and 
 $ \map{ \bbbXms }{[0,\infty) \ts \sS ^{m} }{\Rd }$ such that 
\begin{align} \label{:24b}
&
 \sigmaXms ( t, (u, \mathbf{v})) = 
 {\sigma } (u , \vv + \XX _t^{m*}) ,\ 
\bbbXms ( t, (u, \mathbf{v})) = {\bbb } (u , \vv + \XX _t^{m*}) 
,\end{align}
where $ (u,\mathbf{v}) \in \Sm $, 
$ \vv = \ulab (\mathbf{v}) := \sum_{i=1}^{m-1} \delta_{v_i} \in \sSS $, where 
$\mathbf{v}=(v_1,\ldots,v_{m-1}) \in \sS ^{m-1} $, and 
\begin{align}&\notag 
\XX ^{m*} = \sum_{i=m+1}^{\infty} \delta_{X^i }
.\end{align}
The coefficients $ \sigmaXms $ and $ \bbbXms $ depend on 
both the unlabeled path $ \XX $ and the label $ \lab $, although we omit $ \lab $ from the notation for simplicity. 
Let $ \SSSsdemt $ be the subset of $ \Sm $ such that 
\begin{align} \notag &
\SSSsdemt 
= \{ \s = (s_1,\ldots,s_m) \in \Sm \,;\, \ulab (\s ) + \XX _t^{m*} \in \SSsde \} 
.\end{align}

Let $ \XB $ be a weak solution of \eqref{:23a}--\eqref{:23b} defined on $ \OFPF $. 
Let 
\begin{align}\label{:24y}&
 \Ps = \pP (\cdot | \mathbf{X}_0=\mathbf{s}) 
.\end{align}
Then $ \XB $ under $ \Ps $ is a weak solution of \eqref{:23a}--\eqref{:23b} starting at $ \s =\lab (\sss )$. %
For such a weak solution $ \XB $ defined on $ \OFPFss $, 
we introduce the SDE with random environment $ \XX = \sum_{i\in\N } \delta_{X^i}$ 
describing $ \mathbf{Y}^m=(Y^{m,i})_{i=1}^m $ given by 
\begin{align}\label{:24f}&
dY_t^{m,i} = 
\sigmaXms (t, (Y_t^{m,i},\mathbf{Y}_t^{m,\idia })) dB_t^i + 
 \bbbXms (t, (Y_t^{m,i},\mathbf{Y}_t^{m,\idia })) dt 
,\\\label{:24g}& 
 \mathbf{Y}_t^m \in \SSSsdemt \quad \text{ for all } t 
,\\\label{:24h}&
 \mathbf{Y}_0^{m} = \s ^m 
.\end{align}
Here, we set $ \mathbf{Y}^{m,\idia } = (Y^{m,j})_{j\not=i}^m $. 
Moreover, 
$ \s ^m=(s_1,\ldots,s_m) $ and 
$ \mathbf{B}^m=(B^1,\ldots,B^m)$ denote the 
first $ m$ components of $ \s =(s_i)_{i\in\N } $ and 
$ \mathbf{B}=(B^i)_{i=1}^{\infty}$, respectively. 

We set $ \mathbf{X}^m = (X^1,\ldots,X^m) $ and $ \mathbf{X}^{m*} = (X^i)_{i=m+1}^{\infty}$.

\begin{definition}\label{d:27}
A triplet $ (\mathbf{Y}^{m},\mathbf{B}^{m},\mathbf{X}^{m*}) $ of 
$ \Ft $-adapted continuous processes defined on $ \OFPFss $ 
is called a weak solution of 
\eqref{:24f}--\eqref{:24h} if it satisfies \eqref{:24g}--\eqref{:24h} and, 
for all $ i \in \N $ and $ t \in [0,\infty)$, 
\begin{align}\notag &
Y_t^{m,i} - Y_0^{m,i}= \int_0^t
\sigmaXms (u, (Y_u^{m,i},\mathbf{Y}_u^{m,\idia })) dB_u^i + 
\int_0^t \bbbXms (u, (Y_u^{m,i},\mathbf{Y}_u^{m,\idia })) du 
.\end{align}
We also call this 
a weak solution of \eqref{:24f}--\eqref{:24g} starting at $ \s ^m $. 
\end{definition}

Clearly, $ (\mathbf{X}^m,\mathbf{B}^{m},\mathbf{X}^{m*})$ under $ \OFPFss $ is a weak solution of 
\eqref{:24f}--\eqref{:24h} for $ P \circ \mathbf{X}_0^{-1}$-a.s.\,$ \mathbf{s}$. 
 $ (\mathbf{B}^{m},\mathbf{X}^{m*})$ is given a priori as a part of 
the coefficients of SDE \eqref{:24f}.

We define the notion of strong solutions and a unique strong solution of \eqref{:24f}--\eqref{:24h}. 
Let $ o \in \sS $ and $ \mathbf{o}^m=(o,\ldots,o) \in \Sm $. We set 
$ \mathbf{X}^{m\circ *} = (\mathbf{o}^m, \mathbf{X}^{m*}) \in \WT (\SN ) $. 
By definition, the first $ m $ components of $\mathbf{X}^{m\circ *}$ consist of the constant path $ \mathbf{o}^m$. 
Here, $ o $ does not have any special meaning; it can be taken as any point in $ \sS $. 
Let 
\begin{align}
&\notag 
\PPPm = P \circ (\mathbf{B}^{m},\mathbf{X}^{m\circ *})^{-1}
.\end{align}
Let $ \WRdzm = \{ \mathbf{v} \in \WT (\Rdm ) \, ; \mathbf{v}_0 =\mathbf{0} \}$. 
We set 
\begin{align}&\notag 
\Ehatm = \overline{\mathcal{B} (\WWdm ) } ^{\PPPm }
.\end{align}
Let $ \Bt (\WWdm )= \sigma [(\mathbf{v}_s,\mathbf{w}_s) ; 0\le s \le t ] $. 
We set
\begin{align}&\notag 
\Ehatmt = \overline{\Bt (\WWdm ) }^{\PPPm }
.\end{align}
Let $ \Btm $ be the $ \sigma $-field on $ \WT (\Rdm )$ such that 
$ \Btm = \sigma [\mathbf{w}_s; 0\le s \le t ]$. 

We now state the definition of a strong solution. 
\begin{definition} \label{d:28}
A weak solution $ (\mathbf{Y}^{m},\mathbf{B}^m,\mathbf{X}^{m*})$ of \eqref{:24f}--\eqref{:24h}
defined on $ \OFPFss $ is called a strong solution 
if there exists a function 
\begin{align} &\notag \map{\Fsm }{\WWdm }{\WT (\Rdm ) }
 \end{align}
such that $ \Fsm $ is $ \Ehatm $-measurable, 
$ \Ehatmt /\Btm $-measurable for all $ t $, and satisfies 
\begin{align}&\label{:24p}
\mathbf{Y}^m = \Fsm (\mathbf{B}^{m},\mathbf{X}^{m\circ *}) \quad \text{$ \Ps $-a.s.}
\end{align}
\end{definition}
For simplicity, we write 
$ \Fsm (\mathbf{B}^{m},\mathbf{X}^{m*}) := \Fsm (\mathbf{B}^{m},\mathbf{X}^{m\circ *}) $. 
Then \eqref{:24p} becomes 
\begin{align}&\notag 
\mathbf{Y}^m = \Fsm (\mathbf{B}^{m},\mathbf{X}^{m*}) \quad \text{$ \Ps $-a.s.}
\end{align}

\bs

The solution $ \mathbf{Y}^{m}$ in \dref{d:28} is defined on $ \OFPFss $, 
where the weak solution $ \XB $ is defined. The Brownian motion $ \mathbf{B}^m $ 
in \eqref{:24f} is the first $ m $ components of $ \mathbf{B} $, and 
$ \mathbf{Y}^{m}$ is a function of not only $ \mathbf{B}^m $, but also $ \mathbf{X}^{m*} $. 
These properties are different from those of the conventional strong solutions of SDEs. 

Note that for any weak solution $ \XB $, we obtain the weak solution 
$ (\mathbf{X}^{m},\mathbf{B}^m,\mathbf{X}^{m*})$ of \eqref{:24f}--\eqref{:24h}. 
We recall the notion of a unique strong solution from \cite{o-t.tail}. 
\begin{definition}\label{d:29} 
The SDE \eqref{:24f}--\eqref{:24h} is said to have a unique strong solution for $ \XB $ under $ \Ps $ if 
there exists a function $ \Fsm $ satisfying 
$ \hat{\mathbf{Y}}^m = \Fsm (\mathbf{B}^{m},\mathbf{X}^{m*}) $ a.s.\,and the conditions in \dref{d:28} 
for any solution $ (\hat{\mathbf{Y}}^m ,\mathbf{B}^m,\mathbf{X}^{m*})$ of 
\eqref{:24f}--\eqref{:24h} defined on $ \OFPFss $. 
\end{definition}
The function $ \Fsm $ in \dref{d:28} is called a strong solution of \eqref{:24f}--\eqref{:24h}. 
The SDE \eqref{:24f}--\eqref{:24h} is said to have 
a unique strong solution $ \Fsm $ defined on $ \OFPFss $ if $ \Fsm $ satisfies the condition in \dref{d:29}. 
The function $ \Fsm $ is unique for $ \PPPm $-a.s. 
Following \cite{o-t.tail}, we set the following condition: 

\smallskip 
\noindent 
\As{IFC} The SDE \eqref{:24f}--\eqref{:24h} has a unique strong solution 
$ \Fsm (\mathbf{B}^m, \mathbf{X}^{m*})$ for $ \XB $ 
under $ \Ps $ for $ \pP \circ \mathbf{X}_0^{-1} $-a.s.\,$ \mathbf{s}$ for all $ m \in \N $.

\subsection{A unique strong solution of ISDEs} \label{s:25}

In \sref{s:25}, we quote results in \cite{o-t.tail}, which use \As{IFC} as one of the main assumptions. 
Similarly as \sref{s:24}, 
we set $ \XB $ to be a weak solution of \eqref{:23a}--\eqref{:23b} defined on $ \OFPF $. 
Let $ \Ps $ be as \eqref{:24y}. 
We quote a sufficient condition for $ \XB $ under $ \Ps $ to be a unique strong solution from \cite{o-t.tail}.

Let $ \TS = \bigcap_{r=1}^{\infty} \sigma [\pirc ] $ 
be the tail $ \sigma $-field on the configuration space $ \sSS $ over $ \Rd $. 
Here, $\pirc $ is the projection $ \map{\pirc }{\sSS }{\sSS }$ such that $ \pirc (\sss ) = \sss (\cdot \cap \Sr ^c)$, 
where $ \Sr = \{ |x| < r \} $. 
Let $ \mu $ be a random point field on $ \sSS $. 
$ \mu $ is said to be tail trivial if $ \mu ( A ) \in \{ 0,1 \} $ for all $ A \in \TS $. 
Let $\WSsiNE $ be as in \eqref{:22h}, and set $ \XX = \ulab (\mathbf{X} )$ as before. 
For $ \mathbf{X}=(X^i)$, we set 
\begin{align}& \label{:25f}
\mrXX = \inf \{ m \in \mathbb{N}\, ;\, 
\min_{t\in[0,T]}|X _t^i | > r \text{ for all } i \in \mathbb{N} \text{ such that } i > m \} 
.\end{align}
We make the following assumptions: 
\smallskip 

\noindent \As{TT} \ $ \mu $ is tail trivial. 

\noindent \As{AC}\ $ \pP \circ \XX _t^{-1} \prec \mu $ for all $ \zzti $. 

\noindent \As{SIN} \ $ \pP ( \XX \in \WSsiNE ) = 1 $. 

\noindent \As{NBJ} \ $ \pP ( \mrXX < \infty ) = 1 $ for each $ r , T \in \mathbb{N} $. 

\smallskip

We define the conditions \As{AC}, \As{SIN}, and \As{NBJ} for 
a probability measure $ \widehat{\pP }$ on $ \WRN $ 
by replacing $ \XX $ and $ \mathbf{X}$ by $ \ww $ and $ \mathbf{w} $, respectively.

We introduce the condition \As{MF} 
for a family of strong solutions $ \{ \Fs \} $ of \eqref{:23a}--\eqref{:23b} 
starting at $ \pP \circ \mathbf{X}_0^{-1} $-a.s.\,$ \mathbf{s}$. 

\noindent 
\As{MF} \quad 
$ \pP (\Fs ( \mathbf{B} ) \in A ) $ is 
 $ \overline{\mathcal{B}(\SN )}^{ \pP \circ \mathbf{X}_0^{-1} } $-measurable in $ \mathbf{s} $ 
for any $ A \in \mathcal{B}(\WSN ) $. 

\smallskip

\noindent 
For a family of strong solutions $ \{ \Fs \} $ satisfying \As{MF} 
 we set 
\begin{align}\label{:25z}&
\pPF = \int \pP ( \Fs (\mathbf{B} ) \in \cdot ) \pP \circ \mathbf{X}_0^{-1} (d\mathbf{s}) 
.\end{align}
We remark that, if $ \XB $ is a weak solution under $ \pP $ and 
a unique strong solution under $ \Ps $ for $ \pP \circ \mathbf{X}_0^{-1}$-a.s.\,$ \mathbf{s}$, 
then \As{MF} is automatically satisfied and 
\begin{align}\label{:25y}&
 \pPF = \pP \circ \mathbf{X}^{-1}
.\end{align}
Here, $ \Fs $ denotes the unique strong solution given by $ \XB $ under $ \Ps $. 
Indeed, $ \mathbf{B}$ is a Brownian motion under both $ \pP $ and $ \Ps $, and 
for $ \pP \circ \mathbf{X}_0^{-1} $-a.s.\,$ \mathbf{s}$ 
\begin{align}\label{:25a}&
\pP ( \Fs (\mathbf{B} ) \in \cdot ) = \Ps ( \Fs (\mathbf{B} ) \in \cdot ) = \Ps (\mathbf{X} \in \cdot )
.\end{align}
Hence we deduce \eqref{:25y} from \eqref{:25z} and \eqref{:25a}.

\begin{definition}\label{dfn:43}
For a condition \As{$ \bullet $}, 
we say \eqref{:23a}--\eqref{:23b} 
has a family of unique strong solutions $ \{ \Fs \} $ starting at $ \mathbf{s}$ for 
$ \pP \circ \mathbf{X}_0^{-1} $-a.s.\,$ \mathbf{s}$ 
under the constraints of \As{$\mathbf{MF}$} and \As{$ \bullet $} if 
$ \{ \Fs \} $ satisfies \As{$\mathbf{MF}$} and $ \pPF $ satisfies \As{$ \bullet $}. Furthermore, 
\thetag{i} and \thetag{ii} are satisfied.

\noindent\thetag{i} 
For any weak solution $ (\hat{\mathbf{X}},\hat{\mathbf{B}})$ under $ \hat{P}$ 
of \eqref{:23a}--\eqref{:23b} with 
$ \hat{P}\circ \hat{\mathbf{X}}_0^{-1} \prec P \circ \mathbf{X}_0^{-1} $ 
satisfying \As{$ \bullet $}, it holds that, for $\hat{P} \circ \hat{\mathbf{X}}_0^{-1}$-a.s. $ \mathbf{s}$, 
\begin{align*}&
\text{$ \hat{\mathbf{X}}=\Fs (\hat{\mathbf{B}}) $  $ \hat{P}_{\mathbf{s}} $-a.s.,}
\end{align*}
where $ \hat{P}_{\mathbf{s}} = \hat{P}(\cdot | \hat{\mathbf{X}}_0={\mathbf{s}})$. 

\noindent 
\thetag{ii} 
For an arbitrary $ \RdN $-valued $ \{\mathcal{F}_t\} $-Brownian motion $ \mathbf{B}$ 
defined on $ \OFPF $ with $ \mathbf{B}_0=\mathbf{0}$, the continuous process 
$ \mathbf{X} = \Fs (\mathbf{B})$ is a strong solution of \eqref{:23a}--\eqref{:23b} satisfying \As{$ \bullet $} 
starting at $ \mathbf{s}$ 
for $ P \circ \mathbf{X}_0^{-1}$-a.s.\,$ \mathbf{s}$. 
\end{definition}

We quote two results from \cite{o-t.tail}. 
Both show usefulness of the \As{IFC} condition. \pref{l:22} \thetag{1} is used in \cite{o.udf} 
to prove the identity $ (\underline{\E}, \underline{\dom }) = (\E ,\dom )$ explained 
 in \sref{s:1}. 
\begin{proposition}	[{\cite[{Theorem 3.1}]{o-t.tail}}]\label{l:21}
Assume \As{TT}. 
Assume that \eqref{:23a}--\eqref{:23b} has a weak solution $ \XB $ satisfying {\INmu}. 
Then, \eqref{:23a}--\eqref{:23b} has a family of unique strong solutions $ \{ \Fs \} $ 
starting at $\mathbf{s} $ for $ P \circ \mathbf{X}_0^{-1} $-a.s.\,$ \mathbf{s} $ under the constraints 
of \As{MF}, {\INmu}. 
\end{proposition}

\begin{proposition}	[{\cite[{Corollary 3.2}]{o-t.tail}}]\ \label{l:22}
Under the same assumptions as \pref{l:21} the following hold. 

\noindent \thetag{1} 
The uniqueness in law of weak solutions of \eqref{:23a}--\eqref{:23b} holds 
under the constraints of {\INmu}. 

\noindent \thetag{2} 
The pathwise uniqueness of weak solutions of \eqref{:23a}--\eqref{:23b} holds 
under the constraints of {\INmu}. 
\end{proposition}

\begin{remark}\label{r:23} \thetag{1}
All determinantal random point fields on continuous spaces are tail trivial \cite{o-o.tail,ly.18,bqs}. 
Suppose that $ \mu $ is a quasi-Gibbs measure in the sense of \dref{d:81}. 
Then, $ \mu $ can be decomposed into 
tail trivial components, and each component satisfies {\INmu}. 
We can apply \pref{l:21} to each component $ ($see \cite[Theorem 3.2]{o-t.tail}$ )$. 
Thus, \As{TT} is not restrictive. 
\\\thetag{2} 
\As{AC} is obvious if $ \mu $ is an invariant probability measure of $ \XX _t $ and 
$ \XX _0 \elaw \mu $. 
All examples in the present paper satisfy this condition. 
\As{SIN} and \As{NBJ} are also mild assumptions. We refer to \cite[Sections 10, 12]{o-t.tail} for sufficient conditions. 
\end{remark}

\section{Main theorems (Theorems \ref{l:31}--\ref{l:33}): A sufficient condition for IFC } \label{s:3}

We shall localize the coefficients of SDE \eqref{:24f} to deduce the IFC condition. 
For this, we introduce a set of subsets in $ \SmSS $. 

Let $ \mathbf{a}=\{ \ak \}_{\qq \in\mathbb{N}} $ be a sequence of 
increasing sequences $ \ak = \{ \ak (\rR ) \}_{\rR \in\mathbb{N}} $ of natural numbers such that 
$ \ak (\rR ) < \ak (\rR +1)$ and $ \ak (\rR ) < \akk (\rR )$ for all $\qq , \rR \in \mathbb{N}$. 
We set 
\begin{align}\label{:30w}&
\Ka = \bigcup_{\qq =1}^{\infty} \K [\ak ] , \quad 
\K [\ak ] =\{ \sss \in \sSS \, ;\, \sss ( \overline{\sS }_{\rR }   ) \le \ak (\rR ) \text{ for all }\rR \in \mathbb{N} \} 
.\end{align}
Here $ \overline{\sS }_{\rR } = \{ x \in \sS ; |x| \le \rR \} $. 
By construction, $ \K [\ak ] \subset \K [\akk ]$ 
for all $ \qq \in\mathbb{N}$. It is well known that $ \K [\ak ]$ is relatively compact 
 in $ \sSS $ for each $ \qq \in \mathbb{N}$.

We introduce an approximation of $ \SmSS $. 
Let $ \Ssi $ be as in \eqref{:22g}. 
By definition, $ \Ssi $ is the set consisting of infinite configurations with no multiple points. 
Let $ \mathbf{x}=(x_1,\ldots,x_m) \in \Sm $, $ \ulab (\mathbf{x}) = \sum_{i=1}^m \delta_{x_i}$, 
and $\sss =\sum_i \delta_{s_i}$. We set 
\begin{align}\notag &
\Ssi ^{[m]} = \{ (\mathbf{x},\sss )\in \SmSS \, ;\, \ulab (\mathbf{x}) + \sss \in \Ssi \} 
.\end{align}
Let $ j,k,l=1,\ldots,m $ and set 
\begin{align} \notag &
\RRprCs = \big\{ \mathbf{x} \in \Srm \, ;\, \ 
\inf_{j\not=k } |x_j-x_k | > 2^{-\pp } ,\ 
\inf_{l,i} |x_l-s_i| > 2^{-\pp } \big\} 
,\\\notag &
\RRprs = \big\{ \mathbf{x} \in \RRr \, ;\, \min_{j\not=k } |x_j-x_k | \ge 2^{-\pp } ,\ 
\inf_{l,i} |x_l-s_i| \ge 2^{-\pp } \big\} 
.\end{align}
Then $ \RRprCs $ is an open set and $ \RRprs $ is its closure in $ \Sm $. 

Let $ \{ \ak ^+ (\rR )\}_{\rR \in\mathbb{N}} $ be such that $ \ak ^+ (\rR ) = 1 + \ak (\rR +1)$. 
We set 
\begin{align} \notag &
\Ha _{\pqr } ^{\circ } = \big\{ (\mathbf{x},\sss ) \in \Ssi ^{[m]} \, ;\, \ 
\mathbf{x} \in \RRprCs ,\, \sss \in \Kakk 
\big\} 
,\\ \notag &
\Hb = \big\{ (\mathbf{x},\sss ) \in \Ssi ^{[m]} \, ;\, \ \mathbf{x} \in \RRprs ,\ \sss \in \Kakk \big\} 
.\end{align}
By construction, $ \Hb $ and $ \Ha _{\pqr } ^{\circ }$ are relatively compact. 
Let 
\begin{align} \notag 
&
 \Ha _{\rr }^{\circ } = \bigcup_{\qq =1}^{\infty} \Ha _{\qq ,\rr }^{\circ }, 
\ 
\Ha _{\qq ,\rr }^{\circ } = \bigcup_{\pp =1}^{\infty} \Hb ^{\circ } 
,\ 
 \Ha _{\rr } = \bigcup_{\qq =1}^{\infty} \Ha _{\qq ,\rr }, 
\ 
\Ha _{\qq ,\rr } = \bigcup_{\pp =1}^{\infty} \Hb 
.\end{align}
We set 
\begin{align}&\label{:30c}
\Ha = \bigcup_{\rr =1}^{\infty} \Ha _{\rr }^{\circ } = \bigcup_{\rr =1}^{\infty} \Ha _{\rr } 
.\end{align}
Although $ \Ha _{\pqr } ^{\circ }$ and other quantities depend on $ m \in \N $, 
we omit $ m $ from the notation.

We set $ \NNN = \NNN _1 \cup \NNN _2 \cup \NNNthree $, where 
\begin{align}
\notag &
\NNN _1 = \{ \rr \in \N \}, \ 
\NNN _2 = \{ (\qq , \rr ) \ ;\, \qq , \rr \in \N \} , \ 
\NNNthree = \{ (\pqr ) \ ;\, \pqr \in \N \}
,\end{align}
and for $ \nnNNN $, we define $ \nn + 1 \in \NNN $ as 
\begin{align}
\notag &
 \nn + 1 = 
\begin{cases}
(\pp + 1, \qq , \rr ) &\text{ for $ \nn = (\pqr ) \in \NNNthree $, }\\
(\qq +1, \rr )&\text{ for $ \nn = (\qq , \rr ) \in \NNN _2$, }\\
 \rr +1 &\text{ for $ \nn = \rr \in \NNN _1$.}
\end{cases}
\end{align}
We write $ \Han = \Ha _{\pqr }$ for $ \nn = (\pqr ) \in \NNNthree $, and set 
$ \Han $ for $ \nn = (\qq ,\rr ) \in \NNN _2$ and $ \nn = \rr \in \NNN _1$ similarly. 
We set $ \HanC $ analogously. 
Clearly, for all $ \nn \in \NNN $
\begin{align}\notag &
\Han ^{\circ }\subset \Han , \quad 
\Han \subset \Hann 
.\end{align}
We shall take the limit in $ \nn $ along with the order $ \nn \mapsto \nn + 1$ such that 
\begin{align} \notag &
\limi{\nn } := \limi{\rr }\limi{\qq }\limi{\pp }
.\end{align}

For $ \nn = (\pqr ) \in \NNNthree $ and $ (\mathbf{x} ,\sss ) , (\mathbf{y} ,\sss ) \in \RRprCs $, we set 
$ (\mathbf{x} ,\sss ) \sim_{\nn } (\mathbf{y} ,\sss ) $ 
if $ \mathbf{x} $ and $ \mathbf{y} $ are in the same connected component of 
$ \RRprCs $ and $ \sss \in \HmnPc $. 
Here $ \Pi _2 $ is a projection $ \map{\Pi _2 }{\SmSS }{\sSS }$ given by $ \Pi _2 (\mathbf{x},\sss ) = \sss $.

Let $ \XB $ be a weak solution of ISDE \eqref{:23a}--\eqref{:23c} defined on $ \OFPF $. 
For $ \mathbf{X}=(X^i)_{i\in\N }$, we set the $ m $-labeled process 
$ \XM =(\mathbf{X}^m, \XX ^{m*}) $ such that 
\begin{align}
\label{:31a}&
\mathbf{X}^m = (X^1,\ldots,X^m), \quad 
\XX _t^{m*} = \sum_{j=m+1}^{\infty} \delta_{X_t^j} 
.\end{align}
Let $ \varsigma _{\nn } \UV $ be the exit time from $ \HanC $. 
By definition, $ \varsigma _{\nn } \UV $ is a function on the $ m $-labeled path space $ \WT (\Sm \times \sSS ) $ such that 
\begin{align}& \label{:31d}
 \varsigma _{\nn } \UV = \inf \{ t > 0 \, ;\, \UV _t \not\in \HanC \} 
.\end{align}

\noindent 
\Ass{B1} $ \XM =(\mathbf{X}^m,\XX ^{m*})$ does not exit from 
$ \Ha = \bigcup_{\nn \in \NNN }\HanC $, that is, 
\begin{align}\label{:31e}& 
\pP (\limi{\nn } \varsigma _{\nn }(\mathbf{X}^m , \XX ^{m*}) = \infty ) = 1 
.\end{align}

We extend the domain of $ \ulab $ from $ \Sm $ to $ \SmSS $ such that 
$ \ulab (\mathbf{x},\sss ) = \ulab (\mathbf{x}) + \sss $. 
Let $ \sigmam $ and $ b^m$ be as in \eqref{:24z}. 
Then we make the following assumptions: 

\smallskip 

\noindent 
\Ass{B2} 
The inclusion $ \ulab (\Ha ) \subset \SSsde $ holds. 
Furthermore, for each $ \nn \in \NNNthree $ and $ T \in \N $, there exists a function $ \FF $ defined on $ \Sm \ts \sSS $ 
satisfying for each $ f \in \{ \sigmam , b^m \} $ and for $ \pP $-a.s.
\begin{align}\label{:31b}&
| f (\mathbf{x} , \XX _t^{m*} ) - f (\mathbf{y},\XX _t^{m*} ) | \le |\mathbf{x} - \mathbf{y} |
 \FF (\mathbf{x} , \XX _t^{m*} ) 
\end{align}
 for all $ 0\le t \le T $ and all $ \mathbf{x} , \mathbf{y} \in \HanC $ such that 
$ (\mathbf{x} ,\XX _t^{m*} ) \sim_{\nn } ( \mathbf{y} , \XX _t^{m*} ) $. 
 
\noindent 
\Ass{B3} 
 The coefficient $\sigmam $ is a constant function and, for each $ \nn \in \NNNthree $ and $ T \in \N $, 
\begin{align}& \label{:31f}
E[ \int_0^T 1_{\HanC } (\mathbf{X}_t ^{m*} , \XX _t^{m*} ) 
 \big| \FF (\mathbf{X}_t ^{m*} , \XX _t^{m*} ) \big|^{ \ppp } dt ] < \infty 
\end{align}
for some $ \ppp > 1 $. 

\noindent 
\Ass{B4} For each $ \nn \in \NNNthree $ and $ T \in \N $, 
\begin{align}\label{:31g}&
\sup \{\big| \FF ( \mathbf{x} , \sss ) \big| ; ( \mathbf{x} , \sss )\in \Han \} < \infty 
.\end{align}
Furthermore, filtrations satisfy $ \{ \mathcal{F}_t '\}=\{ \mathcal{F}_t '' \}$. 
Here $ \{ \mathcal{F}_t '\}$ and $\{ \mathcal{F}_t '' \}$ are filtrations on a measurable space 
$(\Omega ',\mathcal{F}')$.

The critical step is to prove the pathwise uniqueness of 
weak solutions to the finite-dimensional SDE \eqref{:24f} of $ \mathbf{Y}^m$ for $ \XB $. 
\begin{theorem}\label{l:31} Let 
$ (\mathbf{Z}^m , \hat{\mathbf{B}}^m , \hat{\XX }^{m*}) $ and 
$ (\ZmH , \hat{\mathbf{B}}^m , \hat{\XX }^{m*}) $ 
be weak solutions of \eqref{:24f}--\eqref{:24h} defined on 
$(\Omega ',\mathcal{F}', P' ,\{ \mathcal{F}_t '\} )$ and 
$(\Omega ',\mathcal{F}', P' ,\{ \mathcal{F}_t '' \} )$, respectively. Assume that 
\begin{align}
\label{:31m}&
\XBXm 
 \elaw (\mathbf{Z}^m , \hat{\mathbf{B}}^m , \hat{\XX }^{m*}) \elaw 
 (\ZmH , \hat{\mathbf{B}}^m , \hat{\XX }^{m*})
.\end{align}
Let \Ass{B1} and \Ass{B2} hold for $ m \in \N $. 
Let either \Ass{B3} or \Ass{B4} hold for $ m \in \N $. 
Then, 
\begin{align} \label{:31n}&
P ' (\mathbf{Z}^m = \ZmH ) = 1 
.\end{align}
\end{theorem}

Once the pathwise uniqueness of weak solutions has been established, we can deduce the existence of a unique strong solution through an analogy of the Yamada--Watanabe theory. 
By using the same argument of the proof of \cite[Proposition 11.1]{o-t.tail}, \tref{l:31} yields the next theorem.

\begin{theorem} \label{l:32}
Assume that \Ass{B1}, \Ass{B2}, and \Ass{B3} hold for all $ m \in \N $. 
Then, $ \XB $ satisfies \As{$ \mathbf{IFC}$}. 
\end{theorem}

\begin{remark}\label{r:31} 
It is plausible that \tref{l:32} holds if we substitute \Ass{B3} by \eqref{:31g}. 
An additional element $ \hat{\XX }^{m*} $ prevents us from direct usage of the Yamada--Watanabe theory. 
Clearly, the condition \eqref{:31f} is weaker than \eqref{:31g}. 
\end{remark}

For $ l \in \{ 0 \}\cup \N $, let 
$ \mathbf{J}_{[l]} = \{ \mathbf{j}= (j_{k,i})_{1 \le k \le m,\, 1 \le i \le d} 
;\ j_{k,i} \in \{ 0 \} \cup \N ,\, \sum_{k=1}^m \sum_{i=1}^{d} j_{k,i}= l \} 
$. 
We set 
$ \partial _{\mathbf{j}} = \prod_{k,i} (\partial /\partial x_{k,i} )^{j_{k,i}} $ for 
$ \mathbf{j}=( j_{k,i}) \in \mathbf{J}^{[l]} $, where 
$ x_k=(x_{k,i})_{i=1}^d \in \Rd $, 
and $ ({\partial }/{\partial x_{k,i}})^{j_{k,i}}$ denotes the identity if $ j_{k,i}=0 $. 

We assume that there exists some $ \ell = \ell (m) \in \N $ satisfying  \Ass{C1} and \Ass{C2}. 

\smallskip
\noindent 
\Ass{C1} 
For each $ m \in \N $ and $ \nn \in \NNNthree $, there exists a constant $ \Ct \label{;F0j}$ 
satisfying the following. 
For $ \mum $-a.e.\,$ (\mathbf{x} ,\sss ) , (\xi ,\sss ) \in \HanC $ satisfying $ (\mathbf{x} ,\sss ) \sim_{\nn } (\xi ,\sss ) $, 
there exists a set of points $ \{ \mathbf{x} _1,\ldots,\mathbf{x} _k \} $ in $ \Sm $ with 
$ (\mathbf{x} _1,\mathbf{x} _k)=(\mathbf{x} , \xi ) $ such that 
\begin{align}
\label{:33j}&
\sum_{j=1}^{k-1} |\mathbf{x} _j-\mathbf{x} _{j+1}| \le \cref{;F0j} |\mathbf{x} -\xi | 
,\quad 
[\mathbf{x} _j,\mathbf{x} _{j+1}] \ts \{ \sss \} \subset \Han \quad \text{ $ (j=1,\ldots,k-1)$}
,\end{align}
and that $ \partial _{\mathbf{j}} \sigmam _{j,\sss }(t)$ and 
$ \partial _{\mathbf{j}} b^m _{j,\sss }(t)$ are absolutely continuous in 
$ t \in [0,1]$ for each 
$ \mathbf{j} \in\JL{\ellell -1}{m} $. 
Here 
$ \partial _{\mathbf{j}} \sigmam _{j,\sss }(t) := 
(\partial _{\mathbf{j}} \sigmam ) ( t \mathbf{x} _j + (1-t) \mathbf{x} _{j+1} ,\sss ) $ 
and we set $ \partial _{\mathbf{j}} b^m _{j,\sss }(t)$ similarly. 
Furthermore, $[\mathbf{x} _j,\mathbf{x} _{j+1}] $ is the segment connecting 
$ \mathbf{x} _j $ and $\mathbf{x} _{j+1} $.

\smallskip 
\noindent 
\Ass{C2} For each $ \mathbf{j} \in \JL{\ellell }{m} $, there exist 
$ \g _{\mathbf{j}},\, \h _{\mathbf{j}}\in C (\sS ^2 \cap \{ x \not= s \}) $ such that, on $ \Ha $,
\begin{align}& \notag 
\partial _{\mathbf{j}} \sigmam (\mathbf{x} ,\sss ) = 
\Big(
\sum_{j\not = k}^m \g _{\mathbf{j}} (x_k - x_j ) 
+
\sum_{i} \g _{\mathbf{j}} (x_k -s_i ) \Big)_{k=1}^m 
,\ 
\\ \notag &
\partial _{\mathbf{j}} b^m (\mathbf{x} ,\sss ) = 
\Big(
\sum_{j\not = k}^m \h _{\mathbf{j}} (x_k - x_j ) 
+
\sum_{i} \h _{\mathbf{j}} (x_k -s_i ) \Big)_{k=1}^m 
,\end{align}
where $ \mathbf{x}=(x_1,\ldots,x_m) \in \Sm $, 
and the constant $ \Ct (\nn )\label{;U2c}$ is finite for each $ \nn \in \NNNthree $: 
\begin{align} \label{:33m} 
& \cref{;U2c}(\nn ):=
\sup \big\{ \sum_{k=1}^m
\sum_{i} |\g _{\mathbf{j}} (x_k- s_i ) | + 
|\h _{\mathbf{j}} (x_k- s_i ) | ; \, (\mathbf{x} ,\sss ) \in \Han \big\} < \infty 
.\end{align}
We refer to \cite[Lemma 13.1]{o-t.tail} for a simple sufficient condition for \eqref{:33m}. 

\begin{theorem} \label{l:33}
Assume that there exists some $ \ell = \ell (m) \in \N $ satisfying \Ass{C1}--\Ass{C2} for each $ m \in \N $. 
Then, \Ass{B2}, \eqref{:31f}, and \eqref{:31g} hold for each $ m \in \N $. 
\end{theorem}

We shall give sufficient conditions for \Ass{B1} in \sref{s:5} and \sref{s:6}. 

\section{Proofs of \tref{l:31}, \tref{l:32}, and \tref{l:33}} \label{s:4}

In this section, we prove \tref{l:31}, \tref{l:32}, and \tref{l:33}. 

\noindent 
{\em Proof of \tref{l:31} } 
Let $ \varsigma _{\nn }$ be the exit time from the tame set $ \HanC $ defined by \eqref{:31d}. 
From \eqref{:31m}, we see that 
\begin{align}\label{:41o}&
\varsigma _{\nn } (\mathbf{X}^m , \XX ^{m*}) \elaw 
\varsigma _{\nn } (\mathbf{Z}^m ,\hat{\XX }^{m*}) \elaw 
\varsigma _{\nn } (\ZmH ,\hat{\XX }^{m*})
.\end{align}
From \Ass{B1}, we can deduce that, for $ \pP $-a.s., 
$ \varsigma _{\nn } (\mathbf{X}^m , \XX ^{m*}) > 0 $ for sufficiently large $ \nn $. 
Combined with \eqref{:41o}, this yields, for $ \pP $-a.s., 
\begin{align}
\label{:41p}& 
\smn := 
\min \{ 
\varsigma _{\nn } (\mathbf{Z}^m ,\hat{\XX }^{m*}) , 
\varsigma _{\nn } (\ZmH ,\hat{\XX }^{m*}) 
\} > 0 
\quad \text{ for sufficiently large $ \nn $}
.\end{align}

From \eqref{:24z} and \eqref{:24b} we rewrite \eqref{:24f} as 
\begin{align}&\label{:41q}
 \mathbf{Y} _t^m - \mathbf{Y} _0^m = 
\int_0^t \sigmam (\mathbf{Y} _u^m ,{\XX }_u^{m*}) d\BmHu + 
\int_0^t b^m (\mathbf{Y} _u^m ,{\XX }_u^{m*}) du 
.\end{align}
Then $ (\mathbf{X} ^m ,\mathbf{B}^m ,{\XX }_u^{m*})$ is a solution of \eqref{:41q}. 
Hence, we deduce from \eqref{:31m} that 
$ (\mathbf{Z}^m , \hat{\mathbf{B}}^m , \hat{\XX }^{m*}) $ and 
$ (\ZmH , \hat{\mathbf{B}}^m , \hat{\XX }^{m*}) $ 
satisfy 
$ \mathbf{Z}_0^m = \ZmH _0 = \mathbf{s}^m $ and 
\begin{align}\notag &
 \mathbf{Z}_t^m - \mathbf{Z}_0^m = 
\int_0^t \sigmam (\Zum ,\hat{\XX }_u^{m*}) d\BmHu + \int_0^t b^m (\Zum ,\hat{\XX }_u^{m*}) du 
,\\ \notag &
\ZmH _t -\ZmH _0 = 
\int_0^t \sigmam (\ZumH ,\hat{\XX }_u^{m*}) d\BmHu + \int_0^t b^m (\ZumH ,\hat{\XX }_u^{m*}) du 
.\end{align}
From these two equations, we have 
\begin{align}\label{:41t}
\mathbf{Z}_t^m - \ZtmH 
 = & 
\int_0^t 
\{ \sigmam (\Zum ,\hat{\XX }_u^{m*}) - 
 \sigmam (\ZumH ,\hat{\XX }_u^{m*}) \} d\BmHu 
\\ \notag &
+ 
\int_0^t \{
b^m (\Zum ,\hat{\XX }_u^{m*})
 - 
b^m (\ZumH ,\hat{\XX }_u^{m*})
\} du 
.\end{align}

Assume \Ass{B3}. Then, because $ \sigmam $ is constant by assumption, the difference in the martingale terms of 
$ \mathbf{Z}^m$ and $ \ZmH $ is canceled out. 
Hence, we have from \eqref{:41t} 
\begin{align}\label{:41a}&
\mathbf{Z}_t^m - \ZtmH = \int_0^t 
b^m (\Zum ,\hat{\XX }_u^{m*})
 - 
b^m (\ZumH ,\hat{\XX }_u^{m*})
du 
.\end{align}
From \eqref{:31b}, we deduce that, for $ 0 \le u \le \smnT $, 
\begin{align}
\label{:41b}& 
|b^m (\Zum ,\hat{\XX }_u^{m*}) - 
b^m (\ZumH ,\hat{\XX }_u^{m*})| \le 
| \Zum - \ZumH | 
 \FF (\Zum , \hat{\XX }_u^{m*} ) 
.\end{align}
Combining \eqref{:41a} and \eqref{:41b}, the H\"{o}lder inequality gives 
for each $ 0 \le t \le \smnT $ 
\begin{align}
\label{:41c}
|\Ztm - \hat{\mathbf{Z}}_{t }^{m} |^{\qqq } & \le 
\{ 
\int_0^t 
| \Zum - \ZumH | 
 \FF (\Zum , \hat{\XX }_u^{m*} ) du \}^{\qqq } 
\\ \notag &
\le \{ \int_0^{t } |\Zum - \ZumH |^{\qqq }du\}
\{ \int_0^{t } \FF (\Zum , \hat{\XX }_u^{m*} )^{\ppp } du \}^{\qqq / \ppp }
\\ \notag &
\le 
\cref{;41C} 
\int_0^{t } |\Zum - \ZumH |^{\qqq }du 
.\end{align}
Here, $ \qqq $ is the H\"{o}lder conjugate of $ \ppp $ and 
$ \Ct \label{;41C}= 
\{ \int_0^{\smnT } \FF (\Zum , \hat{\XX }_u^{m*} )^{\ppp } du \} ^{\qqq / \ppp } $. 
By \eqref{:31f}, we see that $ \cref{;41C} < \infty $ $ \pP $-a.s. 
Hence, from \eqref{:41c}, we can use Gronwall's lemma to obtain the identity 
$ \mathbf{Z}_t^m = \hat{\mathbf{Z}}_t^m$ until 
$ (\mathbf{Z}^m ,\hat{\XX }^{m*}) $ or 
$ (\ZmH ,\hat{\XX }^{m*}) $ exit from $ \HanC $. 
Then, for all $ 0 \le t \le \smnT $, 
\begin{align}& \label{:41d}
(\mathbf{Z}_t^m ,\hat{\mathbf{B}}_t^m,\hat{\XX }_t^{m*}) = 
(\hat{\mathbf{Z}}_t^m ,\hat{\mathbf{B}}_t^m,\hat{\XX }_t^{m*})
.\end{align}
Taking $ T \to \infty $, we see that \eqref{:41d} holds for all $ 0\le t \le \smn $. 
Because, for $ \pP $-a.s., $0 < \smn $ for sufficiently large $ \nn $ by \eqref{:41p}, 
this coincidence and the definition of $ \smn $ imply that, for $ \pP $-a.s., 
\begin{align}&\notag 
\varsigma _{\nn } (\mathbf{Z}^m ,\hat{\XX }^{m*}) = 
\varsigma _{\nn } (\ZmH ,\hat{\XX }^{m*}) = \smn 
\quad \text{ for sufficiently large $ \nn $}
.\end{align}
Combined with \eqref{:31e} and \eqref{:31m}, this yields, for $ \pP $-a.s., 
\begin{align}
\label{:41f}&
\limi{\nn } \smn = \infty 
.\end{align}
From \eqref{:41d} and \eqref{:41f}, we obtain \eqref{:31n}. 

Next, assume \Ass{B4}. Then the two Brownian motions in \eqref{:41t} are equipped with the same increasing families of 
$ \sigma $-fields such that $ \{ \mathcal{F}_t '\}=\{ \mathcal{F}_t '' \}$. 
Hence we obtain 
\begin{align}\notag &
\int_0^t \sigmam (\Zum ,\hat{\XX }_u^{m*}) d\BmHu - \int_0^t \sigmam (\ZumH ,\hat{\XX }_u^{m*}) d\BmHu 
=
\int_0^t \{ \sigmam (\Zum ,\hat{\XX }_u^{m*}) - \sigmam (\ZumH ,\hat{\XX }_u^{m*}) \} d\BmHu 
.\end{align}
Then by the martingale inequality, we have 
\begin{align}\notag & 
E [ \sup_{v \le t } |\int_0^{v \wedge \smnT }
\{ \sigmam (\Zum ,\hat{\XX }_u^{m*}) - 
 \sigmam (\ZumH ,\hat{\XX }_u^{m*}) \} d\BmHu |^2] 
\\ \notag \le &4
E [\langle \int_0^{\cdot }
\{ \sigmam (\Zum ,\hat{\XX }_u^{m*}) - 
 \sigmam (\ZumH ,\hat{\XX }_u^{m*}) \} d\BmHu \rangle _{t \wedge \smnT } ] 
\\ \notag 
= &4 
E[\int_0^{t \wedge \smnT }
\mathrm{tr}\big( \sigmam (\Zum ,\hat{\XX }_u^{m*}) - \sigmam (\ZumH ,\hat{\XX }_u^{m*})\big)\, 
^t \big( \sigmam (\Zum ,\hat{\XX }_u^{m*}) - \sigmam (\ZumH ,\hat{\XX }_u^{m*})\big)
du
]
.\end{align}
From \eqref{:31b} and \eqref{:31g}, the last line is dominated by 
\begin{align} \notag 
 &\cref{;41h} 
E[\int_0^{t \wedge \smnT }
| \Zum - \ZumH | ^2\FF (\Zum , \hat{\XX }_u^{m*} ) ^2 
du
] 
\quad \text{ by \eqref{:31b}}
\\ \notag 
\le &\cref{;41i} 
E[\int_0^{t \wedge \smnT } 
| \Zum - \ZumH |^2 
du
]
\quad \text{ by \eqref{:31g}}
\\ \notag 
\le &\cref{;41i} 
E[\int_0^{t \wedge \smnT } 
\sup_{u \le v }
| \Zum - \ZumH |^2 
dv 
]
.\end{align}
Here $ \Ct \label{;41h} $ and $ \Ct \label{;41i}$ are constants depending on $ d $, 
$ \nn \in \NNNthree $, and $ T \in \N $. Hence, we obtain 
\begin{align}\label{:41h}&
E [ \sup_{v \le t } |\int_0^{v \wedge \smnT }
\{ \sigmam (\Zum ,\hat{\XX }_u^{m*}) - 
 \sigmam (\ZumH ,\hat{\XX }_u^{m*}) \} d\BmHu |^2] 
\\ \notag &\quad \quad 
\le 
\cref{;41i} 
E[\int_0^{t \wedge \smnT } 
\sup_{u \le v }
| \Zum - \ZumH |^2 
dv 
]
.\end{align}

By \eqref{:31b} and \eqref{:31g} there exists a constant $ \Ct \label{;41I}$ 
depending on $ \nn \in \NNNthree $ and $ T \in \N $ such that 
\begin{align}\label{:41i}&
\sup_{v \le t } |\int_0^{v \wedge \smnT }
 b (\Zum ,\hat{\XX }_u^{m*}) - b (\ZumH ,\hat{\XX }_u^{m*}) du |^2 
\le \cref{;41I} 
\int_0^{t \wedge \smnT } 
\sup_{u \le v } | \Zum - \ZumH |^2 
dv 
.\end{align}
Let $ h (t) = E [\sup_{u \le t \wedge \smnT } | \Zum - \ZumH |^2 ]$. 
Then, by \eqref{:41t}, \eqref{:41h}, and \eqref{:41i} we have 
\begin{align}\notag &
h (t) \le 2(\cref{;41i}+\cref{;41I}) \int_0^t h (u) du
.\end{align}
Hence, by Gronwall's lemma we obtain $ h (t) = 0 $ for all $ t $. 
This implies \eqref{:31n}. 
\qed

Recall that $ (\mathbf{X},\mathbf{B})$ under $ \Ps $ 
is a weak solution of \eqref{:23a}--\eqref{:23b} starting at $ \mathbf{s}$. 
Thus, $ \XBXm $ becomes a weak solution of \eqref{:24f}--\eqref{:24h}. 

\smallskip 
\noindent {\em Proof of \tref{l:32}. } %
The proof of \tref{l:32} is the same as that of \cite[Proposition 11.1]{o-t.tail}. 
We explain the correspondence and omit the details of the proof. 

In \cite[Proposition 11.1]{o-t.tail}, \As{IFC} was deduced from the pathwise uniqueness of a weak solution. The pathwise uniqueness in \cite{o-t.tail} was given in \thetag{11.6} of Lemma 11.2 \thetag{3} in \cite{o-t.tail}. 
In the present paper, we deduce this pathwise uniqueness as \eqref{:31n} in \tref{l:31}. 
The assumptions in \tref{l:32} are the same as in \tref{l:31}, and they are used only to derived the conclusion of \tref{l:31}, that is, the pathwise uniqueness of weak solutions. 

The assumptions of \cite[Proposition 11.1]{o-t.tail} are different from those of \tref{l:32}. 
They were used only to guarantee the existence of weak solutions and the pathwise uniqueness of weak solutions in Lemma 11.2 \thetag{3} in \cite{o-t.tail}. Hence the proof of \cite[Proposition 11.1]{o-t.tail} 
is still valid for \tref{l:32}. 
\qed 

\noindent {\em Proof of \tref{l:33}.} 
For simplicity, we prove the case in which $ m=1$, $ \ellell = 2 $, and $ d = 1 $. 
The general case follows from the same argument. 

Let $ (x,\sss ) , (\xi ,\sss ) \in \HanC $ be such that $ (x,\sss ) \sim_{\nn } (\xi ,\sss ) $ and that $ x < \xi $. 
Then, from \Ass{C1} and $ d=1 $, we see $ [x,\xi]\ts \{ \sss \} \subset \Hann $. 
From the Taylor formula 
\begin{align}\label{:43e}
 b (x,\sss ) - b (\xi ,\sss ) = & 
\int_{\xi }^x \int_{\xi }^y \partial ^2 b (z,\sss ) dzdy + 
(x-\xi ) \partial b (\xi ,\sss ) 
.\end{align}
Let $\cref{;U2c}$ be the constant given by \eqref{:33m}. 
From \eqref{:43e} and \Ass{C2}, we have that 
\begin{align}\label{:43f}
| b (x,\sss ) - b (\xi ,\sss ) |
\le &
\cref{;U2c} (\nn ) 
\Big| \int_{\xi }^x \int_{\xi }^y dzdy \Big| 
 + |x-\xi || \partial b (\xi ,\sss ) |
\\ \notag \le & |x-\xi | \{ 
\cref{;U2c} (\nn ) \rr + | \partial b (\xi ,\sss ) | \} 
.\end{align}
Here, in the last line, we used $ \sup \{ |x-\xi |; x,\xi \in \Hann \} \le 2\rr \sqrt{m} = 2\rr $ because $ m=1 $. 
The same inequality holds for $ \sigma $. Hence, we take 
\begin{align} \label{:43g}&
\FF (x,\sss ) = \{2 \cref{;U2c} (\nn ) \rr + | \partial \sigma (x ,\sss ) | + | \partial b (x ,\sss ) | \} 
.\end{align}
We then immediately deduce \Ass{B2} from \eqref{:43f} and \eqref{:43g}. 

By applying the Taylor formula above to $ \partial \sigma (x ,\sss ) $ and $ \partial b (x ,\sss ) $, 
we obtain 
\begin{align}\label{:43h}&
 \sup \{ | \partial \sigma (x ,\sss ) | + | \partial b (x ,\sss ) | ; (x ,\sss ) \in \Han \} < \infty
.\end{align}
Then, \eqref{:31g} follows from \eqref{:43g} and \eqref{:43h}. 
It is clear that \eqref{:31f} follows from \eqref{:31g}. 
\qed

\section{A sufficient condition for \Ass{B1} in non-symmetric case} \label{s:5}

Throughout this section, $\XB $ is a weak solution of \eqref{:23a} and \eqref{:23b} 
defined on $ \OFPF $. 
We write $ \mathbf{X}=(X^i)_{ i \in\N }$ and $ \XM = (\mathbf{X} ^m,\XX ^{m*}) $. 

The purpose of this section is to present a sufficient condition for \Ass{B1}. 
Assumption \Ass{B1} implies the non-exit of the $ m $-labeled process 
$ \XM $ from $ \Ha $ given by \eqref{:30c}. 
By definition, $ \Ha $ is intersection of the set of the single configurations $ \Si $ and 
the tame set $ \Ka $. 
In \sref{s:5A}, we prove the non-exit of the unlabeled dynamics $ \XX $ from $ \Si $ in \pref{l:51}. 
In \sref{s:5B}, we prove the non-exit from $ \Ka $ in \pref{l:55}. 
The main results in the present section are Theorems \ref{l:58} and \ref{l:5X} given in \sref{s:53}.

\Ssection{Non-collision property } \label{s:5A}
Recall that $ \Ss $ is the subset of $ \sSS $ consisting of configurations with no multiple points. 
In this subsection, we derive a sufficient condition such that solutions move in the subset $ \Ss $. 
In other words, we pursue the condition under which particles do not collide with each other. 
In many examples, the drift coefficient $ b $ is of the form 
\begin{align}&\notag 
b (x,\sss ) = \frac{\beta}{2}\sum_{i}\nabla \Psi (x-s_i)
,\end{align}
where $ \sss =\sum_i \delta_{s_i}$, and $ \Psi (0) = \infty $. 
Hence, $ b (x,\sss ) $ is not well defined if $ \delta_x + \sss \not\in \Ss$. 
Thus, we need some criterion for the non-collision of particles.

We set $ \SR = \{x \in \sS ; |x| < R \} $ and $  \TRe =\{ (x,y)\in \SR ^2 ;\, |x-y|> \varepsilon \} $, where $ R \in \N $ 
and $ 0 \le \epsilon \le 1 $. 
Let $ \tauRe = \tauR ^{\epsilon ,i,j} $ be the exit time of $ (X^i,X^j)$ from $ \TRe $ such that 
\begin{align}&\label{:51q}
\tauRe = \inf \{ t >0 ; (X_t^i,X_t^j) \not\in \TRe \} 
.\end{align}
Let $  \map{\Upsilon }{[0,\infty )}{(0,\infty ]}$ be a positive, convex and decreasing function such that 
$ \Upsilon $ is smooth on $ (0,\infty )$ and $ \Upsilon (0) = \infty $. 
We set $ \upsilon (t) = - 1/ \Upsilon ' (t)$. Then $  \upsilon (0) = 0 $ and $ \upsilon $ 
is positive and increasing. Furthermore,  $ \upsilon $ is an Osgood-type function in the sense that 
\begin{align} \label{:51o}&
\int_0^1 \frac{1}{\upsilon (t) } dt = \infty 
.\end{align}
We set $ \vartheta ( x ) = x / |x| $ for $ x \in \Rd \backslash \{ 0 \} $. 
We make the following assumptions: 

\smallskip

\noindent \Ass{C3} For each $ R , i\not = j \in \N $,
\begin{align}
\label{:51u}&
E[ \Upsilon ( |X_0^i-X_0^j| )  ; (X_0^i,X_0^j) \in \SR \ts \SR ] < \infty 
.\end{align}

\noindent \Ass{C4} For each $ T , R , i\not = j \in \N $, 
\begin{align}\label{:51s}& 
\sup_{0 \le t \le T }
\sup_{0 < \epsilon \le 1 } 
E\Big[ 
 \Big| 
\int_0^{\tauRet }
 \Big( 
 \6 , 
b (X_u^i,\XX _u^{\idia })
 \Big) _{\Rd } 
du 
 \Big| 
\Big]< \infty 
.\end{align}

\noindent 
\Ass{C5} For each $ 0 \le t < \infty $ and $ i\not = j \in \N $, 
\begin{align} \label{:51t}& 
E[\int_0^{t} 
1_{\SR }(X_u^i) 1_{\SR }(X_u^j) 
\Big( \frac{ 1 }{\upsilon (|X_u^i-X_u^j|)^2 } + 
 \frac{ 1 + \upsilon ' (|X_u^i-X_u^j| ) } {\upsilon (|X_u^i-X_u^j|) \, |X_u^i-X_u^j| } \Big)
du ] < \infty 
.\end{align}

We note here $ \upsilon $ and $ \upsilon '$ are positive.

Let $ \sigma = \sigma (x,\sss )$ be the coefficient in $ \eqref{:23a} $. We set 
$ \map{a }{\SSsde ^{[1]} }{\mathbb{R}^{d^2}}$ such that 
\begin{align} \label{:34a}&
a = \sigma {}^t \sigma 
.\end{align}

\noindent 
\Ass{UB} 
$ a = (\akl (x,\sss ) )\klD $ is uniformly elliptic. 
Furthermore, $  a $ is bounded with upper bound $ \Ct \label{;34}$: 
\begin{align} \label{:34b}&
\sumklD \akl (x,\sss ) \xi _{\kK } \xi _{\lL } \le \cref{;34} |\xi | ^2 
\quad \text{ for all } \xi \in \Rd , (x,\sss ) \in \SSsde ^{[1]}
.\end{align}

\begin{proposition} \label{l:51}
Assume that \Ass{C3}--\Ass{C5} and \Ass{UB} hold. Then, 
\begin{align} \label{:51z} &
P (\XX _t \in \Ss \text{ for all } 0\le t < \infty )=1 
.\end{align}
\end{proposition}
\begin{proof}  
For \eqref{:51z}, it is sufficient to prove that, for each pair $ (i,j)$ such that $ i\not=j$, 
\begin{align} \label{:51a}&
P (X_t^i=X_t^j \text{ for some } 0 \le t < \infty ) = 0 
.\end{align}
We only prove \eqref{:51a} for $ (i,j)=(1,2)$, because the proof of the general case 
is similar. 

Let $ \varphi \in C_0^{\infty}((\Rd )^2)$ be such that $0\le \varphi (x,y) \le 1$, 
$ \varphi (x,y) = \varphi (y,x)$, and 
\begin{align}\notag &%
\varphi (x,y) = 
\begin{cases}
1 & (x,y) \in \TRe \\
0 & (x,y) \not\in \sS _{R+1}^2
.\end{cases}
\end{align}
Applying It$ \hat{\mathrm{o}}$'s formula to $ \varphi (x,y) \Upsilon ( |x-y| ) $ with $ (X^1,X^2)$ 
and noting that $ \varphi (x,y) = 1$ on the closure of $ \TRe $, 
we then have that, for each $ 0 < \epsilon \le 1 $ and $ R \in \N $, 
\begin{align} \label{:51d}
& \varphi (X_{\tauRet }^1 , X_{\tauRet }^2 ) 
 \Upsilon ( |X_{\tauRet }^1 - X_{\tauRet }^2| ) 
 = 
\varphi ( X_{0}^1,X_{0}^2 ) \Upsilon ( |X_{0}^1-X_{0}^2| ) 
\\ \notag & - 
 \sum_{\bullet } \int_0^{\tauRet } 
 \Big( \6 , \sigma (X_u^i,\XX _u^{\idia }) dB_u^i \Big) _{\Rd }
\\ \notag &  - 
 \sum_{\bullet } \int_0^{\tauRet } 
\Big(  \6 , b (X_u^i,\XX _u^{\idia })  \Big) _{\Rd } du 
\\ \notag & + 
 \sum_{\bullet } \int_0^{\tauRet } 
 \Big( \half a (X_u^i,\XX _u^{\idia })  \6  , \6  \Big) _{\Rd } du 
 \\  \notag & +
 \sum_{\bullet } \int_0^{\tauRet }
 \Big( \half a (X_u^i,\XX _u^{\idia })
\upsilon ' (|X_u^i-X_u^j|)
\6 ,
\frac{ \vartheta (X_u^i-X_u^j) }{ |X_u^i-X_u^j| }
  \Big) _{\Rd }
  du 
\\ \notag & - 
\sum_{\bullet } \int_0^{\tauRet } \frac{1}{ \upsilon (|X_u^i-X_u^j| ) \, |X_u^i-X_u^j| }
\sumkD  \half a_{\kK \kK } (X_u^i,\XX _u^{\idia }) du 
.\end{align}
Here, the sum $ \sum_{\bullet }$ is taken over $ (i,j)=(1,2), (2,1)$. 
We shall estimate each term of the right-hand side. 
Without loss of generality, we assume that $ (i,j)=(1,2)$ in the rest of the proof, 
and estimate the expectation of each term on the right-hand side of \eqref{:51d}. 

A direct calculation, together with \eqref{:34a}, yields 
\begin{align}
\label{:51f}&
 E[ |
\int_0^{\tauRet } 
 \Big( 
 \7 
\sigma (X_u^1,\XX _u^{\onedia }) 
, 
dB_u^1
 \Big)_{\Rd } 
 |^2]
\\ \notag =&
E[ \Big\langle 
\int_0^{\cdot } 
\Big( 
 \7  
\sigma (X_u^1,\XX _u^{\onedia }) 
, 
dB_u^1 \Big)_{\Rd } 
\Big\rangle_{\tauRet } ] 
\\ \notag = & E[ 
\int_0^{\tauRet } 
 \Big( a (X_u^1,\XX _u^{\onedia }) \7  , 
\7  \Big) _{\Rd } du ] 
.\end{align}
By \eqref{:34b} and \eqref{:51t}, we can see that for each $ 0 < \epsilon \le 1 $ and $ 0 \le t < \infty $ 
\begin{align}\label{:51g}&
 E[ \int_0^{\tauRet } 
 \Big( a (X_u^1,\XX _u^{\onedia }) \7  , 
\7  \Big) _{\Rd } du ]
\\ \notag \le & \ \cref{;34} E[ 
\int_0^{\tauRet } \frac{1}{\8 } du ] \quad \text{ by \eqref{:34b}}
\\ \notag 
 \le & \ \cref{;34} 
E[ \int_0^{t} \frac{1_{\SR }(X_u^1) 1_{\SR }(X_u^2) } {\8 }du ] 
< \infty \quad \text{ by \eqref{:51t}}
.\end{align}

Next, we prove the $ L^1$-boundedness of each term of \eqref{:51d} in 
$ 0\le t \le \tauReT $ and $ 0 < \epsilon \le 1 $ for each $ \TRN $. 
By \eqref{:51u}, the first term on the right-hand side of \eqref{:51d} is in $ L^1 $. 
By \eqref{:51f} and \eqref{:51g}, the second term in \eqref{:51d} turns to be $ L^2 $-martingale. 
Thus, these terms are uniformly integrable. 
By \eqref{:51s}, the third term on the right-hand side of \eqref{:51d} is $ L^1$-bounded. 
From \eqref{:51t}--\eqref{:34b},  we see that the fourth, fifth, and sixth terms on the right-hand side are $ L^1 $-bounded. 

%

Collecting these, we have that all the terms on the right-hand side are $ L^1$-bounded. 
Thus, we deduce that the left-hand side of \eqref{:51d} is $ L^1$-bounded in 
$ 0\le t \le \tauReT $ and $ 0 < \epsilon \le 1 $ for each $ \TRN $, that is, 
\begin{align}& \label{:51l}
\sup_{ 0\le t \le \tauReT ,\, 0 < \epsilon \le 1 }
 E[ \varphi (X_{\tauRet }^1 , X_{\tauRet }^2 ) 
\Upsilon ( \big|X_{t\wedge \tauRe }^1 - X_{t\wedge \tauRe }^2 \big| ) ] < \infty
.\end{align}

We see that $ \tauRz =\limz{\epsilon }\tauRe $ because $ \{ (x,x)\in \sS ^2 \} $ is a closed set. 
Then, taking $ t \to T $ and then $ \epsilon \to 0 $, 
we have from Fatou's lemma and \eqref{:51l} that for each $ \TRN $ 
\begin{align} \label{:51m}
 E \Big[&
\varphi (X_{T\wedge \tauRz }^1 , X_{T\wedge \tauRz }^2 ) 
\Upsilon ( \big|X_{T\wedge \tauRz }^1 - X_{T\wedge \tauRz }^2 \big| ) \Big] 
\\ \notag & \le 
\liminfz{\epsilon }\lim_{ t \to T }
E \Big[  \varphi (X_{\tauRet }^1 , X_{\tauRet }^2 ) 
\Upsilon ( \big|X_{t\wedge \tauRe }^1 - X_{t\wedge \tauRe }^2 \big| ) \Big] 
< \, \infty 
.\end{align}
Let $ \tauR $ be the exit time of $ (X^1,X^2)$ from $ \SR ^2 $. 
Then we deduce $ T\wedge \tauRz = T\wedge \tauR $ a.s.\,for all $ T , R \in \N $ from \eqref{:51m}. 
Hence, $ \tauRz = \tauR $ a.s.\,for all $ R \in \N $. 

By assumption, each tagged particle $ X^i $ of $ \mathbf{X}=(X^i)_{i\in\N }$ does not explode. 
Hence, $ \limi{R} \tauR = \infty $ a.s. 
Together with $ \tauRz = \tauR $ a.s.\, for all $ R \in \N $, this implies 
\begin{align}\label{:51p}&
\limi{R} \tauRz = \infty 
.\end{align}
Let $ \tauz = \inf \{t> 0; (X_t^1,X_t^2) \in \{ x=y \} \} $ be the first hitting time of $ (X^1,X^2)$ to the set $ \{ x=y \} \subset \sS ^2$. 
Then $ \tauz = \limi{R} \tauRz $. Hence \eqref{:51p} implies $ \tauz = \infty $ a.s. 
Therefore, we deduce that $ X^1 $ and $ X^2 $ do not collide with each other. 
\end{proof}

\Ssection{Non-exit from $ \Ka $. }\label{s:5B}

Let $ \Ka $ and $ \Kak $ be the sets given by \eqref{:30w}. 
Let $ \kappaq $ be the exit time of $ \XX $ from $ \Kak $, that is, 
$ \kappaq = \inf \{ t>0; \XX _t \notin \Kak \} $. We set $ \kappai := \limi{\qq } \kappaq $. 

In \sref{s:5B}, we shall prove non-exit of $ \XX $ from $ \Ka $ 
in such a way that 
\begin{align}
\label{:55z}&
\pP (\kappai = \infty )= 1
.\end{align}
The strategy of the proof is to reduce the problem to the construction of a specific function 
$ \chiwtI $ on $ \sSS $ in \eqref{:54p} 
that diverges on $ \Ka ^c$ and satisfies $ E [|\chiwtI ( \XX _{t\wedge \kappai } )|] < \infty $. 

For $ \qQ \in \N \cup \{ \infty \} $ 
let $ \KaQ = \bigcup _{\qq = 1}^{\infty} \KaQk $, where $ \KaQk $ is such that 
\begin{align} \notag 
\KaQk &=\{ \sss \in \sSS \, ;\, \sss (\SrR ) \le \ak (\rR ) \text{ for all }\rR \le \qQ \} \quad \text{ for } \qQ < \infty 
,\\\notag &
= \{ \sss \in \sSS \, ;\, \sss (\SrR ) \le \ak (\rR ) \text{ for all }\rR < \infty \} \quad \text{ for } \qQ = \infty 
.\end{align}
Then $ \Ka = \KaQ $ and $ \Kak = \KaQk $ for $ \qQ = \infty $. 

Recall that $ \mathbf{a}=\{ \ak \}_{\qq \in\N } $ is a sequence 
of increasing sequences $ \ak = \{ \ak ( \rR ) \}_{\rR \in\mathbb{N}} $ and that 
$ \ak ^+ = \{1 + \ak (\rR +1)\}_{\rR =1}^{\infty}$ for 
$ \ak =\{ \ak (\rR ) \}_{\rR =1}^{\infty} $. 
Both $ \{ \KaQk \}_{\qq =1} ^{\infty}$ and $ \{ \KQ [\ak ^+ ] \}_{\qq =1} ^{\infty}$ 
are increasing sequences of compact sets in $ \sSS $ if and only if $ \qQ = \infty $. 
In addition to \eqref{:30w}, assume that, for all $ \qQ \in \N \cup \{ \infty \} $, 
\begin{align} & \label{:52m} 
 \KaQk \subset \KQ [\ak ^+ ] \subset \KQ [\akk ] 
.\end{align}
Note that $ \KaQk \subset \KQ [\ak ^+ ] $ is clear because $ \ak < \ak ^+ $. 
Suppose 
$ \ak ( \rR ) = C(\qq ) \rR ^{\alpha }$ for some $ \alpha > 0 $ and an increasing function $ C (\qq )$ 
with $ C (\qq ) \to \infty $. Then, taking a new sequence from 
$ \mathbf{a}=\{ \ak \}_{\qq \in\mathbb{N}} $ more intermittently, 
we can easily retake $ \mathbf{a}$ such that $ \ak ^+ (\rR ) < \akk (\rR )$ for all $ \rR \in \N $. 
For such $ \mathbf{a} $ we obtain $ \KQ [\ak ^+ ] \subset \KQ [\akk ]$. 
Then \eqref{:52m} holds.

We set for $ \qQ \in \N \cup \{ \infty \} $
\begin{align}
\label{:52p}&
\LL [\ak ] = \KQ [\ak ^+ ] \backslash \KQ [\ak ] 
.\end{align}
Then we have from \eqref{:52m} 
\begin{align} \label{:52q}&
\LL [\ak ] \cap \LL [{a}_{\rr } ] = \emptyset \quad \text{ for each }\qq \not= \rr \in \N 
.\end{align}

We next generalize $ \check{f}$ given by \eqref{:21h} to non-local functions $ f $. 

Let $ \SST $ be the set of all countable sums of point measures on $ \sS $ including the zero measure. 
Let $ \mathbb{\sS } = \{ \bigcup_{\qq =0}^{\infty} \sS ^\qq \} \bigcup \SN $ as before. 
For a function $ f $ defined on $ \SST $, there exists a unique function $ \check{f}$ 
defined on $ \mathbb{\sS }$ such that $ \check{f}|_{\sS ^m} $ is symmetric in 
$ \mathbf{s} = (s_i)_{i=1}^m$ and that 
$ \check{f} ( \mathbf{s} ) = f (\ulab (\mathbf{s} )) $, 
where $ m\in \N \cup \{ \infty \} $ and $ \sS ^{\infty }= \SN $.

By convention, $ \sS ^0 = \{ \emptyset \} $ denotes the set consisting of the empty set and 
$ \check{f}|_{\sS ^0}$ is a constant. 
For a function $ f $ on $ \sSS $, we define a function $f_{\bullet} $ on $ \SST $ 
by taking $ f_{\bullet} (\sss ) = 0 $ for $ \sss \in \SST \backslash \sSS $. 
Then we take $ \check{f} $ for $ f $ as the restriction of $ \check{f_{\bullet}}$ on $ \ulab ^{-1}(\sSS ) $. 
The relation between $ \check{f}$ and $ \check{f}_{\SR }$ given by \eqref{:21h} for 
a $ \sigma [\SR ]$-measurable local function $ f $ is, if $ x_1,\ldots,x_m \in \SR $ and 
$ x_j \not\in \SR $ for $ j > m$, 
\begin{align}\label{:52u}
\check{f} (x_1,\ldots,x_m,x_{m+1},,\ldots ) = \check{f}_{\SR }(x_1,\ldots,x_m)
.\end{align}

Let $ \SRm = \SR \ts \cdots \ts \SR $ be the $ m$-product of $ \SR $. 
Let $ \SSRm = \{ \sss \in \sSS \, ;\, \sss (\SR ) = m \} $ for $ \rR , m \in\N $. 
We set maps $ \map{\piR , \piRc }{\sSS }{\sSS }$ such that $ \piR = \pi _{\SR }$ and $ \piRc = \pi _{\SR ^c}$. 
For $\sss\in \SSRm $, we call 
$\mathbf{x}_{\rR } ^m(\sss)=(x_{\rR } ^{i}(\sss))_{i=1}^m \in \SRm $ 
an $\SRm $-coordinate of $\sss$ if $ \piR (\sss)=\sum_{i=1}^m \delta_{x_{\rR } ^{i}(\sss)}$. 

For a function $f: \sSS \to \R$ and $\rR , m \in\N $, we define 
an $ \SRm $-representation $ \{ f_{\rR ,\sss}^m \}_{\sss }$ of $ f $ 
using an $\SRm $-coordinate $\mathbf{x}_{\rR } ^m(\sss)$ of $\sss$. 

\begin{definition}\label{d:52} 
We call 
$ \{ f_{\rR ,\sss}^m \}_{\sss } $ 
an $ \SRm $-representation of $ f $ if \thetag{1}--\thetag{4} hold. 

\noindent 
\thetag{1} $f_{\rR ,\sss}^m$ is a permutation invariant function on $\SRm $ for each $\sss\in \SSRm $.
\\
\thetag{2} $f_{\rR ,\sss(1)}^m =f_{\rR ,\sss(2)}^m $ if $\piRc (\sss(1))=\piRc (\sss(2))$ for $\sss(1), \sss(2)\in \SSRm $.
\\
\thetag{3} $f_{\rR ,\sss}^m(\mathbf{x}_{\rR } ^m(\sss))=f(\sss)$ for $\sss\in \SSRm $. 
\\
\thetag{4} $f_{\rR ,\sss}^m(\mathbf{x}_{\rR } ^m(\sss))=0$ for $\sss\notin \SSRm $.
\end{definition}
By definition, we have a relation among $ \check{f}$, 
 $\mathbf{x}_{\rR } ^m $, and $ f_{\rR ,\sss}^m$ such that 
$$ \check{f} (\mathbf{x}_{\rR } ^m(\sss), \mathbf{s} ) = f_{\rR ,\sss}^m (\mathbf{x}_{\rR } ^m(\sss))
\quad \text{ for } \sss \in \SSRm 
.$$
We say that a function $ f $ on $ \sSS $ is of $ C^k $-class if its $ \SRm $-representation $ f_{\rR ,\sss}^m $ 
is in $ C^k (\SRm )$ for each $ \rR , m \in \N $ and $ \sss \in \sSS $. 
Let $ C^k (\sSS ) $ be the set consisting of the functions of $ C^k $-class. 
We set $ C^{\infty} (\sSS ) = \cap_{k=0}^{\infty} C^k (\sSS )$. 
Note that a function $ f $ on $ \sSS $ of $ C^k $-class is not necessary continuous on $ \sSS $ 
because we equip $ \sSS $ with the vague topology.

Let $ a $ be given by \eqref{:34a} and $ \DDDa $ be the carr\'{e} du champ operator such that 
\begin{align}
\label{:52a}&
\DDDa [f,g] (\sss ) = 
\half \sum_{i} (a (s_i , \sss ^{\idia }) 
\PD{\check{f}}{s_i}, \PD{\check{g}}{s_i} )_{\Rd }
.\end{align}
By \eqref{:52u}, we easily see that $ \DDDa [f,f] $ does not depend on the choice of 
$ \check{f} $ or $ \check{f}_{\SR } $ for a $ \sigma [\piR ]$-measurable function $ f $.

Next, we introduce a family of cut-off functions $ \{ \chiqQ \}_{\qq \in\N } $. 

We take a label $ \lab = (\labi )$ such that $ |\labi (\sss )| \le |\lab ^{i+1} (\sss )|$ for all $ i $. 
We set for $ \qQ \in \N \cup \{ \infty \} $ 
\begin{align}\label{:52f}& 
\dkQ (\sss ) = 
\Big\{ \sum_{\rR =1}^{\qQ } \sum_{i\in \Jrs (\ak ) }
(\rR - |\labi (\sss )|)^2 \Big\}^{1/2} 
,\end{align}
where 
$ \Jrs (\ak ) = \{ i \, ;\, i > \ak (\rR ) ,\, \labi (\sss ) \in \SR \} $. 
Let $ \theta \in C^{\infty}(\R )$ such that $ 0 \le \theta (t) \le 1 $ for all $ t \in \R $, 
$ \theta (t) = 0 $ for $ t \le \epsilon $, and 
$ \theta (t) = 1 $ for $ t \ge 1 - \epsilon $ for a sufficiently small $ \epsilon > 0 $. 
Furthermore, we assume that $ |\theta ' (t)| \le \sqrt{2} $ for all $ t $. 
Let 
\begin{align} \label{:52g}& 
\chiqQ (\sss )
= \theta \circ \dkQ (\sss )
.\end{align}
\begin{lemma} \label{l:52}
\thetag{1} For each $\qq \in \N $ and $ \qQ \in \N \cup \{ \infty \} $, $ \chiqQ \in C^{\infty} (\sSS )$. 
\\\thetag{2} 
Assume \eqref{:52m}. Then, $ \chiqQ $ satisfies the following: 
\begin{align}\label{:52b}
&
0 \le \chiqQ \le 1 
,
&
\chiqQ (\sss ) = 
\begin{cases}
0 & \text{ for } \sss \in \KaQk \\
1 & \text{ for } \sss \not\in \KQ [\ak ^+] 
,\end{cases}
\\
\label{:52c}&
0\le \DDDa [\chiqQ ,\chiqQ ] \le \cref{;34}, 
&
\DDDa [\chiqQ ,\chiqQ ] = 0 \text{ for } \sss \not\in \LL [\ak ] 
.\end{align}
Here $ \KaQk $, $ \KQ [\ak ^+] $, and $ \LL [\ak ]$ are the same as in \eqref{:52p}, and 
$ \cref{;34}$ is given by \eqref{:34b}. 
\end{lemma}

\begin{proof}  
A direct calculation shows 
$ \chiqQ \in C^{\infty} (\sSS )$, \eqref{:52b}, and the equality in \eqref{:52c}. 

Clearly, $ 0 \le \DDDa [\chiqQ ,\chiqQ ] (\sss )$. 
A straightforward calculation shows that by \eqref{:34b} 
\begin{align} \notag 
\DDDa [\chiqQ ,\chiqQ ] (\sss ) \le 
& \frac{\cref{;34}}{2}
\Big\{ \frac{\theta' (\dkQ (\sss ))}
{\dkQ (\sss )} \Big\}^2 
 \sum_{\rR =1}^{\qQ } \sum_{i\in \Jrs (\ak ) } ( \rR -|\labi (\sss )|)^2
\\ \notag = & 
\frac{\cref{;34}}{2}
\theta' (\dkQ (\sss ))^2 \le \cref{;34} 
.\end{align}
Hence, we see that $ \chiqQ $ satisfies the inequalities in \eqref{:52c}. 
\end{proof}

For $ \NN \in \N $ and $ \qQ \in \N \cup \{ \infty \} $, we set 
\begin{align}\label{:53x}&
\chiwtN = \sum_{\qq =1}^{\NN } \chiqQ 
.\end{align}
We regard $ \chiwtN $ as a {\em coordinate} of $ \sss $ from the viewpoint of $ \KQ [\ak ] $. 
\begin{lemma} \label{l:53} 
\thetag{1} If $ \NN $ and $ \qQ \in \N $, then $ \chiwtN $ is bounded and continuous on $ \sSS $. 
\\\thetag{2} 
If $ \NN \in \N $ and $ \qQ \in \N \cup \{ \infty \} $, then $ \chiwtN \in C^{\infty} (\sSS ) $. 
Furthermore, the following hold. 
\begin{align}\label{:53a}&
\begin{cases}
\qq -1 \le \chiwtN (\sss ) \le \qq &\text{ for } \sss \in \KQ [\ak ^+] \backslash \KQ [\ak ] , \ \qq \le \NN 
,\\
\chiwtN (\sss ) = \qq &\text{ for } \sss \in \KQ [\akk ] \backslash \KQ [\ak ^+] 
, \ \qq \le \NN 
,\\
\chiwtN (\sss ) = \NN &\text{ for } \sss \in \KQ [a _{\NN +1} ]^c 
,\end{cases}
\\\label{:53b}&
 \DDDa [\chiwtN ,\chiwtN ] (\sss )
\begin{cases}\le \cref{;34} 
&\text{ for } \sss \in \KQ [a _{\NN +1} ] 
,\\=0
&\text{ for } \sss \in \KQ [a _{\NN +1} ]^c 
.\end{cases}
\end{align}
\end{lemma}
\begin{proof}  
\thetag{1} is clear by \eqref{:52f}, \eqref{:52g}, and \eqref{:53x}. 
\eqref{:53a} follows from \eqref{:52m}, \eqref{:52b}, and \eqref{:53x}. 
The equality in \eqref{:53b} follows from \eqref{:53a}. 

We finally prove the inequality in \eqref{:53b}. By \eqref{:53x} 
\begin{align}\label{:53d}&
\DDDa [\chiwtN ,\chiwtN ] (\sss ) 
= \DDDa [\sum_{\qq =1}^{\NN } \chiqQ ,\sum_{\qq =1}^{\NN } \chiqQ ] 
= \sum_{\qq ,\rr = 1}^{\NN } \DDDa [ \chiqQ , \chirQ ] 
.\end{align}
From the Schwarz inequality, \eqref{:52q}, and \eqref{:52c}, we have for $ \qq \not= \rr $ 
\begin{align}\label{:53e}&
 \DDDa [\chiqQ , \chirQ ] ^2 \le 
\DDDa [\chiqQ ,\chiqQ ] \DDDa [\chirQ , \chirQ ] 
= 0 
.\end{align}
From \eqref{:52q} and \eqref{:52c}, we see for $ \sss \in \KaQ $ 
\begin{align}
\label{:53f}&
\sum_{\qq =1}^{\NN } \DDDa [ \chiqQ , \chiqQ ] 
= \sum_{\qq =1}^{\NN } 1_{\LL [\ak ]} \DDDa [ \chiqQ , \chiqQ ] 
\le \cref{;34}
.\end{align}
From \eqref{:53d}, \eqref{:53e}, and \eqref{:53f} we obtain the inequality in \eqref{:53b}. 
\end{proof} 
Let $ \chic ^{\NN }(\mathbf{s}) = \chic ^{\NN }((s_1,s_2,\ldots))$ be the symmetric function on $ \SN $ such that $ \chic ^{\NN }(\mathbf{s}) = \chiwtN (\ulab (\mathbf{s}))$. 
Recall that $ \mathbf{X}_t = (X_t^i)_{i\in\N }$ and 
 $ \ulab (\mathbf{X}_t)= \sum_{i=1}^{\infty} \delta_{X_t^i}= \XX _t $. 
Hence, we have 
\begin{align}& \label{:54x}
 \chic ^{\NN }(\mathbf{X} _t) = \chiwtN (\ulab (\mathbf{X} _t))=\chiwtN (\XX _t)
.\end{align}
We regard $ \chic ^{\NN }$ as a smooth function on $ \SN \cap \{ \chic ^{\NN }< \infty \} $. 
Let $ \partial _i =(\partial _{i,\kK })\kD $ and set 
\begin{align}\label{:54y}&
 \chiwtp (x, \mathfrak{y})= (\partial_{1,\kK } \chic ^{\NN }) (x,\mathbf{y} ) ,\quad 
 \chiwtpq (x, \mathfrak{y})= (\partial_{1,\kK }\partial_{1,\lL } \chic ^{\NN }) (x,\mathbf{y}) 
.\end{align}
Here $ x = (x_1,\ldots,x_d) \in \Rd $ and $ \mathfrak{y} = \ulab (\mathbf{y})$. 

Assume $  \qQ < \infty $. Then, $ \chiwtN $ is a local function. 
For $ j \in \N $ we set 
\begin{align}& \label{:54A}
\1 = 
\sum_{i=1}^{j}
 \int_0^t 
\sumklD 
 \chiwtp (X_u^i,\XX _u^{\idia }) \sigmakl (X_u^i,\XX _u^{\idia }) dB_u^{\il }
.\end{align}
By \eqref{:34a}, \eqref{:52a}, \eqref{:53a}, and \eqref{:53b}, $ \2 $ is a continuous $ L^2$-martingale and 
\begin{align}\label{:54j}
& 
\langle \2 \rangle_t = \sum_{i=1}^{j}
\langle \int_0^t
\sumklD 
 \chiwtp (X_u^i,\XX _u^{\idia }) \sigmakl (X_u^i,\XX _u^{\idia }) dB_u^{\il }
 \rangle_t \le 2 \cref{;34} t 
.\end{align}
Hence $ \{ \2 \}_{j\in\N} $ is a Cauchy sequence in the space of continuous $ L^2$-martingales 
and converges to the continuous $ L^2$-martingale $ \4 $. We easily see from \eqref{:54j} that 
\begin{align} \label{:54C}
 \langle \4 \rangle_t &= 
 \sum_{i=1}^{\infty}
\langle \int_0^t
\sumklD 
 \chiwtp (X_u^i,\XX _u^{\idia }) \sigmakl (X_u^i,\XX _u^{\idia }) dB_u^{\il }
 \rangle_t 
\\ \notag & = 
\langle 
 \sum_{i=1}^{\infty}
\int_0^t
\sumklD 
 \chiwtp (X_u^i,\XX _u^{\idia }) \sigmakl (X_u^i,\XX _u^{\idia }) dB_u^{\il }
 \rangle_t 
.\end{align}
From \eqref{:54A}--\eqref{:54C}, we obtain 
\begin{align} \label{:54b}&
E[ | \3 |^2 ] = E[ \langle \4 \rangle_t ] \le 2\cref{;34} t 
.\end{align}

Applying It$\hat{\mathrm{o}}$'s formula to $ \mathbf{X} $ and $ \chic ^{\NN } $ 
together with \eqref{:54x} and \eqref{:54y}, we deduce that 
$ \chiwt ^{\NN } (\XX _t) $ is a continuous semi-martingale such that 
\begin{align} \label{:54Z} 
\chiwt ^{\NN } (\XX _t) = \chiwt ^{\NN } (\XX _0) & +
\0 
\sumklD 
 \chiwtp (X_u^i,\XX _u^{\idia }) \sigmakl (X_u^i,\XX _u^{\idia }) dB_u^{\il } 
\\ \notag &
+
\int_0^t \sum_{i=1}^{\infty} \sumkD 
b_{\kK } (X_u^i,\XX _u^{\idia }) 
\chiwtp (X_u^i,\XX _u^{\idia }) 
du
\\ \notag &
+ \half 
\int_0^t \sum_{i=1}^{\infty} \sumklD 
 \akl (X_u^i,\XX _u^{\idia }) 
\chiwtpq (X_u^i,\XX _u^{\idia }) du 
.\end{align}
Here, 
$ \sigma = (\sigmakl )\klD $, $ B^i=(B^{i,\kK })\kD $, 
$ b = (b_{\kK } )\kD $, and $ a = (\akl )\klD $.

By construction, for each $ \sss $, $ \chiwtN (\sss )$ is increasing 
in $ \qQ $ for each $ \NN \in \N \cup \{ \infty \} $, and 
in $ \NN $ for each $ \qQ \in \N \cup \{ \infty \} $. 
Hence we set 
\begin{align}& \label{:54p}
 \chiwtI (\sss ) := \limi{\NN } \limi{\qQ }\chiwtN (\sss )
.\end{align}
Then we have 
\begin{align}\label{:54z1}&
\chiwtI (\XX _t) = \limi{\NN } \limi{\qQ } \chiwtN (\XX _t) 
.\end{align}
From \eqref{:53x} and \lref{l:53}, 
we see $ \chiwtI (\sss ) < \infty $ if and only if 
$ \sss \in \Ka $. 
Hence $ \chiwtI (\XX _t) < \infty $ if and only if $ \XX _t \in \Ka $. 
So our task is to prove $ \chiwtI (\XX _t) < \infty $ for all $ t $ a.s. 

\begin{lemma} \label{l:54y} 
Assume \eqref{:52m}. Assume that 
\begin{align}\label{:55h}&
\sum_{\qq =1}^{\infty} \qq ^2 \pP ( \XX _0 \in \K [\ak ]^c ) < \infty 
.\end{align}
Then $ \chiwtI (\XX _0) < \infty $ a.s.\,and 
\begin{align}\label{:55k}&
E [ \chiwtI (\XX _0)^2 ] < \infty 
.\end{align}
\end{lemma}
\begin{proof}  
From \eqref{:55h}, we see $ \pP (\XX _0 \in \cap_{\qq =1}^{\infty} \{ \K [\ak ]^c \} ) = 0 $. 
Then $ \pP ( \XX _0 \in \Ka ^c ) = 0 $. 
Combining this with \eqref{:53a}, \eqref{:54p}, and \eqref{:55h}, we obtain 
\begin{align}\notag 
 E [\chiwtI (\XX _0 ) ^2 ] &= 
 E [ 1_{\Ka }(\XX _0 ) \big| \chiwtI (\XX _0 ) \big|^2 ] = 
\limi{\NN }\limi{\qQ } 
E [ 1_{\Ka }(\XX _0 ) \big| \chiwtN (\XX _0 ) \big|^2 ] 
\\ \notag &
\le 
 \sum_{\qq =1}^{\infty} \qq ^2 \, \pP ( \XX _0 \in \K [\akk ] \backslash \K [\ak ] )
 \le 
 \sum_{\qq =1}^{\infty} \qq ^2 \, \pP (\XX _0 \in \K [\ak ] ^c ) 
< \infty 
.\end{align}
This yields \eqref{:55k}. The first claim is clear from \eqref{:55k}. 
\end{proof}

\begin{lemma} \label{l:54z} 
Assume \eqref{:52m}. 
Assume $ \chiwtI (\XX _0) < \infty $. Assume that for each $ t $ 
\begin{align}\label{:54z2}
&
\limi{\NN }\limi{\qQ }
\int_0^t 
1_{ \Ka } (\XX _u ) 
\sum_{i=1}^{\infty} \sumkD 
b_{\kK } (X_u^i,\XX _u^{\idia }) 
\chiwtp (X_u^i,\XX _u^{\idia }) 
du
\\ \notag & \quad \quad \quad = 
\int_0^t 
1_{ \Ka } (\XX _u )
\sum_{i=1}^{\infty} \sumkD 
b_{\kK } (X_u^i,\XX _u^{\idia }) 
\chiwtIp (X_u^i,\XX _u^{\idia }) 
du 
\quad \text{ a.s.}
,\\ \notag &
\limi{\NN }\limi{\qQ } 
\half 
\int_0^t 
1_{ \Ka } (\XX _u )
\sum_{i=1}^{\infty} \sumklD 
 \akl (X_u^i,\XX _u^{\idia }) 
\chiwtpq (X_u^i,\XX _u^{\idia }) du 
\\ \notag &
\quad \quad \quad = 
\half 
\int_0^t 
1_{ \Ka } (\XX _u )
\sum_{i=1}^{\infty} \sumklD 
 \akl (X_u^i,\XX _u^{\idia }) 
\chiwtIpq (X_u^i,\XX _u^{\idia }) du 
\quad \text{ a.s.}
,\end{align}
and that the right-hand sides of the equations in \eqref{:54z2} are continuous processes and finite for all t. 
Then $ \chiwtI (\XX _t) $ is finite for all $ t $ and a continuous semi-martingale such that 
\begin{align} \label{:54z} 
\chiwtI (\XX _t) = \chiwtI (\XX _0) & 
+ 
\0
1_{ \Ka } (\XX _u )
\sumklD 
 \chiwtIp (X_u^i,\XX _u^{\idia }) \sigmakl 
(X_u^i,\XX _u^{\idia }) dB_u^{\il } 
\\ \notag &
+
\int_0^t 
1_{ \Ka } (\XX _u )
 \sum_{i=1}^{\infty} \sumkD 
b_{\kK } (X_u^i,\XX _u^{\idia }) 
\chiwtIp (X_u^i,\XX _u^{\idia }) 
du
\\ \notag &
+ \half 
\int_0^t 
1_{ \Ka } (\XX _u )
\sum_{i=1}^{\infty} \sumklD 
 \akl (X_u^i,\XX _u^{\idia }) 
\chiwtIpq (X_u^i,\XX _u^{\idia }) du
.\end{align}
\end{lemma}
\begin{proof} 
By \eqref{:52f}, \eqref{:52g}, and \eqref{:54z1}, we easily see 
$ \DDDa [\chiwtN ,\chiwtN ] (\sss )$ are increasing in $ \qQ $ for each $ \NN $ and also 
$ \DDDa [\chiwtI _{\infty }^{\NN } ,\chiwtI _{\infty }^{\NN } ] (\sss )$ are increasing in $ \NN $. 
Furthermore, 
\begin{align}\label{:54bb}
1_{ \Ka } (\sss )
\DDDa [\chiwtI ,\chiwtI ] (\sss ) 
&= 
1_{ \Ka } (\sss )
\limi{\NN } \limi{\qQ } \DDDa [\chiwtN ,\chiwtN ] (\sss ) 
\\\notag 
&= 
\limi{\NN } \limi{\qQ } \DDDa [\chiwtN ,\chiwtN ] (\sss ) 
\le \cref{;34}
\end{align}
and $ \limi{\NN } \limi{\qQ } 1_{\Ka } (\sss ) \DDDa [ \chiwtI - \chiwtN ,\chiwtI - \chiwtN ] (\sss ) = 0 $ 
for each $ \sss $. 

From \eqref{:52a} and \eqref{:54bb}, we deduce that 
the second term of the right-hand side of \eqref{:54z} is a continuous $ L^2$-martingale and 
is the limit of the third term of \eqref{:54Z} in the space of the continuous $ L^2$-martingales 
on $ \OFPF $. 

By \eqref{:52f}, \eqref{:52g}, \eqref{:52b}, and \eqref{:53x}, 
we see for $ (x , \sss ) \in \Rd \ts \sSS $ such that $ \delta_x + \sss \in \Ka ^c $ 
\begin{align}\label{:54B}&
\chiwtp (x , \sss ) = \chiwtpq (x , \sss ) = 0 
.\end{align}

Take $ \qQ \to \infty $ and then $ \NN \to \infty $ in \eqref{:54Z}. 
Then we obtain \eqref{:54z} from \eqref{:54b}, \eqref{:54z1}, \eqref{:54z2}, \eqref{:54bb}, and \eqref{:54B}. 
Each term of the right-hand side of \eqref{:54z} is finite and continuous in $ t $ by assumption and 
the argument as above. 
Hence, $ \chiwtI (\XX _t) < \infty $ for all $ t $ and $ \chiwtI (\XX _t) $ is a continuous semi-martingale 
satisfying \eqref{:54z}. 
\end{proof} 
\begin{lemma} \label{l:54} 
Let $ \kappaq $ be the exit time of $ \XX $ from $ \K [\ak ] $. 
For each $ t \ge 0 $,
\begin{align}\label{:54a}&
\sup_{\qq \in \N } 
E[\Big| 
 \sum_{i=1}^{\infty} 
\int_0^{t\wedge \kappaq }
\sumklD 
\chiwtIp (X_u^i,\XX _u^{\idia }) \sigmakl 
(X_u^i,\XX _u^{\idia }) dB_u^{\il } \Big|^2 ] < \infty 
.\end{align}
\end{lemma}
\begin{proof} 
We deduce \eqref{:54a} from \eqref{:54b} easily. 
\end{proof}

\begin{proposition} \label{l:55}
Assume \eqref{:52m}, \eqref{:55k}, and \eqref{:54z2}. Assume that 
\begin{align}\label{:55i}&
\sup_{\qq \in \N } \Big| E [ \int_0^{t\wedge \kappaq } 
\sum_{i=1}^{\infty} \sumkD 
b_{\kK } (X_u^i,\XX _u^{\idia }) 
\chiwtIp (X_u^i,\XX _u^{\idia }) du 
] \Big| < \infty 
,\\ & \notag 
\sup_{\qq \in \N } \Big| E [ \int_0^{t\wedge \kappaq } 
\sum_{i=1}^{\infty} \sumklD 
 \akl (X_u^i,\XX _u^{\idia }) 
\chiwtIpq (X_u^i,\XX _u^{\idia }) du
] \Big| < \infty 
.\end{align}
Then, we obtain \eqref{:55z}. 
\end{proposition}
\begin{proof}  
Note that $ \chiwtN $ are non-negative and continuous for all $ \NN , \qQ \in \N $. 
Then, by \eqref{:54z1}, the monotone convergence theorem (MCT), and Fatou's lemma, 
we obtain for each $ t $ 
\begin{align}
\label{:55a}
E [\chiwtI (\XX _{t\wedge \kappai })] & = 
\limi{\NN }
\limi{\qQ }
E [\chiwtN (\XX _{t\wedge \kappai })] 
&&\text{ by MCT}
\\ \notag &
\le 
\limi{\NN }
\limi{\qQ }
\liminfi{\qq } 
E[\chiwtN (\XX _{t\wedge \kappaq }) ] 
&&\text{ by Fatou's lemma}
\\ \notag &
\le 
\limi{\NN }
\limi{\qQ }
\liminfi{\qq } 
E[\chiwtI (\XX _{t\wedge \kappaq }) ] 
&&\text{ by $ \chiwtN \le \chiwtI $}
\\ \notag &
=
\liminfi{\qq } E[\chiwtI (\XX _{t\wedge \kappaq }) ] 
.\end{align}

By \eqref{:55k}, $ \chiwtI (\XX _0) < \infty $ a.s. 
By assumption, \eqref{:52m} and \eqref{:54z2} hold. 
Then the assumptions of \lref{l:54z} are fulfilled. 
Hence we obtain \eqref{:54z}. From \eqref{:54z} we see 
\begin{align} \label{:55b} 
\chiwtI (\XX _{t\wedge \kappaq } ) = \chiwtI (\XX _0) & + 
 \sum_{i=1}^{\infty} \int_0^{t\wedge \kappaq }
\sumklD 
 \chiwtIp (X_u^i,\XX _u^{\idia }) \sigmakl 
(X_u^i,\XX _u^{\idia }) dB_u^{\il } 
\\ \notag &
+
\int_0^{t\wedge \kappaq } \sum_{i=1}^{\infty} \sumkD 
b_{\kK } (X_u^i,\XX _u^{\idia }) 
\chiwtIp (X_u^i,\XX _u^{\idia }) 
du
\\ \notag &
+ \half 
\int_0^{t\wedge \kappaq } \sum_{i=1}^{\infty} \sumklD 
 \akl (X_u^i,\XX _u^{\idia }) 
\chiwtIpq (X_u^i,\XX _u^{\idia }) du
.\end{align}
Taking the expectation for each term in \eqref{:55b} and applying \eqref{:55k}, \eqref{:54a}, and \eqref{:55i} 
to the right-hand side of \eqref{:55b}, we deduce 
\begin{align}\label{:55c}&
\sup_{\qq \in \N } E [ \chiwtI (\XX _{t\wedge \kappaq } ) ] < \infty \quad \text{ for each $ t $}
.\end{align}

By \eqref{:55a} and \eqref{:55c},  $ E [ \chiwtI (\XX _{t\wedge \kappai }) ] < \infty 
 $  for each $ t $. Hence $ \chiwtI (\XX _{t\wedge \kappai }) < \infty $ a.s.\,for each $ t $. 
By  \lref{l:54z}, $ \{ \chiwtI (\XX _t ) \} $ is a continuous process on $ [0,\infty )$. 
From \eqref{:53x}, $ \chiwtI (\sss ) = \infty $ for $ \sss \notin \Ka $. 
Hence, if $ \kappai < \infty $, then 
$ \chiwtI (\XX _{ \kappai })  =  \lim_{t\downarrow 0 }\chiwtI (\XX _{ \kappai + t }) = \infty $ a.s.
Combining these yields 
$ \pP (\kappai \le t ) = 0 $ for each $ 0 \le t < \infty $. 
We therefore obtain $ \pP (\kappai < \infty ) = 0 $, which implies \eqref{:55z}. 
\end{proof}

\subsection{Sufficient condition for \Ass{B1}. Theorems \ref{l:58} and \ref{l:5X}} \label{s:53}

We now present a sufficient condition of \Ass{B1} for non-symmetric stochastic dynamics. 
We shall apply \tref{l:5X} to \eref{e:78}. 

\begin{theorem} \label{l:58}
Assume that \Ass{UB}, \Ass{C3}--\Ass{C5}, \eqref{:52m}, \eqref{:55k}, \eqref{:54z2}, and \eqref{:55i} hold. 
Then $ \XB $ satisfies \Ass{B1} for each $ m \in \N $. 
\end{theorem}
\begin{proof} 
\Ass{B1} for $ m = 0$ follows immediately from \pref{l:51} and \pref{l:55}. 
Each tagged particle $ X^i $ has the non-collision and non explosion properties. 
Then $ \lpath $ is well defined and \Ass{B1} for each $ m \ge 1 $ follows from that for $ m = 0$. 
\end{proof} 
\begin{corollary} \label{l:59}
Assume that \Ass{UB}, \Ass{C3}--\Ass{C5}, \eqref{:52m}, \eqref{:55h}, \eqref{:54z2}, and \eqref{:55i} hold. 
Then $ \XB $ satisfies \Ass{B1} for each $ m \in \N $. 
\end{corollary}
\begin{proof} 
\corref{l:59} follows from \lref{l:54y} and \tref{l:58}. 
\end{proof} 

\begin{theorem}	\label{l:5X}
Assume that \Ass{UB}, \Ass{C3}--\Ass{C5}, \eqref{:52m}, and \eqref{:55h}. 
Furthermore, assume $ \la :\elaw \XX _0 $ is an invariant probability measure of $ \XX $ and 
\begin{align}\label{:5X1}&
\int_{\sSS } 
\sum_{i=1}^{\infty} \sumkD 
\Big| 
b_{\kK } (s^i,\sss ^{\idia }) \chiwtIp (s^i,\sss ^{\idia }) 
\Big| d\la
< \infty 
,\\ \notag &
\int_{\sSS } 
\sum_{i=1}^{\infty} \sumklD 
\Big| 
 \akl (s^i,\sss ^{\idia })
\chiwtIpq (s^i,\sss ^{\idia })
\Big| d\la 
< \infty 
.\end{align}
Then $ \XB $ satisfies \Ass{B1} for each $ m \in \N $. 
\end{theorem}

\begin{proof}  
For $ \sss \in \Ka $, as $ \qQ \to \infty $ and then $ \NN \to \infty $, 
we have by \eqref{:52m} and \lref{l:53} 
\begin{align}\label{:5X2}
\sum_{i=1}^{\infty} \sumkD 
\Big| 
b_{\kK } (s^i,\sss ^{\idia }) 
\chiwtp (s^i,\sss ^{\idia })
\Big| 
& 
\uparrow 
\sum_{i=1}^{\infty} \sumkD 
\Big| 
b_{\kK } (s^i,\sss ^{\idia })
\chiwtIp (s^i,\sss ^{\idia }) 
\Big| 
,\\ \notag 
\sum_{i=1}^{\infty} \sumklD 
\Big| 
 \akl (s^i,\sss ^{\idia })
\chiwtpq (s^i,\sss ^{\idia })
\Big| 
& 
\uparrow 
\sum_{i=1}^{\infty} \sumklD 
\Big| 
 \akl (s^i,\sss ^{\idia })
\chiwtIpq (s^i,\sss ^{\idia })
\Big| 
.\end{align}
Then we deduce \eqref{:54z2} and \eqref{:55i} from \eqref{:5X1}, \eqref{:5X2}, the monotone convergence theorem, and the assumption that $ \la $ is an invariant probability measure of $ \XX $. 
Hence we obtain \Ass{B1} from \corref{l:59}. 
\end{proof} 

We remark that \eqref{:5X1} can be rewritten as 
\begin{align}\label{:5X3}&
\int_{\sS \ts \sSS } 
 \sumkD 
\Big| 
b_{\kK } (x,\sss ) 
\chiwtIp (x,\sss ) 
\Big| d\la ^{[1]}
< \infty 
,\\ \notag &
\int_{\sS \ts \sSS } 
 \sumklD 
\Big| 
 \akl (x,\sss ) 
\chiwtIpq (x,\sss ) 
\Big| d\la ^{[1]}
< \infty 
.\end{align}

\section{A sufficient condition for \Ass{B1} in the symmetric case} \label{s:6}

Let $ \la $ be a random point field such that $ \la (\SSsde )=1 $ and let 
$ \map{\lab }{\Ssi }{\SN }$ be a label, as before. 
We shall consider the ISDE \eqref{:23a}--\eqref{:23b} with the initial distribution $ \la \circ \lab ^{-1}$. 

Let $\{ \QQs \}$ be a family of probability measures on $ \OFF $ such that $ \XB $ defined on $ \OFQFs $ 
 is a weak solution of \eqref{:23a}--\eqref{:23b} starting at 
$ \s = \lab (\sss )$ for $ \lambda $-a.s.\,$ \sss $. 
We assume $\{ \QQs \}$ is a measurable family in the following sense.

\smallskip 
\noindent 
\Ass{MF} 
$ \QQs (A)$ is 
$ \overline{\mathcal{B}(\sSS )}^{\la } $-measurable in $ \sss $ for each $ A \in \mathcal{F} $. 
\smallskip 

We remark that \Ass{MF} is a counterpart of \As{MF} in \sref{s:25}. 
Indeed, $ \lambda $ and $ \QQs $ correspond to 
$ \pP \circ \mathbf{X}_0^{-1} $ and $ \pP (\Fs ( \mathbf{B} ) \in \cdot ) $ with $ \mathbf{s} = \lab (\sss ) $, 
respectively.

For a family of probability measures $ \{ \QQs \} $ satisfying \Ass{MF}, we set 
\begin{align}&\notag %
 \QQla = \int _{\sSS } \QQs d\la 
.\end{align}
Then, $ \XB $ under $ \QQla $ is a solution of 
\eqref{:23a}--\eqref{:23b} with the initial distribution $ \la \circ \lab ^{-1}$. 

For $ m \in \zN $, we denote by $ \XM =(\mathbf{X}^m, \XX ^{m*}) $ 
 the $ m $-labeled process given by \eqref{:31a}, where $ \mathbf{X}^{[0]} = \XX $. 
Let $ \QQxsM $ be the distribution of $ \XM $ under 
$ \QQ _{\ulab (\mathbf{x}) + \sss } $. 
Then, 
\begin{align}\label{:34C}&
\int f (\XM _t) d \QQ _{\ulab (\mathbf{x}) + \sss }  = 
\int _{\WT (\sS ^m \ts \sSS ) }f (\www _t^{[m]}) d \QQxsM 
,\end{align}
where $ \www ^{[m]}  =( w ^1,\ldots, w ^m, \sum_{i=m+1}^{\infty} \delta_{ w ^i})$ as in \eqref{:22q}. 

Let $ \la ^{[0]} = \la $. Let $ \la ^{[m]}$ be the $ m $-Campbell measure of $ \la $ for $ m \in \N $. 
We set
\begin{align}\label{:34c} & 
\QQlaM = \int _{\sS ^m \ts \sSS } \QQxsM d\la ^{[m]} \quad \text{  for $ m \in \zN $}
.\end{align}
By definition, $ \QQla ^{[0]} = \QQla \circ \XX ^{-1}$.

We set $ \mathbf{B}^m = (B^1,\ldots,B^m) $ for $ \mathbf{B}=(B^i)_{i\in\N }$. 
We make assumptions. 

\smallskip 

\noindent 
\Ass{BX} 
$ \sigma [\mathbf{B}_s^{m} ; s\le t] \subset \sigma [\mathbf{X} _s^{[m]} ;s\le t ]$ for all $ t $ under $ \QQla $ 
for each $  m \in \N $. 

\smallskip 
\noindent 
\Ass{S$_{\la }$} For each $ m \in \zN $, the $ m $-labeled process $ \XM $ under $ \OFQFm $ 
gives a symmetric, Markovian semi-group $ T_t^{[m]} $ on $ L^2(\Sm \ts \sSS ,\la ^{[m]})$ defined by 
\begin{align}
\label{:34d}&
T_t^{[m]} f ( \mathbf{x},\sss ) = \int _{\WT (\Sm \ts \sSS ) }f (\www _t^{[m]}) d \QQxsM 
.\end{align}
Furthermore, $ \la ^{[m]}$ is an invariant measure of $ T_t^{[m]} $.

\noindent \Ass{D} 
Let $ \rho _{\la }^2 $ be the two-point correlation function of $ \la $. Then, for each $ R \in \N $, 
\begin{align}&\notag 
\int_{\SR \ts \SR } \frac{1}{ \upsilon ( | x - y | )^{2}} 
\rho _{\la }^2 (x,y) dxdy < \infty 
.\end{align}
Here the function $  \upsilon $ is given at the beginning of \sref{s:5A}. 

\begin{theorem} \label{l:34}
Assume that \Ass{UB}, \Ass{MF}, \Ass{BX}, \Ass{S$_{\la }$}, and \Ass{D} hold. 
Assume $ \int_{\sSS } | \chiwtI |^2 d\la < \infty $. 
Then, $ \XB $ under $ \QQla $ satisfies \Ass{B1} for each $ m \in \{ 0 \}\cup \N $. 
\end{theorem}

\begin{remark}\label{r:33}
Under \Ass{S$_{\la }$}, a symmetric Dirichlet form associated with the solution $ \XB $ exists 
through the $ L^2$-symmetric semi-groups $ T_t^{[m]} $. 
However, the Dirichlet form is not necessarily quasi-regular. 
Hence, we can not apply the Dirichlet form technique, including the concept of capacity, directly to the solution. 
We use the fact that $ \XB $ is a solution of \eqref{:23a} and the existence of 
the associated $ L^2$-symmetric semi-groups instead. 
If the Dirichlet form associated with $ T_t^{[m]} $ is a lower Dirichlet form in the sense of \cite{o.udf} and if 
tagged particles explode, then we conjecture that the Dirichlet form is not quasi-regular. 
\end{remark}

\begin{proposition} \label{l:62} 
Under the same assumptions as for \tref{l:34}, \eqref{:51z} with $ \pP = \QQla $ holds. 
\end{proposition}

\begin{proof}  
We set $ \tauz ( \www ^{[2]} ) = \inf \{t> 0; ( w _t^1, w _t^2) \in \{ x=y \} \} $. 
Then \eqref{:51z} follows from 
\begin{align}\label{:61a}&
\QQtwo (\tauz (\www ^{[2]} )  < \infty ) = 0 
.\end{align}

Let $ \Upsilon $ be as in \sref{s:5A}. 
For $ 0 < \epsilon < 1 $ let $ \Upsilon ^{\epsilon } \in C^{\infty}([0,\infty))$ such that 
$ \Upsilon ^{\epsilon } (t)$ is constant for $ 0 \le t \le \epsilon / 2 $, 
 $ 0 \le \Upsilon ^{\epsilon }(t) \le \Upsilon (t) $ for $ \epsilon / 2 \le t \le \epsilon $,  and 
$ \Upsilon ^{\epsilon }(t) = \Upsilon (t) $ for $  \epsilon \le t $. 
Let $ \G $ and $ \Ge $ be the functions on $ \sS ^2\ts \sSS $ such that 
\begin{align}\label{:61B}&
 \G (x_1,x_2,\sss ) = \Upsilon ( |x_1-x_2| ),\quad  \Ge (x_1,x_2,\sss ) = \Upsilon ^{\epsilon } ( |x_1-x_2| )
.\end{align}
Then we have \eqref{:92w} with $ m = 2$ for $ \G $ and $ \Ge $ from \Ass{UB}, \Ass{S$_{\la }$} and \Ass{D}. 
We easily see that $\{ \Ge (\www ^{[2]} _t)\} $ are continuous semi-martingales. 
By \Ass{MF}, \Ass{BX},  and \Ass{S$_{\la }$}, we can apply \lref{l:92} to $ \{ \Ge (\www ^{[2]}  _t) \} $ 
under $ \QQla ^{[2]}$. Then we have 
\begin{align} \label{:61b}&
\Ge (\www ^{[2]}  _t) - \Ge (\www ^{[2]}  _0) = 
 \half \{ M_t^{[\Ge ]} + M_{T-t}^{[\Ge ]}(r_T ) - M_{T}^{[\Ge ]}(r_T ) \} 
.\end{align}
Here, 
$ \map{r_T}{C([0,T]; \sS ^2\ts \sSS )}{C([0,T]; \sS ^2\ts \sSS )}$ such that 
$ r_T(\mathbf{w}^{[2]}) (t) = \mathbf{w}^{[2]} ({T-t}) $. 
Furthermore, $ M^{[\Ge ]} $ is the continuous local martingale under $ \QQla ^{[0]}$ such that 
\begin{align}\label{:61d}& 
 \langle M ^{[\Ge ]} (\www ^{[2]}) \rangle _t = \sum_{i=1}^{2} 
\int_0^t 
\big ( a ( w _u^i, \ww _u^{\idia } ) \partial_i \Ge (\www _u^{[2]}), \partial_i \Ge (\www _u^{[2]}) 
\big)_{\Rd }du 
.\end{align}
Here we set $ \ww ^{\idia } = \sum_{j\not=i} \delta_{ w ^j}$ for $ \www = (\w ^j)_{j\in\N }$. 

Note that $ \QQtwo $ is not a probability measure. 
By abuse of notation, $ \Etwo [\cdot ]$ denotes the integral with respect to the measure 
$ \QQtwo $. 

Recall that $ \Upsilon ' (t) =  -1/ \upsilon (t) $. 
From \Ass{D} and $ | \Ge | \le | G |$, we have
\begin{align}\label{:61p}&
\Etwo [ \frac{1}{\upsilon ( | w _0^1- w _0^2| )^2}; w _0^1, w _0^2 \in \SR ] < \infty 
,\\\label{:61q} & \sup_{0 < \epsilon < 1 } 
 \Etwo [ |\Ge ( \www _0^{[2]} )|^2 ; w _0^1, w _0^2 \in \SR ] \le 
 \Etwo [ |\G ( \www _0^{[2]} )|^2 ; w _0^1, w _0^2 \in \SR ] < \infty 
.\end{align}
Let $ \tauRe (\www ^{[2]}) = \inf\{ t>0 ;  (\w _t^1 , w_t^2) \not\in \TRe \}  $ in \eqref{:51q}. 
From \eqref{:61b}, we see for $ 0 \le t \le T $ 
\begin{align}\label{:61e} & 
\Etwo [ \big| \Ge (\www ^{[2]}  _{t\wedge \tauRe }) - \Ge (\www ^{[2]}  _0) \big| ^2 ] 
\\\notag &= \frac{1}{4} 
\Etwo [ 
 \big| M_{t\wedge \tauRe }^{[\Ge ]} (\www ^{[2]}  ) + 
 M_{T- (t\wedge \tauRe ) }^{[\Ge ] } ( r_T(\www ^{[2]} ) ) - 
M_{T}^{[\Ge ] } ( r_T(\www ^{[2]} ) ) \big| ^2 ] 
\\ \notag & \le \half \Big\{
\Etwo [ \big| M_{t\wedge \tauRe }^{[\Ge ]} (\www ^{[2]}  ) \big| ^2 ] + 
\Etwo [ \big| M_{ T - (T \wedge \tauRe )}^{[\Ge ] }( r_T(\www ^{[2]} ) ) 
 - 
M_{T}^{[\Ge ] } ( r_T(\www ^{[2]} ) ) 
\big| ^2 ]
\Big\}
\\ \notag & \le \half \Big\{
\Etwo [ \big| M_{T\wedge \tauRe }^{[\Ge ]} (\www ^{[2]}  ) \big| ^2 ] + 
\Etwo [ \big| M_{ T - (T \wedge \tauRe )}^{[\Ge ] }( r_T(\www ^{[2]} ) )  - M_{T}^{[\Ge ] } ( r_T(\www ^{[2]} ) ) 
\big| ^2 ]
\Big\}
\\ \notag & = 
\Etwo [ \big| M_{T \wedge \tauRe }^{[\Ge ]} (\www ^{[2]}  ) \big| ^2 ] \quad \text{ by \Ass{S$_{\la }$}}
.\end{align}
From \eqref{:61B}, \eqref{:61d}, \Ass{UB}, and \Ass{S$_{\la }$}, we deduce that 
\begin{align}\label{:61f}
\Etwo & [ \big| M_{ T \wedge \tauRe }^{[\Ge ]} (\www ^{[2]}) \big| ^2] 
=
\Etwo [ \langle M ^{[\Ge ]} (\www ^{[2]}) \rangle _{ T \wedge \tauRe } ] &
\\ \le \notag & 
\cref{;34}
\Etwo [
\int_0^{ T \wedge \tauRe } 
 \frac{1}{\upsilon ( | w _u^1- w _u^2| )^2} du 
 ]
\le 
\cref{;34} T 
\Etwo [ \frac{1}{\upsilon ( | w _0^1- w _0^2| )^2} ;  w _0^1, w _0^2 \in \SR ]
.\end{align}

Putting \eqref{:61p}--\eqref{:61f} together, we deduce 
\begin{align}\label{:61G} 
&\sup_{0 < \epsilon < 1 } \Etwo [ \Big|\Ge (\www ^{[2]}  _{ T \wedge \tauRe }) \Big|^2 
 ;  w _0^1, w _0^2 \in \SR ] 
\\ \notag =&
\sup_{0 < \epsilon < 1 }
\Etwo [ \Big| \Ge (\www ^{[2]}  _{ T \wedge \tauRe }) - 
\Ge (\www ^{[2]}  _0) + \Ge (\www ^{[2]}  _0) \Big|^2 
 ;  w _0^1, w _0^2 \in \SR ] 
\\ \notag \le &
2 \Big\{ \sup_{0 < \epsilon < 1 }
\Etwo [ \Big| \Ge (\www ^{[2]}  _{ T \wedge \tauRe }) - 
\Ge (\www ^{[2]}  _0) \Big|^2 ] 
+
\sup_{0 < \epsilon < 1 }
\Etwo [ \Big| \Ge (\www ^{[2]}  _0) \Big|^2 ;  w _0^1, w _0^2 \in \SR ] \Big\} 
\\ \notag <&
 \infty 
.\end{align}
Let $ \tauRz = \limz{\epsilon } \tauRe $. Then, from Fatou's lemma and \eqref{:61G}, we obtain 
\begin{align}\notag 
\Etwo [ | \G (\www ^{[2]}  _{ T \wedge \tauRz }) |^2 ;  w _0^1, w _0^2 \in \SR ] 
& \le \liminfz{\epsilon } 
\Etwo [ | \G (\www ^{[2]}  _{ T \wedge \tauRe }) |^2 ;  w _0^1, w _0^2 \in \SR ] 
\\ \notag &
= \liminfz{\epsilon } \Etwo [ | \Ge (\www ^{[2]}  _{ T \wedge \tauRe }) |^2 ;  w _0^1, w _0^2 \in \SR ] 
< \infty 
.\end{align}
Hence we have for all $ R , T \in \N $ 
\begin{align}\label{:61j}&
1_{\SR ^2 } ( w _0^1, w _0^2) | \G (\www ^{[2]}  _{ T \wedge \tauRz })| < \infty \quad \text{ $ \QQtwo $-a.e}
.\end{align}
From \eqref{:61B} and \eqref{:61j}, we see $ \tauRz = \tauR $ holds $ \QQtwo$-a.e.\,for each $ R \in \N $, where 
$ \tauR $ is the exit time of $ ( w ^1, w ^2)$ from $ \SR ^2 $ as before. 
From this combined with \eqref{:34C} and \eqref{:34c}, we see 
$ \tauRz = \tauR $ holds $ \QQtwo $-a.s.\,for each $ R \in \N $. Hence, 
\begin{align}\label{:61k}&
 \QQtwo ( \limi{R} \tauRz < \infty ) = \QQtwo ( \limi{R} \tauR < \infty ) 
.\end{align}
Because each tagged particle does not explode, we see 
\begin{align}\label{:61l}&
 \QQtwo ( \tauz = \limi{R}\tauRz ) = 1, \quad \QQtwo ( \limi{R} \tauR < \infty ) = 0
.\end{align}
Hence, we obtain \eqref{:61a} from \eqref{:61k} and \eqref{:61l}. 
\end{proof} 

\begin{proposition} \label{l:63}
Make the same assumptions as for \tref{l:34}. Assume \eqref{:52m} in addition. Then \eqref{:55z} holds. 
\end{proposition}

\begin{proof}  
Note that $ \chiwtN $ is local smooth and $ \chiwtN $ and their derivatives are continuous on $ \sSS $. 
Applying \lref{l:92} to $ \chiwtN $ we have under $ \QQla ^{[0]}$ 
\begin{align}
\label{:63a}&
\chiwtN (\ww _t) - \chiwtN (\ww _0) = 
\half \{ M_t^{[\chiwtN ]} + M_{T-t}^{[\chiwtN ] }(r_T ) - M_{T}^{[\chiwtN ] }(r_T ) \} 
.\end{align}
Here, $ \map{r_T}{C([0,T];\sSS )}{C([0,T];\sSS )}$ is such that $ r_T(\ww ) (t) = \ww ({T-t}) $, 
where $ \ww =\{ \ww (t) \} $. Furthermore, 
$ M^{[\chiwtN ]} $ is a continuous local martingale under $ \QQla ^{[0]}$ such that 
\begin{align} \label{:62c}
 \langle M ^{[\chiwtN ]} \rangle _t (\ww ) & = 2 
\int_0^t 
\DDDa [ \chiwtN ,\chiwtN ] (\ww _u) du 
.\end{align}
By \eqref{:54b}, \eqref{:54bb}, \eqref{:63a}, and \eqref{:62c}, 
we deduce that $  M^{[\chiwtN ]} $  converges to the continuous local $ L^2$-martingale 
$ M^{[\chiwtI ]} $ such that 
\begin{align}
\label{:63c}&
\chiwtI (\ww _{t\wedge \kappaq }) - \chiwtI (\ww _0) = 
\half \{ M_{t\wedge \kappaq }^{[\chiwtI ]} + 
M_{T- ({t\wedge \kappaq (r_T ) })}^{[\chiwtI ] }(r_T ) - M_{T}^{[\chiwtI ] }(r_T ) \} 
,\\ \label{:63C}& 
 \langle M ^{[\chiwtI ]} \rangle _{t\wedge \kappaq } (\ww ) 
= 2 
\int_0^{t\wedge \kappaq }  1_{\Ka } (\ww _u) 
\DDDa [ \chiwtI ,\chiwtI ] (\ww _u) du 
.\end{align}
Taking the expectation of the square of both sides of \eqref{:63c}, we have 
\begin{align}
\label{:62d}
E[ |\chiwtI (\ww _{t\wedge \kappaq }) - & \chiwtI (\ww _0) |^2 ] 
= \frac{1}{4} 
E[ 
| M_{t\wedge \kappaq }^{[\chiwtI ]} + 
 M_{T- ({t\wedge \kappaq (r_T ) })}^{[\chiwtI ] }(r_T ) -M_{T}^{[\chiwtI ] }(r_T ) 
| ^2 ]
\\ \notag & \le \half \Big\{
E[ | M_{t\wedge \kappaq }^{[\chiwtI ]}|^2 ] + 
E[ | M_{T- ({t\wedge \kappaq (r_T ) })}^{[\chiwtI ] }(r_T ) -M_{T}^{[\chiwtI ] }(r_T )| ^2 ]
\Big\}
\\ \notag & \le 
E[ | M_{T \wedge \kappaq }^{[\chiwtI ]}|^2 ] 
.\end{align}
Then, we deduce from \eqref{:54bb} and \eqref{:63C} that
\begin{align}
\label{:62p}& 
E[ | M_{T \wedge \kappaq }^{[\chiwtI ]}|^2 ] 
= 
E[ \langle M ^{[\chiwtI ]}\rangle _{T \wedge \kappaq } ] 
= 2 
E[ \int_0^{ T \wedge \kappaq } 
 1_{\Ka } (\ww _u) 
\DDDa [ \chiwtI ,\chiwtI ] (\ww _u)
 du ] 
\le 2 T \cref{;34}
.\end{align}
Combining \eqref{:62d} with \eqref{:62p}, we see that for each $ 0\le t \le T < \infty $
\begin{align} \label{:62q} 
 \sup_{\qq \in \N } E[ | \chiwtI (\ww _{t\wedge \kappaq }) | ^2 ] & \le 
\half\sup_{\qq \in \N } 
 \{E[ |\chiwtI (\ww _{t\wedge \kappaq }) - \chiwtI (\ww _0) |^2 ] + E[ | \chiwtI (\ww _0) | ^2 ] \} 
\\ \notag &
\le \half \{ 2 T \cref{;34} + E[ | \chiwtI (\ww _0) | ^2 ] \} 
\\ \notag &
< \infty \quad \text{ by $ \int_{\sSS } | \chiwtI | ^2 d\la < \infty $}
.\end{align}
Let $ \kappai =\limi{\qq } \kappaq $ as before. 
Then, similarly as \eqref{:55a}, we have by \eqref{:62q} 
\begin{align}
\label{:62s}&
E [\chiwtI (\ww _{t\wedge \kappai })] \le 
\liminfi{\qq } E[\chiwtI (\ww _{t\wedge \kappaq }) ] < \infty 
\quad \text{ for each $ 0 \le t < \infty $}
.\end{align}

By \eqref{:62s}, we see $ \chiwtI (\ww _{t\wedge \kappai }) < \infty $ a.s.\,for each $  t < \infty $. 
By \eqref{:53a}, we see $ \chiwtI (\ww _{ \kappai }) = \infty $. 
Combining these, we deduce $ \QQtwo ( t < \kappai ) = 1 $ for each $ t $, which implies \eqref{:55z}. 
\end{proof} 

\noindent {\em Proof of \tref{l:34}. } 
\tref{l:34} follows from \pref{l:62} and \pref{l:63}. 
\qed

\section{Examples} \label{s:7}

Following \cite{o-t.tail}, we present examples satisfying the assumptions of the main theorems. 
In all the examples in this section, $ \sigma $ is the unit matrix. 
In \eref{e:71}--\eref{e:77}, $ b (x,\yy ) = \frac{1}{2} \dmu (x,\yy )$, 
where $ \dmu $ is the logarithmic derivative of the random point field $ \mu $ associated with the ISDE. 

The first three examples are infinite particle systems in one-dimensional space, and the fourth example is in $ \R^2 $. 
These four examples arise from the random matrix theory and have logarithmic interaction potential. 

\eref{e:75}--\eref{e:77} are related to Ruelle's class interaction potentials. 
The equilibrium states for these examples are canonical Gibbs measures described by the Dobrushin--Lanford--Ruelle (DLR) equation. 
We consider only non-symmetric solutions in \eref{e:78}. 
Here, non-symmetric means the associated unlabeled dynamics are not reversible to the given equilibrium state. 
We construct such dynamics by adding skew-symmetric drift coefficients. 

\begin{example}[sine$_{\beta }$ random point fields] \label{e:71}
Let $ \sS = \mathbb{R}$. We consider 
\begin{align}& \label{:71a} 
 dX_t^i = dB_t^i + \frac{\beta }{2} 
 \lim_{r\to\infty }\sum_{|X_t^i - X_t^j |< r, \ j\not= i}^{\infty} 
 \frac{1}{X_t^i - X_t^j } dt \quad (i\in\mathbb{Z})
.\end{align}
Let $ \mu _{\mathrm{sin},\beta }$ be the sine$_{\beta }$ random point field \cite{Meh04,forrester}. 
We take $ \beta = 1,2,4$. 
By definition, $ \mu _{\mathrm{sin},2 } $ is the random point field on 
$ \mathbb{R}$ for which the $ n $-point correlation function $ \rho_{\mathrm{sin},2 }^{n} $ 
with respect to the Lebesgue measure is given by 
\begin{align} \notag &
 \rho_{\mathrm{sin},2 }^{n} (\mathbf{x}^n ) = \det [K _{\mathrm{sin},2 } (x_i,x_j)]\ijn 
. \end{align}
Here, $K _{\mathrm{sin},2 } (x,y) = \sin \pi (x-y)/\pi (x-y)$ is the sine kernel. 
$ \musinone $ and $ \musinfour $ are also defined by correlation functions 
given by quaternion determinants \cite{Meh04}. 
For $ \beta =1,2,4$, $ \mu _{\mathrm{sin},\beta } $ 
are quasi-Gibbs measures \cite{o.rm,o.isde} and the logarithmic derivatives are given by 
\begin{align}\label{:71c}&
 \dlog ^{ \mu _{\mathrm{sin},\beta } } (x,\yy ) = \beta \limi{r} \sum_{|x-y_i|<r} 
\frac{1}{x-y_i}
\quad \text{ in }L_{\mathrm{loc}}^1 ( \R \ts \sSS , \mu _{\mathrm{sin},\beta }^{[1]})
.\end{align}
If $ \beta = 2 $, the solution of \eqref{:71a} is called the Dyson model in infinite dimensions \cite{Spo87}. 
\end{example}


\begin{example}[Bessel random point fields]\label{e:73} 
Let $ \sS = [0,\infty)$ and $ 1 \le \alpha < \infty $. 
Let $ \beta = 2 $. We consider 
\begin{align} \label{:73a} 
& dX_t^i = dB_t^i + \{ \frac{\alpha }{2X_t^i } + \sum _{ j\not = i }^{\infty}
\frac{1}{X_t^i - X_t^j} \} dt \quad (i \in \N )
.\end{align}
Let $ \mu _{\mathrm{Be},\alpha }$ be the Bessel$_{2,\alpha }$ random point field. 
The $ n $-point correlation function $ \rho _{\mathrm{Be},\alpha }^n $ 
with respect to the Lebesgue measure is given by 
\begin{align}\notag &
 \rho _{\mathrm{Be},\alpha }^n (\mathbf{x}^n) = 
 \det [ K _{\mathrm{Be},\alpha }(x_i,x_j)] \ijn 
.\end{align}
Here, $ K _{\mathrm{Be},\alpha } $ is a continuous kernel given by 
the Bessel function $ J_{\alpha }$ for $ \alpha $ of the first kind such that 
\begin{align}\notag 
&
 K _{\mathrm{Be},\alpha }(x,y) = 
 \frac{J_{\alpha } (\sqrt{x}) \sqrt{y} J_{\alpha }' (\sqrt{y}) - 
 \sqrt{x} J_{\alpha }' (\sqrt{x})J_{\alpha }(\sqrt{y}) }{2(x-y)}
\quad (x\not=y)
.\end{align}
In \cite{h-o.bes}, it was proved that the logarithmic derivative $ \dlog ^{ \mu _{\mathrm{Be},\alpha }}$ 
of $ \mu _{\mathrm{Be},\alpha }$ is given by 
\begin{align}\label{:73d}&
 \dlog ^{ \mu _{\mathrm{Be},\alpha }} (x,\yy ) = \frac{\alpha}{x} + \sum_{i} \frac{2}{x-y_i}
\quad \text{ in }L_{\mathrm{loc}}^1 (\R \ts \sSS , \mub ^{[1]})
.\end{align}
The sum in \eqref{:73d} converges absolutely, unlike in the previous examples. 
\end{example}

\begin{example}[Ginibre random point field] \label{e:74} 
Let $ S=\R ^2 $ and $ \beta = 2 $. We consider two ISDEs: 
\begin{align}\label{:74a}&
dX_t^i = dB_t^i + \limi{r} \sum_{|X_t^i-X_t^j|<r,\ j\not=i } 
\frac{X_t^i-X_t^j}{|X_t^i-X_t^j|^2} dt \quad (i\in\mathbb{N})
\end{align}
and 
\begin{align}\label{:74b}&
dX_t^i = dB_t^i - X_t^i dt + \limi{r} \sum_{|X_t^j|<r,\ j\not=i } 
\frac{X_t^i-X_t^j}{|X_t^i-X_t^j|^2} dt \quad (i\in\mathbb{N})
.\end{align}
Note that ISDEs \eqref{:74a} and \eqref{:74b} have 
the same weak solutions \cite{o.isde}. The associated unlabeled diffusion is reversible with respect to the Ginibre random point field $ \mug $, which is the determinantal random point field on $ \Rtwo $ that has the kernel 
$ \kg (x,y) = e^{x\bar{y}}$ with respect to the complex Gaussian measure $ (1/\pi ) e^{-|z|^2}dz $. 
Here, we regard $ \Rtwo $ as $ \mathbb{C} $ in an obvious manner. $ \mug $ is quasi-Gibbs and has logarithmic derivatives with plural representations \cite{o.rm,o.isde}. 
In $ L_{\mathrm{loc}}^2(\Rtwo \ts \sSS , \mug ^{[1]})$, we have 
\begin{align*}
\dlog ^{\mug } (x,\sss )&
= \limi{ r } 2\sum_{|x-s_i|\le r } \frac{x-s_i}{|x-s_i|^2} 
 = -2x + \limi{ r } 2\sum_{|s_i| \le r } \frac{x-s_i}{|x-s_i|^2}
.\end{align*}
\end{example}

All of the above examples are related to random matrix theory. 
ISDEs \eqref{:71a} with $ \beta = 1,2,4 $ 
are the bulk scaling limit of the finite-particle systems of 
Gaussian orthogonal/unitary/symplectic ensembles, respectively. 
ISDE \eqref{:73a} 
is the hard-edge scaling limit of finite-particle systems of the Laguerre ensembles. 
 ISDEs \eqref{:74a} and \eqref{:74b} are bulk scaling limits of the Ginibre ensemble, which is a system of eigenvalues of non-Hermitian Gaussian random matrices.

\begin{example}[Ruelle's class potentials]\label{e:75} 
Let $ \sS = \Rd $ with $ d \in \N $. Let $ \Phi = 0$ and $ \Psi (x,y) = \beta \Psi_0 (x-y)$. 
The ISDE then becomes 
\begin{align}\notag &
dX^i_t = dB^i_t - 
\frac{\beta }{2} \sum ^{\infty}_{j=1,j\ne i} 
\nabla \Psi_0 (X_t^i - X_t^j ) dt 
\quad (i\in \N )
. \end{align}
Assume that $ \Psi_0 $ is a Ruelle's class potential that is smooth outside the origin. 
That is, $ \Psi_0 $ is super-stable and regular in the sense of Ruelle \cite{ruelle.2}. Here, we say that $ \Psi_0$ is regular if there exists a positive decreasing function $ \map{\psi }{\mathbb{R}^+}{\mathbb{R}}$ and $ R_0 $ such that 
\begin{align} \notag & 
\Psi _0 (x) \ge - \psi (|x|) \quad \text{ for all } x ,
 \quad 
\Psi _0 (x) \le \psi (|x|) \quad \text{ for all } 
 |x| \ge R_0 ,
\\ \notag &
\int_0^{\infty} \psi (t)\, t^{d-1}dt < \infty 
.\end{align}
Let $ \mupsi $ be a canonical Gibbs measure with interaction $ \Psi _0$. 
Let $ \rho ^m $ be the $ m $-point correlation function of $ \mupsi $. 
We assume a quantitative condition in \eqref{:75c}. 
\\
\noindent 
\Ass{Gib} For each $ \pp \in \N $, there exist positive constants 
$ \Ct \label{;29c} $ and $ \Ct \label{;29d} $ satisfying 
\begin{align} \label{:75c}&
\sum_{r=1}^{\infty} 
\frac{\int_{\Sr }\rho ^1 (x)dx } {r^{ \cref{;29c} +1 }} < \infty , \quad 
\limsup_{r\to\infty } 
\frac{\int_{\Sr }\rho ^1 (x)dx } {r^{ \cref{;29c} }} < \infty 
,\\
\notag 
& 
 | \nabla \Psi _0 ( x )|,\ | \nabla ^2 \Psi _0 ( x )| \le
\frac{\cref{;29d} }{ (1+|x|)^{ \cref{;29c} }}
\quad \text{ for all $ x $ such that $ |x|\ge 1/\pp $. }
\end{align}
It was proved in \cite[{Lemma 13.5}]{o-t.tail} that the logarithmic derivative of $ \mupsi $ is given by 
\begin{align}
\label{:75e}&
\dmu (x.\yy ) = -\beta \sum_{j=1}^{\infty} \nabla \Psi _0 (x-y_j) 
.\end{align}
The sum in \eqref{:75e} converges absolutely, unlike Examples \ref{e:71} and \ref{e:74}. 
\end{example}

The next two examples are individual cases of Example \ref{e:75}. 
We present only the interaction potentials and ISDEs. 
\begin{example}[Lennard--Jones 6-12 potentials]\label{e:76}
Let $ d= 3$, $ \beta > 0$, and 
$$\Psi_{6,12}(x) = \{|x|^{-12}-|x|^{-6}\} .$$
The interaction $ \Psi_{6,12}$ is called the Lennard--Jones 6-12 potential. 
The ISDE is 
\begin{align} \notag 
dX^i_t = dB^i_t + \frac{\beta }{2} \sum ^{\infty}_{j=1,j\ne i} \{
\frac{12(X ^i_t-X^j_t)}{|X^i_t-X^j_t|^{14}} - 
\frac{6(X ^i_t-X^j_t)}{|X^i_t-X^j_t|^{8}}\, \} dt \quad (i\in \N )
. \end{align}
\end{example}
\begin{example}[Riesz potentials] \label{e:77}
Let $ d < a \in \R $, $ 0 < \beta $, 
and set $\Psi_a(x)=(\beta / a )|x|^{-a}$. 
The corresponding ISDE is 
\begin{align}& \notag 
dX^i_t=dB^i_t + \frac{\beta }{2}
\sum ^{\infty}_{j=1,j\ne i} \frac{X ^i_t-X^j_t}{| X^i_t-X^j_t|^{a+2}}dt 
\quad (i\in \N ) 
.\end{align}
\end{example}

\begin{example}[Non-symmetric case] \label{e:78}
Let $ \Psi _0 $ be a Ruelle's class potential. We assume $ \Psi _0 \in C_0 ^3 (\Rd )$ and $ d \ge 3 $. 
Let $ \mu $ be an associated canonical Gibbs measure. We take $ \la = \mu $. 
We assume that $ \mu $ has locally bounded $ m $-point correlation functions for all $ m $. 
Then, the logarithmic derivative of $ \mu $ is given by 
\begin{align} \notag &
\dmu (x,\sss ) = - \beta \sum _{i} \nabla \Psi _0 (x-s_i )
.\end{align}
Let $ \gamma _0 $ be an $ \Rd $-valued function on $ \Rd $ such that $ \gamma _0 \in C_0^2 (\Rd ) $. 
Let 
\begin{align} \notag &
\gamma (x,\sss ) = \beta \sum_{i} \gamma _0 (x-s_i)
.\end{align}
We consider the ISDE
\begin{align}\label{:78z}&
dX_t^i = dB_t^i + \half \{\dmu (X_t^i, \XX _t^{\idia }) + \gamma (X_t^i, \XX _t^{\idia })\}dt 
\end{align}
under the assumption that 
\begin{align}\label{:78c}&
\mathrm{div} \gamma + \gamma \cdot \dmu = 0
.\end{align}
An example of $ \Psi _0 $, $ \mu $, and $ \gamma _0$ satisfying \eqref{:78c} is 
$ \Psi _0 = 0 $, $ \la = \mu $ is the Poisson random point field whose intensity is the Lebesgue measure, and 
$ \gamma _0 = (\gamma _{0 \kK })_{\kK =1}^d$ 
is the derivative of a skew-symmetric potential $ \Gamma = (\Gamma _{\kl })\klD $ such that 
\begin{align} \notag &
\gamma _0 (x) = \sumkD \PD{\Gamma _{\kl }}{x_{\lL }} (x), \quad \Gamma _{\kl } (x)= -\Gamma _{\lL \kK }(x) 
.\end{align}
Here $ x=(x_1,\ldots,x_d)\in\Rd $. 
From \eqref{:78c}, $ \mu ^{[m]}$ is an invariant measure of $ \XX ^{[m]}$. 

To apply \tref{l:31} and \tref{l:32}, we check \Ass{B1}, \Ass{B2}, and \Ass{B3}. 
 \Ass{B2} and \Ass{B3} follow from \tref{l:33}. 
Indeed, we can take $ \ell = 1$ because $ \Psi _0$ and $ \Gamma $ have compact supports. 
As for the construction of a weak solution of \eqref{:78z}, we can use \lref{l:80}. 
We can take a suitable finite particle approximation $ \muN $ because $ \Psi _0 $ is of Ruelle's class 
and has a compact support. Moreover, $ \gamma _0 $ also has a compact support. 

To obtain \Ass{B1}, we use \tref{l:5X}. 
So we quickly check the assumptions of \tref{l:5X}. 
\Ass{UB} is obvious. 
\Ass{C3} is clear because the two-point correlation function is bounded. 
\Ass{C4} and \Ass{C5} follow from boundedness of the two-point correlation function, $ d \ge 3 $, 
and the assumptions such that 
$ \Psi _0 \in C_0^3(\Rd )$ and $ \gamma _0 \in C_0^2 (\Rd ) $ and that 
$ \mu ^{[m]}$ are invariant measures of $ \XX ^{[m]}$. 
It is not difficult to see that $ \mu $ satisfies \eqref{:52m}, \eqref{:55h}, and \eqref{:5X1} 
if $ \mu $ is a translation invariant Poisson random point field. 
Hence, we obtain \Ass{B1} from \tref{l:5X}. 

It is plausible that one can generalize the example to canonical Gibbs measures and the long-range case. For this, more work is required and is left to the reader. 
\end{example}

\section{Appendix I: Weak solutions of ISDEs} \label{s:8}

In this section, we quickly review some previous results. 
In \cite{o.tp,o.isde,o.rm,o.rm2,k-o.fpa}, we presented weak solutions to ISDEs; 
 \cite{o.tp,o.isde,o.rm,o.rm2} were devoted to symmetric cases, whereas \cite{k-o.fpa} considered both non-symmetric and symmetric cases. 
Hence, together with the results in \cite{o-t.tail} and the present paper, we obtain unique strong solutions of ISDEs. 

\subsection{Construction of weak solutions: non-symmetric case} \label{s:80}
In \sref{s:80}, we follow the process for constructing weak solutions in \cite{k-o.fpa}. 
The results are valid for non-symmetric solutions.

Let $ \{\muN \} $ be a sequence of random point fields on $\sS $ 
such that $ \muN (\{ \sss (\sS ) = \nN \} ) = 1 $. 
Let $ \labN $ be a label of $ \muN $ and $ \labmN = (\lab ^{\nN ,1},\ldots,\lab ^{\nN ,m})$, where 
$ m \le \nN $. 
We assume the following. \\
\Ass{H1} Each $ \muN $ has a correlation function $ \{\rho ^{N,n}\} $ with respect to the Lebesgue measure satisfying, for each $r \in\N$, 
\begin{align}\label{:80a} & 
\limi{N} \rho ^{N,n} (\mathbf{x})= 
\rho ^{n} (\mathbf{x}) \quad \text{ uniformly on $\Sr ^{n}$ for all $n\in\N$} 
,\\ \label{:80b}& 
\sup_{N\in\N } \sup_{\mathbf{x} \in \Sr ^{n}} \rho ^{N,n} (\mathbf{x}) \le 
\cref{;40b} ^{n} n ^{\cref{;40c}n} \quad \text{ for all $n\in\N$}
,\end{align}
where $ 0 < \Ct \label{;40b}(r) < \infty $ and $ 0 < \Ct \label{;40c}(r)< 1 $ are constants independent of $ n \in \N $.

\smallskip 
\noindent
\Ass{H2} For each $m\in\N$, 
$ \limi{N}\mu^{\nN } \circ (\labmN )^{-1} =\mu \circ (\labm )^{-1}$ 
weakly in $ \sS ^m $. 

\smallskip 

We take $ \mu^{\nN } \circ (\labN )^{-1} $ as an initial distribution of the labeled finite-particle system, 
and \Ass{H2} refers to the convergence of the initial distribution of the labeled dynamics. 

For $ \mathbf{X}^{\nN }=(X^{\nN ,i})_{i=1}^{\nN }$, we set 
$ \XX _t^{N,\idia}= \sum_{j\not=i}^{ N } \delta_{X_t^{N,j}} $, 
where $X_t^{N,\idia}$ denotes the zero measure for $ \nN = 1 $. 
Let $ \map{\sN }{\sS \ts \sSS }{\R ^{d^2}}$ and 
$ \map{\bbb ^{\nN } }{\sS \ts \sSS }{\Rd }$ be measurable functions. 
The finite-dimensional SDE of $ \mathbf{X}^{\nN } =(X^{\nN ,i})_{i=1}^{\nN } $ is given by 
\begin{align}\label{:80d}
dX_t^{N,i} &= 
\sN (X_t^{N,i},\XNidt )dB_t^i + \bbb ^{\nN }(X_t^{N,i},\XNidt )dt 
\quad \text{ ($ 1\le i\le N $)}
,\\\label{:80e}
\mathbf{X} _0^{\nN } & = \mathbf{s} 
.\end{align}

\noindent \Ass{H3} 
SDE \eqref{:80d} and \eqref{:80e} has a weak solution for $ \muN \circ (\labN )^{-1}$-a.s.\! $ \mathbf{s}$ 
for each $ \nN $ and this solution neither explodes nor hits the boundary (when $ \partial \sS $ is non-void).

\smallskip 

\noindent 
\Ass{H4} 
$\sN $ are bounded and continuous on $ \sS \ts \sSS $, 
and converge uniformly to $\sigma $ on $ \SrSS $ for each $ r \in \N $. 
Furthermore, $\aN := \sN {}^{t}\sN $ are uniformly elliptic on $ \Sr \ts \sSS $ for each $ r \in \N $ and 
$ \PD{}{x} \aN (x,\sss ) $ are uniformly bounded on $ \sS \ts \sSS $.

\smallskip 

Let $ \overline{\mathbf{X}}_T^{N,m}$ be the maximal module variable of 
the first $ m $ particles such that 
\begin{align}&\notag 
\overline{\mathbf{X}}_T^{N,m}= \max_{i=1,\ldots,m} 
 \sup_{t\in [0,T] }|X_t^{N,i}| 
.\end{align}

\noindent 
\Ass{I1} For each $ T , m \in \N $, 
\begin{align}&\notag %
\limi{a} \liminfi{\nN }P ^{\muNl } (\overline{\mathbf{X}}_T^{N,m} \le a ) = 1 
\end{align}
and, for each $ m , a \in \N $, there exists a constant $\Ct \label{;41a}=\cref{;41a}(m,a)$ such that, for $ 0 \le t , u \le T $,
\begin{align}&\notag 
\supN \sum_{i=1}^m 
\mathrm{E} ^{\muNl }
[|X _t^{N,i} - X _u^{N,i}|^{4 };\overline{\mathbf{X}}_T^{N,m} \le a ] 
 \le \cref{;41a} |t-u|^{2} 
.\end{align}
Furthermore, $ \mrT $, defined by \eqref{:25f}, satisfies 
\begin{align}&\notag 
\limi{L} \liminfi{\nN }P ^{\muNl } ( \mrXXN \le L ) = 1 
\quad \text{ for each $ r \in \N $}
.\end{align}

Let $ \muNone $ be the one-Campbell measure of $ \muN $. 
Set $ \Ct \label{;43}(\rrr , N)= \muNone ( \Sr \ts \sSS )$. 
Then, by \eqref{:80b}, $\sup_{N}\cref{;43}(\rrr ,N)<\infty $ for each $\rrr \in \N $. 
Without loss of generality, we can assume that $\cref{;43}(\rrr , N) > 0 $ for all $\rrr , N $. 
Let $ \muNone _r = \muNone (\cdot \cap \{\Sr \ts \sSS \}) $. 
Let $ \muNonebar $ be the probability measure defined as 
\begin{align*}&
\muNonebar (\cdot )= \muNone (\cdot \cap \{\Sr \ts \sSS \})/\cref{;43}(\rrr , N)
.\end{align*}
Let $ \varpi _{r,s}$ be the map from $ \Sr \ts \sSS $ to itself such that 
$ \varpi _{r,s} (x,\sss ) = (x,\sum_{|x-s_i|<s} \delta_{s_i})$, where $ \sss = \sum_i \delta_{s_i}$. 
Let $ \mathcal{F}_{\rrr ,s} = \sigma [\varpi _{r,s}]$ be the sub-$ \sigma $-field of 
$ \mathcal{B}(\Sr \ts \mathcal{\sSS }) $. 
Because $ \Sr \subset \sS $, we can and do regard 
$ \mathcal{F}_{\rrr ,s} $ as a $ \sigma $-field on $ \sS \ts \sSS $, 
which is trivial outside $ \Sr \ts \sSS $. 

We set a tail-truncated coefficient $ \bNrs $ of $ \bN $ and their tail parts $ {\bbb _{r,s}^{N,\tail}}  $ by 
\begin{align} \label{:80f}&
 \bbb _{r,s}^{\nN } = \mathrm{E}^{\muNonebar } [ \bbb ^{N}|\mathcal{F}_{\rrr ,s} ] ,\quad 
\bN = \bbb _{r,s}^{\nN } + {\bbb _{r,s}^{N,\tail}}  
.\end{align}
By construction, $ \muNonebar ( \Sr ^c \ts \sSS ) = 0 $. 
Hence, we can and do take a version of $ \bNrs $ such that 
\begin{align}
\label{:80g}& 
\bNrs (x,\yy ) = 0 \text{ for } x \not\in \Sr 
,\quad 
 \bNrs (x,\yy ) = \bbb _{r+1,s}^{N} (x,\yy ) 
\text{ for } x \in \Sr 
.\end{align}
Let $ \bNrsp $ be a continuous and $ \mathcal{F}_{\rrr ,s}$-measurable function on $ \sS \ts \sSS $ such that 
\begin{align}\notag 
& \bNrsp (x,\yy ) =0 \quad \text{ for } x \not\in \Sr 
,\\\notag 
& \bNrsp (x,\yy ) = \bbb _{r-1,s,\p }^{N} (x,\yy ) \quad \quad \text{ for } x \in \sS _{r-1} 
,\\\label{:80h}&
 \bNrsp (x,\yy ) = 0 \quad \text{ for $ (x,\yy ) \in (\sS \ts \sSS )_{r,\p +1} $}
,\\ \label{:80i}&
 \bNrsp (x,\yy ) = \bbb _{r,s}^{N} (x,\yy ) 
 \quad \text{ for $ (x,\yy ) \not\in (\sS \ts \sSS )_{r,\p } $}
.\end{align}
Here, $ (\sS \ts \sSS )_{r,\p } = \{ (x,\yy ) \in \Sr \ts \sSS 
;\, |x- y_i|\le 1/2^{\p } \text{ for some } y_i \}$, where $ \yy =\sum_i\delta_{y_i}$.

The main requirements for $ \bN $ and $ \bNrsp $ are the following:

\noindent 
\Ass{I2} There exists some $ \phat $ such that 
$ 1 < \phat $ and, for each $ r \in \N $,
\begin{align}\notag &
\limsupi{\nN } \int _{\SrSS }|\bN | ^{\phat } d\muNone < \infty 
.\end{align}
For each $ r , i , T \in \N $, there exists a constant $ \Ct \label{;40i}$ such that 
\begin{align}\notag &
\sup_{\p \in\N } 
\sup_{\nN \in \N } E ^{\muNl } 
[\int_0^T | \bNrsp (X_t^{N,i},\XNidt ) | ^{\phat } dt ] \le 
\cref{;40i}
.\end{align}

We decompose $ \bNrs $ as 
\begin{align}\label{:80j}&
\bNrs = \bNrsp +  {\bNrs - \bNrsp } 
.\end{align}

Let $ \| \cdot \|_{\sS \ts \SSrm } $ be the uniform norm on the space of functions $ \sS \ts \SSrm $, 
where $ \SSrm = \{\sss \in \sS \ts \sSS \, ; \sss (\Sr ) = m\}$. 
Then $ \| f\|_{\sS \ts \SSrm } = \sup \{ |f (x,\sss ) | \, ; \, (x,\sss ) \in \sS \ts \SSrm \} $ by definition. 

\medskip 
\noindent 
\Ass{I3} 
For each $ m ,\p , r , s \in \N $ such that $ r < s $, there exists some $ \bbb _{r,s,\p } $ such that 
\begin{align}
\label{:80k} & 
\limi{\nN } \| \bbb _{r,s,\p }^{\nN } - \bbb _{r,s,\p } \|_{\sS \ts \SSsm }
= 0 
.\\
\intertext{Moreover, $ \bbb _{r,s,\p }^{\nN } $ are differentiable in $ x $ and satisfy the bounds: }
\notag &
\supN \| \PD{}{x} \bbb _{r,s,\p }^{\nN } \|_{\sS \ts \SSsm } < \infty 
,\\\notag &
\limi{\p } \supN \| \bNrsp - \bNrs \|_{ L^{\phat }(\Sr \ts \sSS , \muNone ) } =0 
.\end{align}
Furthermore, for each $ i , r < s , T \in \N $, we assume that 
\begin{align}\notag &
\limi{\p } \limsupi{\nN } 
E ^{\muNl } 
[\int_0^T | \{ \bNrsp - \bNrs \} (X_t^{N,i},\XNidt ) | ^{\phat }
dt ] = 0 
,\\\notag &
\limi{\p }
E ^{\mul } 
[\int_0^T | \{ \brsp - \brs \} (X_t^{i},\Xidt ) | ^{\phat }
dt ] = 0 
,\end{align}
where $ \bbb _{r,s} $ is such that 
\begin{align}\label{:80p}&
 \bbb _{r,s} (x,\yy ) = \limi{\nN } \bbb _{r,s}^{\nN } (x,\yy ) 
\quad \text{ for each } (x,\yy ) \in \bigcup_{\p \in \N }
 (\sS \ts \sSS )_{r,\p } ^c 
.\end{align}
By definition 
$ \bigcup_{\p \in \N } (\sS \ts \sSS )_{r,\p } ^c = \{\Sr ^c \ts \sSS \} \cup 
\{ (x,\yy ) ; x \not= y_i \text{ for all }i\} $. 
$ \bbb _{r,s} (x,\yy ) = 0 $ for $ x \not\in \Sr $ by \eqref{:80g}. 
The limit in \eqref{:80p} exists because of \eqref{:80h}, \eqref{:80i}, and \eqref{:80k}. 
\smallskip 

\noindent 
\Ass{I4} 
There exists some $ \btail \in C (\sS ; \Rd )$, independent of $ r \in \N $ and $ \sss \in \sSS $, such that 
\begin{align}\notag &
\limi{s} \limsupi{\nN } \| \bNrstail - \btail 
\|_{ L^{\phat }(\Sr \ts \sSS , \muNone ) } 
 =0 
.\end{align}
Furthermore, for each $ r , i , T \in \N $, 
\begin{align}\notag &
\limi{s} \limsupi{\nN } E ^{\muNl } 
[\int_0^T | ( \bNrstail - \btail ) (X_t^{N,i},\XNidt ) | ^{\phat } dt ] = 0 
.\end{align}

We remark that $ \btail $ is independent of $ r $ because of the consistency of \eqref{:80i}. 
By assumption, $ \btail = \btail (x)$ is a function of $ x $. 
From \eqref{:80f} and \eqref{:80j}, we have that 
\begin{align}\label{:80s}&
\bN = \bNrsp + \btail + \{ {\bNrs - \bNrsp } \} + \{ {\bbb _{r,s}^{N,\tail}}  -\btail \} 
.\end{align}
Then, \Ass{I3} and \Ass{I4} imply that the last two terms 
$ \{ {\bNrs - \bNrsp } \} $ and $ \{ {\bbb _{r,s}^{N,\tail}}  -\btail \} $ in \eqref{:80s} are asymptotically negligible. 
Under these assumptions, there exists some $ \bbb $ such that, for each $ r \in \N $ (see \cite[Lemma 3.1]{k-o.fpa}),
\begin{align}
&\notag 
\limi{s} \| \brs - \bbb \|_{ L^{\phat }( \Sr \ts \sSS ,\muNone ) } = 0 
.\end{align}
\Ass{I5} For each $ i , r , T \in \N $, 
\begin{align} \notag &
\limi{s }
E ^{\mul } 
[\int_0^T | ( \brs - \bbb ) (X_t^{i},\Xidt ) | ^{\phat }
dt ] = 0 
.\end{align}

A sequence $ \{ \mathbf{X}^{\nN } \} $ of $ C([0,T];\sS ^{\nN })$-valued 
random variables is said to be tight if, for any subsequence, we can choose a subsequence denoted by the same symbol such that 
$ \{ \mathbf{X}^{\nN , m } \}_{\nN \in \N  , m \le \nN  } $, 
where $\mathbf{X}^{\nN ,m} =  (X^{\nN ,1},\ldots,X^{\nN ,m}) $, 
 is convergent in law in $ C([0,T]; \sS ^m )$ for each $ m \in \N $. 
We quote: 
\begin{lemma}[{\cite[{Theorem 2.1}]{k-o.fpa}}]\label{l:80}
Assume that \Ass{H1}--\Ass{H4} and \Ass{I1}--\Ass{I5} hold.
Then, for each $ T \in \N $, $ \{ \mathbf{X}^{\nN } \} _{\nN \in \N } $ 
is tight in $ C([0,T];\sS ^{\N } )$ and any limit point 
$ \mathbf{X} = (X^i)_{i\in\N } $ of $ \{ \mathbf{X}^{\nN } \} _{\nN \in \N } $ 
is a weak solution of the ISDE 
\begin{align}\label{:80z}&
dX_t^i = \sigma (X_t^i,\XX _t^{\idia }) dB_t^i + 
\{\bbb (X_t^i,\XX _t^{\idia })+ 
\btail (X_t^i)\} dt 
.\end{align}
\end{lemma}
Clearly, we deduce from \lref{l:80} that $ \{ \mathbf{X}^{\nN } \} _{\nN \in \N } $ 
is tight in $ C([0,\infty );\sS ^{\N } )$ and any limit point 
$ \mathbf{X} = (X^i)_{i\in\N } $ 
is a weak solution of the ISDE \eqref{:80z} on $ [0,\infty )$.

\subsection{Construction of weak solutions: symmetric case} \label{s:81}

The second author constructed weak solutions of ISDEs under some mild assumptions through Dirichlet form theory \cite{o.isde}. 
In this subsection, we recall the results. We begin by constructing $ \mu $-reversible diffusions \cite{o.dfa,o.rm}.

We denote by $ \Lambda_r $ the Poisson random point field whose intensity is the Lebesgue measure on $ \Sr $. 
We also set $ \Lambda_r^m = \Lambda_r (\cdot \cap \SSrm ) $, 
where $ \SSrm = \{ \sss \in \sSS \, ;\, \sss (\Sr ) = m \} $. 
Let $ \map{\pir , \pirc }{\sSS }{\sSS }$ by $ \pir (\sss ) = \sss (\cdot \cap \Sr )$ and 
$ \pirc (\sss ) = \sss (\cdot \cap \Sr ^c)$ as before. 

Let $ \map{\Phi }{\sS }{\R \cup \{ \infty \} }$ and $ \map{\Psi }{\sS \ts \sS }{\R \cup \{ \infty \} }$ 
be measurable functions. We set 
\begin{align}& \notag 
\mathcal{H}_r (\xx ) = \sum_{x_i\in \Sr } \Phi (x_i) + \sum_{ i < j ,\, x_i, x_j \in \Sr } \Psi (x_i ,x_j )
\quad \text{ for } \xx = \sum_i \delta_{x_i} 
.\end{align}
We call $ \mathcal{H}_r $ a Hamiltonian on $ \Sr $ 
with free potential $ \Phi $ and interaction potential $ \Psi $. 

\begin{definition}\label{d:81}
 A random point field $ \mu $ is called 
 a $ ( \Phi , \Psi )$-quasi Gibbs measure with inverse temperature $ \beta > 0 $ 
if its regular conditional probabilities 
 $$ 
 \mu _{r,\xi }^{m} = 
 \mu (\, \pir ( \xx )\in \cdot \, | \, \pirc (\xx ) = 
 \pirc (\xi ),\, \xx (\Sr ) = m )
 $$
satisfy, for all $r,m\in \mathbb{N}$ and $ \mu $-a.s.\! $ \xi $, 
\begin{align}\notag &
\cref{;10Q}^{-1} e^{- \beta \mathcal{H}_r(\xx ) } \Lambda_r^{m} (d\xx ) \le 
 \mu _{r,\xi }^{m} (d\xx ) \le 
\cref{;10Q} e^{- \beta \mathcal{H}_r(\xx ) } \Lambda_r^{m} (d\xx ) 
.\end{align}
 Here, $ \Ct \label{;10Q} = \cref{;10Q} (r,m,\xi )$ 
 is a positive constant depending only on $ r ,\, m ,\, \pirc (\xi ) $. 
 For two measures $ \mu , \nu $ on a $ \sigma $-field $ \mathcal{F} $, we write 
$ \mu \le \nu $ if $ \mu (A) \le \nu (A) $ for all $ A \in \mathcal{F} $. 
\end{definition}
%

We make the following assumptions. 

\noindent 
\Ass{A1} $ \mu $ is a $ ( \Phi , \Psi )$-quasi Gibbs measure such that 
 there exist upper semi-continuous functions $ (\hat{\Phi }, \hat{\Psi })$ 
 and positive constants $ \Ct \label{;A21}$ and $ \Ct \label{;A22} $ satisfying 
\begin{align} & \notag 
\cref{;A21}^{-1} \hat{\Phi } (x)\le \Phi (x) \le \cref{;A21} \hat{\Phi }(x) ,\quad 
\cref{;A22}^{-1} \hat{\Psi } (x,y)\le \Psi (x,y) \le \cref{;A22} \hat{\Psi }(x,y) 
.\end{align}
\Ass{A2} \ For each $ r \in \N $, $ \mu $ satisfies $ \sum_{m=1}^{\infty} m \mu (\SSrm ) < \infty 
$. 

\smallskip 

Let $ (\Eamu ,\di ^{\amu } )$ be the bilinear form on $ \Lmu $ with domain $ \di ^{\amu }$ defined by 
\begin{align} 
\notag &
\Eamu (f,g) = \int_{\sSS } \DDDa [f,g] \, \mu (d\sss )
.\end{align}
Here $ \DDDa $ is as in \eqref{:52a} and $ \di ^{\amu } = \{ f \in \di \cap \Lmu \, ;\, \Eamu (f,f) < \infty \} $, 
where $ \di $ is the set of all bounded, local smooth functions on $ \sSS $. 

A family of probability measures $ \{ \PPs \}_{\sss \in \sSS } $ 
on $ (\WS ,\mathcal{B}(\WS )) $ is called a diffusion if the canonical process 
$ \XX =\{ \XX _t \} $ under $ \PPs $ is a continuous process 
with the strong Markov property starting at $ \sss $. 
Here, $ \XX _t(\ww ) = \wwt $ for $ \ww =\{ \wwt \} \in \WS $ by definition. 
$ \XX $ is adapted to $ \Ft $, where $ \mathcal{F}_t = \cap_\nu \mathcal{F}_t^{\nu} $ and the intersection is taken over all Borel probability measures $ \nu $; 
$ \mathcal{F}_t^{\nu}$ is a $ \sigma $-field generated by 
all $ \PP _{\nu }$-null sets of $ \overline{\cup_{s\ge 0 }\mathcal{F}_s }^{\PP _{\nu }}$ and 
$ \mathcal{F}_t^+ = \cap_{\epsilon > 0 } \mathcal{B}_{t+\epsilon } (\sSS ) $, 
where $ \Bt (\sSS ) = \sigma [\ww _s; 0 \le s \le t ]$. 
Furthermore, $ \{ \PPs \}_{\sss \in \sSS } $ is said to be $ \nu $-stationary if 
$ \nu $ is an invariant probability measure. We say $ \{ \PPs \}_{\sss \in \sSS } $ is $ \nu $-reversible if $ \{ \PPs \}_{\sss \in \sSS } $ is $ \nu $-symmetric and -stationary.

\begin{lemma}[\cite{o.dfa,o.rm,o-t.tail}] \label{l:81} 
Assume that \Ass{A1} and \Ass{A2} hold. Then, $ (\Eamu ,\di ^{\amu } )$ is 
closable on $ \Lmu $, and its closure $ (\Eamu ,\Damu )$ is a quasi-regular Dirichlet form on $ \Lmu $. 
 Moreover, the associated $ \mu $-reversible diffusion $ (\XX ,\{\PPs \}_{\sss \in\hH }) $ exists. 
\end{lemma}

We refer to \cite{mr} for the definition of quasi-regular Dirichlet forms and related notions. 
We also refer to \cite{fot.2} for details of Dirichlet form theory. 

Let $ \{\PPs \}_{\sss \in\hH }$ be as in \lref{l:81}. 
Note that $ \mu (\hH ) = 1 $ and set $ \Pm = \int_{\hH } \PPs \, \mu (d\sss ) $. 
Let $\lpath $ be as in \eqref{:22i}. Let $ \SSsde $ and $ \SSSsde $ be as in \eqref{:23y}. 
Then we assume the following: 

\smallskip 
\noindent 
\Ass{A3} \ $ \Pm (\WSsiNE ) = 1 $ and $ \Pm \circ \lpath ^{-1}(\WT (\SSSsde )) = 1 $. 


\begin{definition}[{\cite{o.isde}}]\label{d:82}
An $ \Rd $-valued function $ \dmu $ is called the logarithmic derivative of $\mu $ if 
$ \dmu \in L_{\mathrm{loc}}^1 (\sS \ts \sSS , \muone ) $ and, for all 
$\varphi \in C_{0}^{\infty}(\sS )\otimes \di $, 
\begin{align}&\notag 
 \int _{\sS \times \sSS } 
 \dmu (x,\sss )\varphi (x,\sss ) 
 \muone (dx d\sss ) = 
 - \int _{\sS \times \sSS } 
 \nabla_x \varphi (x,\sss ) \muone (dx d\sss ) 
.\end{align}
Here we write $ f \in L_{\mathrm{loc}}^p(\sS \ts \sSS , \muone ) $ if $ f \in L^p(\Sr \ts \sSS , \muone )$ for all $ r \in\N $, and we set 
$ \nablax a (x,\sss ) = (\PD{\akl }{x_{\lL }}(x,\sss ))\lD $, where $ x=(x_1,\ldots,x_d)$. 
\end{definition}
%

\smallskip 
\noindent \Ass{A4} 
$ \mu $ has a logarithmic derivative $ \dmu $, and the coefficients $ (\sigma , b )$ satisfy 
\begin{align}\notag &
\sigma ^{t}\sigma = a ,\quad b = \half \nablax a + \half \dmu 
.\end{align}

\begin{lemma}[{\cite[Theorem 26]{o.isde}}] \label{l:82} 
Assume that \Ass{A1}--\Ass{A4} hold. 
Then, there exists some $ \{ \mathcal{F}_t \} $-Brownian motion $ \mathbf{B}$ such that 
$ ( \lpath (\XX ),\mathbf{B}) $ is a weak solution of \eqref{:23a}. 
\end{lemma}

In \lref{l:82}, $ \OFPFs $ is the filtered space introduced before \lref{l:81}. 
We can take $ \hH $ in \eqref{:23c} uniquely up to capacity zero and 
$ \SSSsde $ in \eqref{:23b} as $ \ulab ^{-1} (\hH )$. 
ISDE \eqref{:23a}--\eqref{:23c} has a weak solution 
$ ( \lpath (\XX ),\mathbf{B}) $ defined on $ \OFPFs $ for each $ \sss \in \hH $ and 
$ \PmNI{\hH } $ and, in particular, $ \mu (\hH ) = 1 $.

\section{Appendix II: Lyons--Zheng decomposition for weak solutions of ISDEs}\label{s:9}

Let $ \XB $ be a weak solution of \eqref{:23a}--\eqref{:23b}. 
We shall derive the Lyons--Zheng type decomposition of additive functionals of $ \mathbf{X} $.

Let $ m \in \{ 0 \}\cup \N $ and $ F \in C^2 (\sS ^m\ts \Si )$, where 
$ \Si $ is given by \eqref{:22g} and 
$ C^2 (\sS ^m\ts \Si )$ is the set of functions on $ \sS ^m\ts \Si $ of $ C^2$-class. 
Here, we say that $ F $ is of $ C^2$-class if $ \check{F} \in C^2(\SN )$ 
in the sense that 
$ \check{F}(s_1,\ldots,s_n,s_{n+1},\ldots )$ is $ C^2 $ in $ (s_1,\ldots,s_n) $ 
for fixed $ (s_{n+1},\ldots )$ for all $ n \in \N $. 
The function $ \check{F}$ is such that, for $ \mathbf{x}=(x_1,\ldots,x_m)$ and 
$ \sss =\sum_{i=m+1}^{\infty} \delta_{s_i}$, 
\begin{align}&\notag 
F (\mathbf{x},\sss ) = \check{F} (x_1,\ldots,x_m,s_{m+1},s_{m+2},s_{m+3}, \ldots )
\end{align}
and for any permutation $ \mathfrak{p}$ on $ \N \backslash \{ 1,\ldots,m \} $, 
\begin{align}&\notag 
\check{F} (x_1,\ldots,x_m ,s_{m+1},s_{m+2},s_{m+3},\ldots )
=\check{F} (x_1,\ldots,x_m ,s_{\mathfrak{p} (m+1)},s_{\mathfrak{p} (m+2)}, s_{\mathfrak{p}(m+3)}, \ldots )
.\end{align}

Let $ \wwwm $ be as in \eqref{:22q}. Let 
$ \map{r_T}{C([0,T]; \sS ^m\ts \Si )}{C([0,T]; \sS ^m\ts \Si )}$ be such that 
\begin{align}\label{:92q}&
 r_T(\wwwm ) (t) = \wwwm ({T-t})
.\end{align}
We regard $ \QQlaM $ as a measure on $ C([0,T]; \sS ^m\ts \Si )$, 
where $ \QQlaM $ is given by \eqref{:34c}. 
Indeed, by \eqref{:23b} each tagged particle $ X^i $ of $ \mathbf{X}$ 
does not explode under $ \QQla $. 
Hence, each tagged path of $ \wwwm $ does not explode for $ \QQlaM $-a.e.\,$ \wwwm $. 

\begin{lemma} \label{l:91}
$ \QQlaM = \QQlaM \circ r_T^{-1} $. 
\end{lemma}
\begin{proof}  
By \eqref{:34c}, $ \QQlaM = \int _{\sS ^m \ts \sSS } \QQxsM d\la ^{[m]} $. 
We deduce from \Ass{S$_{\la }$} that $ \{ \QQxsM \} $ is 
a $ \la ^{[m]}$-symmetric Markov process and $ \la ^{[m]} $ is an invariant measure of $ \{ \QQxsM \} $. 
From these we conclude \lref{l:91}. 
\end{proof} 

Let $ \XM $ be the $ m $-labeled process of $ \mathbf{X}$, that is, 
 $ \XM =(X^1,\ldots,X^m,\sum_{i=m+1}^{\infty} \delta_{X^i})$. 
Recall that 
$ \sigma [\mathbf{B}_s^m ; s\le t] \subset \sigma [\mathbf{X} _s^{[m]} ; s\le t ]$ 
 for all $ t $ under $ \QQla $ by \Ass{BX}. 
Hence $ \mathbf{B}^m $ is a function of $ \mathbf{X}^{[m]}$. 
 Then, there exists a function $\hat{\mathbf{B}}^m $ defined on the path space $ \WT ( \Sm \ts \sSS )$ such that 
$ \mathbf{B}^m = \hat{\mathbf{B}}^m (\XM )$ under $ \QQla $. Clearly, 
$\hat{\mathbf{B}}^m (\wwwm ) $ under $ \QQxsM $ for $ \la ^{[m]}$-a.e.\,$ (\mathbf{x},\sss )$
 is a $ dm $-dimensional Brownian motion. 
Here we recall $ \QQlaM $ is not necessary a probability measure for $ m \in \N $. 
Below, \lq\lq under $ \QQlaM $'' means 
\lq\lq under $ \QQxsM $ for $ \la ^{[m]}$-a.e.\,$ (\mathbf{x},\sss )$''. 

Let $ \www ^m = (w^1,\ldots,w^m)$, where $ \wwwm =(\www ^m , \ww ^{m*}) $. 
Then $ \www ^m $ under $ \QQlaM $ is a weak solution of \eqref{:24f} with Brownian motion
 $\hat{\mathbf{B}}^m =(\hat{B}^{m,i})_{i=1}^m $ and $ \XX = \ww $, where 
$ \ww = \sum_{i=1}^{\infty} \delta_{w^i}$ as before. 
The coefficients of \eqref{:24f} depends only on $ \XX ^{m*} = \sum_{i=m+1}^{\infty} \delta_{X^i } $, 
so does $ \ww ^{m*} $ in the present case. 
By \eqref{:24z}, we can rewrite \eqref{:24f} as 
\begin{align}\label{:92s}&
d w _t^{i} = 
\sigma (w_t^{i}, \ww _t^{\idia } ) d \hat{B} _t^{m ,i} + 
 \bbb (w_t^{i}, \ww _t^{\idia } ) dt 
\quad \text{ for $ i=1,\ldots,m $}
.\end{align}
Here, for $ \www = (w^i)_{i=1}^{\infty}$, we set 
$ \ww ^{\idia } = \{ \ww _t^{\idia } \}_t $ by 
$ \ww _t^{\idia } = \sum_{j\not=i,\, j=1}^{\infty} \delta_{w_t^j} $. 

Let $ \Sm _{\neq} = \{ \mathbf{x}=(x_i)_{i=1}^m \, ;\, x_i\neq x_j \text{ for all } i\neq j \} $ 
and $ F \in C^2(\Sm _{\neq} \ts \sSS ) $. 
Below, we write $ \wwwm (t) = \wwwm _t $. 
Applying It$ \hat{\mathrm{o}}$'s formula to $ F $ and \eqref{:92s} informally, we see 
under $ \QQlaM $ 
\begin{align}\label{:92c}&
F (\wwwm _t) - F (\wwwm _0 ) = 
\sum_{i=1}^{\infty} 
 \int_0^t 
\big( 
\partial_i \check{F} (\www _u) , 
\sigma (w_u^i, \ww _u^{\idia }) 
d\hat{B}_u^{m,i} 
\big)_{\Rd } + 
\\ \notag &
 \int_0^t \sum_{i=1}^{\infty}
\big( 
 b (w_u^i, \ww _u^{\idia }) 
, 
 \partial_i \check{F} (\www _u) 
\big)_{\Rd }
du 
+
 \int_0^t \sum_{i=1}^{\infty} \sumklD 
 \half \akl (w_u^i, \ww _u^{\idia }) 
\partial_{i,\kK } \partial_{i,\lL } \check{F} (\www _u) 
du 
.\end{align}
The equality \eqref{:92c} can be justified if $ F $ is a local smooth function, and each term is integrable. 
We shall assume $ F (\wwwm _t)$ is a continuous semi-martingale satisfying \eqref{:92c}. 
\begin{lemma}	\label{l:92}
Consider the same assumptions as for \tref{l:34}. Let $ m \in \{ 0 \}\cup \N $ and 
$ F \in C^2(\Sm _{\neq} \ts \sSS ) $. 
Assume that for $ \QQlaM $-a.e.\,$ \wwwm $ 
\begin{align}\label{:92w}&
 \int_0^t \sum_{i=1}^{\infty} \sumklD 
 \akl (w_u^i,\ww _u^{\idia }) 
\partial_{i,\kK } \check{F} (\mathbf{\www }_u) 
\partial_{i,\lL } \check{F} (\mathbf{\www }_u) 
du < \infty 
\quad \text{ for all }t
\end{align}
and that $ F (\wwwm _t)$ is a continuous semi-martingale under $ \QQlaM $ satisfying \eqref{:92c}. 
Then, under $ \QQlaM $, we obtain for $ 0 \le t \le T $ 
\begin{align}\label{:92x}&
F (\wwwm _{t}) - F (\wwwm _0 ) = \half 
\Big\{ M_t (\wwwm ) + \big( M_{T-t} (r_T(\wwwm )) -M_T (r_T(\wwwm )) \big) \Big\} 
.\end{align}
Here, $ M $ is a continuous local martingale under $ \QQlaM $ such that 
\begin{align}\label{:92y}&
 M_t = \sum_{i=1}^{\infty} \int_0^t 
\big( 
\partial_i \check{F} (\www _u) , 
\sigma (w_u^i, \ww _u^{\idia }) d\hat{B}_u^{m,i} 
\big)_{\Rd }
.\end{align}
The quadratic variation of $ M $ is given by 
\begin{align}\label{:92z}&
\langle M \rangle_t (\wwwm ) = 
 \int_0^t \sum_{i=1}^{\infty} \sumklD 
 \akl (w_u^i,\ww _u^{\idia }) 
\partial_{i,\kK } \check{F} (\mathbf{\www }_u) 
\partial_{i,\lL } \check{F} (\mathbf{\www }_u) 
du 
.\end{align}
Furthermore, $ \{M_{T-t} (r_T(\wwwm )) -M_T (r_T(\wwwm )) \}$ 
is a continuous local martingale under $ \QQlaM $ with respect to the inverse filtering. 
\end{lemma}
\begin{proof}  
We modify the argument of \cite[Theorem 5.7.1]{fot.2} according to the current situation. 
Note that the weak solution $ \XB $ is not associated with any quasi-regular Dirichlet forms; 
there exists no $ L^2 $-semi-group associated with the labeled process $\mathbf{X} $. 
We can still use the $ L^2$-semi-group associated with the $ m $-labeled process $ \XM $ 
 (equivalently, $ \wwwm $ under $ \QQlaM $) given by \eqref{:34d} for any $ m \in \N $.

By  \Ass{BX} we see $ \hat{B} ^{m ,i} $ in \eqref{:92s} is a function in $ \www ^{[m]}$ and 
 we can write \eqref{:92s} as 
\begin{align}\label{:92D}&
d w _t^{i} = 
\sigma (w_t^{i}, \ww _t^{\idia } ) d \hat{B} _t^{m ,i} (\www ^{[m]}) + 
 \bbb (w_t^{i}, \ww _t^{\idia } ) dt 
\quad \text{ for $ i=1,\ldots,m $}
.\end{align}
For $ \mathbf{x} = (x_i)_{i=1}^{\infty}$, we set 
$ \mathbf{x}^{[m]} = (x_1,\ldots,x_m, \sum_{i=m+1}^{\infty} \delta_{x_i}) $ and 
$ \mathfrak{x}^{\idia } = \sum_{j\not=i,\, j=1}^{\infty} \delta_{x_j} $. 
Let 
\begin{align}\label{:92d}
G (\mathbf{x}^{[m]}) = &
\sum_{i=1}^{\infty} \big(  b (x_i,\xx ^{\idia }) 
,
\partial_i \check{F} (\mathbf{x})
\big)_{\Rd }
+
\sum_{i=1}^{\infty} \sumklD 
 \half \akl (x_i,\xx ^{\idia } ) 
\partial_{i,\kK }\partial_{i,\lL } \check{F} (\mathbf{x}) 
.\end{align}
Then, from \eqref{:92c}, \eqref{:92w}, \eqref{:92y}, \eqref{:92D}, and \eqref{:92d}, we have that under $ \QQlaM $ for $ 0 \le t \le T $ 
\begin{align}\label{:92f}&
F (\wwwm _t) - F (\wwwm _0 ) = M_t (\wwwm ) + \int_0^t G (\wwwm _u ) du 
.\end{align}

By \lref{l:91}, $ \QQlaM = \QQlaM \circ r_T ^{-1}$. 
Hence, $ M_t (r_T(\wwwm ))$ is well-defined for $ \QQlaM $-a.e.\,$ \wwwm $. 
We see then from \eqref{:92f} the following. 
\begin{align} \label{:92g}&
F ( r_T( \wwwm )_t ) - F ( r_T( \wwwm )_0 ) = M_t (r_T(\wwwm )) + \int_0^t G ( r_T( \wwwm )_u ) du 
 .\end{align}
By the definition of $ r_T $, we can rewrite \eqref{:92g} as 
\begin{align} \label{:92gg}
F (\wwwm _{T-t}) - F (\wwwm _T ) 
=&
 M_t (r_T(\wwwm )) + \int_0^T G ( \wwwm _u ) du - \int_0^{T-t} G ( \wwwm _u ) du 
.\end{align}
Hence, from \eqref{:92gg}, we obviously have 
\begin{align}\label{:92h}&
 M_t (r_T(\wwwm )) = F (\wwwm _{T-t}) - F (\wwwm _T ) 
- \int_0^T G ( \wwwm _u ) du + \int_0^{T-t} G ( \wwwm _u ) du 
.\end{align}
Take $ t $ to be $ T-t$ and $ T $ in \eqref{:92h}. Then we have 
\begin{align}\label{:92i}
M_{T-t}(r_T(\wwwm )) &= F (\wwwm _{t}) - F (\wwwm _T ) 
- \int_0^T G ( \wwwm _u ) du + \int_0^{t} G ( \wwwm _u ) du 
,\\\label{:92j}
M_T (r_T(\wwwm )) &= F (\wwwm _{0}) - F (\wwwm _T ) 
- \int_0^T G ( \wwwm _u ) du 
.\end{align}
Subtract both sides of \eqref{:92j} from those of \eqref{:92i}. 
Then using \eqref{:92f} we obtain 
\begin{align}\notag 
M_{T-t} (r_T(\wwwm )) -M_T (r_T(\wwwm )) = &
F (\wwwm _{t}) - F (\wwwm _0 ) + \int_0^t G (\wwwm _u ) du 
\\ \notag =&
2 \{F (\wwwm _{t}) - F (\wwwm _0 )\} - M_t (\wwwm ) 
.\end{align}
Hence, we have under $ \QQlaM $ for $ 0 \le t \le T $ 
\begin{align}\notag 
&
F (\wwwm _{t}) - F (\wwwm _0 ) = \half 
\Big\{ M_t (\wwwm ) + \big( M_{T-t} (r_T(\wwwm )) -M_T (r_T(\wwwm )) \big) \Big\} 
.\end{align}
This completes the proof of \eqref{:92x}. 
Equation \eqref{:92z} follows immediately from \eqref{:92y}. 
The last claim follows from \lref{l:91} and the definition of $ r_T $. 
\end{proof}

\section{Acknowledgments. } \label{s:X}

The authors thank the anonymous referee for his/her careful reading, constructive comments, and in particular, a suggestion on Osgood's condition. 
%
%
%
H.O. is supported in part by JSPS KAKENHI 
Grant Numbers JP20K20885, JP18H03672, JP16H06338. 
H.T. is supported in part by JSPS KAKENHI Grant Number JP19H01793.



\end{document}